\documentclass[reqno,10pt]{amsart}
\usepackage{amsmath}
\usepackage{amssymb}
\usepackage{amsfonts}
\usepackage{graphicx}
\usepackage[abs]{overpic}
\usepackage{caption}
\usepackage{wrapfig}
\usepackage{float}
\usepackage{graphicx}
\usepackage{mathrsfs}
\usepackage{accents}
\usepackage{amsthm}
\usepackage{hyperref}

\usepackage{tikz-cd}
\usepackage[colorinlistoftodos]{todonotes}
 
\usepackage{nccmath}

\newtheorem{theorem}{Theorem}[section]
\newtheorem{lemma}[theorem]{Lemma}
\newtheorem{proposition}[theorem]{Proposition}

\newtheorem{definition}[theorem]{Definition}

\newtheorem{rem}[theorem]{Remark}

\newcommand{\eop}{\nopagebreak\hspace*{\fill}$\Box$\smallskip}

\newcommand{\N}{\Bbb N}

\newcommand{\R}{\Bbb R}
\newcommand{\C}{\Bbb C}

\newcommand{\wpr}{{w^\prime}}
\newcommand{\wprpr}{w^{\prime\prime}}
\newcommand{\twpr}{{\tilde w^\prime}}
\newcommand{\twprpr}{{\tilde w^{\prime\prime}}}

\DeclareMathOperator*{\argmin}{arg\,min}
\newcommand{\diff}{{\rm diff}}

\def\grad{\nabla}

\def\eps{\varepsilon}

\def\XXint#1#2#3{{\setbox0=\hbox{$#1{#2#3}{\int}$}
     \vcenter{\hbox{$#2#3$}}\kern-.5\wd0}}

\usepackage{color}

\newcommand{\sym}{{\rm sym}}

\newcommand{\ZZZ}{\color{black}}
\newcommand{\MMM}{\color{black}} 
\newcommand{\NNN}{\color{black}} 
\newcommand{\BBB}{\color{black}} 
\newcommand{\EEE}{\color{black}} 

\newcommand{\RRR}{\color{black}}
 \newcommand{\PPP}{\color{black}} 
\newcommand{\second}{\color{black}}

\numberwithin{equation}{section}

\title{Derivation of a one-dimensional von K\'{a}rm\'{a}n theory for  viscoelastic ribbons}

\author{Manuel Friedrich}
\author{Lennart Machill}

\subjclass[2010]{74D10, 35A15, 35Q74, 49J45}
 \keywords{Viscoelasticity, metric gradient flows,  dimension reduction, $\Gamma$-convergence, dissipative distance, curves of maximal slope, minimizing movements.}

\address[Manuel Friedrich]{Applied Mathematics,  
Universit\"{a}t M\"{u}nster, Einsteinstr. 62, D-48149 M\"{u}nster, Germany}
\email{manuel.friedrich@uni-muenster.de}

\address[Lennart Machill]{Applied Mathematics,  
Universit\"{a}t M\"{u}nster, Einsteinstr. 62, D-48149 M\"{u}nster, Germany}
\email{lennart.machill@uni-muenster.de}

\begin{document}

\maketitle

\begin{abstract}
We consider a two-dimensional model of viscoelastic von K\'arm\'an plates in the Kelvin's-Voigt's rheology derived from a    three-dimensional model at a finite-strain setting in \cite{MFMKDimension}. As the width of the plate goes to zero, we perform a dimension-reduction  from 2D to 1D and identify an effective one-dimensional model for a viscoelastic ribbon comprising stretching, \EEE bending, \EEE and twisting both in the elastic and the viscous stress. Our arguments rely on the abstract theory of gradient flows in metric spaces by Sandier and Serfaty \cite{S1}  and complement the  $\Gamma$-convergence analysis of elastic von K\'{a}rm\'{a}n \MMM ribbons \EEE in \cite{Freddi2018}.  Besides convergence of the gradient flows, we also show convergence of associated time-discrete approximations, and we provide a corresponding commutativity result. 

\end{abstract}

\section{Introduction}

The derivation of effective theories for thin structures such as  plates, rods, or ribbons is a classical problem in continuum mechanics. Despite the long history of the subject with contributions already by Euler, Kirchhoff, and von K\'arm\'an (see \cite{Antmann:04, ciarlet2} for surveys), rigorous results  relating lower-dimensional theories  to three-dimensional elasticity have only been obtained comparably  recently. They were triggered by the use of variational methods, particularly by  $\Gamma$-convergence \cite{DalMaso:93} together with  quantitative  rigidity estimates \cite{FrieseckeJamesMueller:02}. In this present work, we are interested in effective descriptions for viscoelastic ribbons, i.e., bodies with three different length scales: a length $l$ which is much larger than the width $\eps$ which, in turn, is much larger than the thickness $h$. This difference in the characteristic dimensions  allows to model the material effectively as a one-dimensional continuum \cite{fried}.

In elasticity theory, the study of such models dates back to works by {\sc Sadowsky} \cite{hinz} and {\sc Wunderlich} \cite{Todres} who proposed and formally justified one-dimensional energies from a two-dimensional Kirchhoff plate model, corresponding  to $h=0$ and to the limiting passage $\eps \to 0$. Recently, {\sc Freddi, Hornung, Mora,  and Paroni} \cite{Mora3} gave  a rigorous justification of a corrected model and addressed also related effective descriptions derived from von K\'arm\'an plate models \cite{Freddi2018}. More generally, a hierarchy of  one-dimensional models has been derived from three-dimensional nonlinear elasticity by considering the simultaneous limit $h \ll \eps \to 0$ for appropriate rates $h/\eps$ \cite{Freddi2012, Freddi2013}. On the contrary, the assumption $h \sim \eps$ leads to different effective rod models identified in \cite{ABP, Mora, Mora2}. \EEE

 The present work is devoted to a similar scenario of deriving one-dimensional theories for \emph{viscoelastic} von K\'{a}rm\'{a}n ribbons. Let us start by considering a nonlinear  three-dimensional model  of a viscoelastic material with reference configuration $\Omega_{\eps,h} = (-\frac{l}{2}, \frac{l}{2}) \times (-\frac{\eps}{2}, \frac{\eps}{2}) \times (-\frac{h}{2}, \frac{h}{2}) $ at finite strains  in Kelvin's-Voigt's \EEE rheology, i.e.,   a spring and a damper coupled in parallel. Neglecting inertia, the nonlinear system of equations takes the form 
\begin{align}\label{eq:viscoel}
- \text{div} \bigg( \partial_F W(\grad y) + \partial_{\dot{F}} R(\grad y, \partial_t \grad y) \bigg) = f \text{ in } [0,T] \times \Omega_{\eps,h},
\end{align}
where $[0,T]$ is the process time interval and  $y\colon [0,T] \times \Omega_{\eps,h} \to \R^3$ denotes a deformation with gradient $\grad y$. The tensors $\partial_F W$ and $\partial_{\dot{F}} R$ correspond to the \emph{elastic stress} and the \emph{viscous stress}, respectively, and are related to a  stored energy density  $W$ as well as a  (pseudo-\EEE)potential $R$  of dissipative forces. (Here, $F$ and  $\dot{F}$ are placeholders for $\nabla y$ and $\partial_t \nabla y$, respectively.) Eventually, $f$ describes an external force.  In contrast to the rapidly developed static theory of nonlinear elasticity, due to the physically relevant assumptions on frame indifference for both $W$ and $R$ (see \cite{Antmann, Ball:77}), existence of solutions remains a challenging problem and results are scarce. We refer, e.g., to \cite{Lewick}  for  local in-time existence or to \cite{demoulini} for the existence of  measure-valued solutions. To date, weak solutions in a general finite strain setting \cite{MFMK, MielkeRoubicek} can only be guaranteed by using the concept of  second-grade nonsimple materials  where   the elastic properties additionally depend on the second gradient of the deformation \cite{Toupin:62, Toupin:64}. At the same time, let us mention \EEE that a main justification of the model investigated in \cite{MFMK} lies in the observation that,  in the small strain limit,  the problem leads to the standard system of linear viscoelasticity without second gradient.

Recently in \cite{MFMKDimension}, \BBB starting from a version of \eqref{eq:viscoel} for nonsimple materials, \EEE  a dimension \EEE reduction \EEE  has been performed to derive a von K\'arm\'an-like viscoelastic plate model on a two-dimensional plate $S_\eps = (-\frac{l}{2},\frac{l}{2}) \times  (-\frac{\eps}{2},\frac{\eps}{2})$. \BBB This complements the static $\Gamma$-convergence analysis for elastic materials \cite{hierarchy}, which \EEE has justified \EEE the von K\'arm\'an plate equations proposed more than 100 years ago \cite{VonKarman}. \EEE In contrast to previous works on viscoelastic plates \cite{Bock1, Bock2, Bock3, Park}, where the starting point is already a plate model, \cite{MFMKDimension} constitutes a rigorous derivation by proving that solutions to \eqref{eq:viscoel} converge in a suitable sense to effective two-dimensional equations. More precisely, there are  in-plane
\EEE displacements \EEE  $u\colon \EEE [0,T] \times \EEE S_\eps \to \R^2$  and out-of-plane displacements $v\colon \EEE [0,T] \times \EEE S_\eps \to \R$ such that
\begin{align}\label{eq: equation-simp-intro2}
\begin{cases} 
g^{2D} =  \EEE  - \EEE {\rm div}\Big(\C_W^2\big( e(u)  + \frac{1}{2} \nabla v \otimes \nabla v  \big) +  \C_R^2 \big( e(\partial_t u) +   \nabla \partial_t v    \otimes  \nabla v \big) \Big)  , &\vspace{0.1cm} \\ 
f^{2D} = -{\rm div}\Big(\Big(\C_W^2\big( e(u)  + \frac{1}{2} \nabla v \otimes \nabla v  \big)  +  \C_R^2 \big( e(\partial_t u) +   \nabla \partial_t v   \otimes \nabla v \big) \Big) \nabla v \Big)  &\vspace{0.1cm}\\
 \quad \quad   \ \ \   + \tfrac{1}{12}  {\rm div} \, {\rm div}\Big( \C_W^2 \nabla^2 v + \C_R^2 \nabla^2 \partial_t v \Big)    
\end{cases}
\end{align}
in $[0,T] \EEE \times S_\eps$, where $e(u) :=\frac{1}{2}( \grad u+ \grad u^T)$ denotes the linear strain tensor, and $\C_W^2$ as well as $\C_R^2$  are tensors of elasticity and viscosity coefficients, respectively, derived suitably from $W$ and $R$. In addition to the vertical force $f^{2D}$, we also consider a horizontal force $g^{2D}$ that was not included in \cite{MFMKDimension} for  simplicity. \EEE
 The first equation features a \emph{membrane term} both in the elastic and the viscous stress, and the second equation contains also a \emph{bending term}. The problem is closely related to the   \emph{von K\'arm\'an functional}  given by (neglecting the forces) \EEE
\begin{align*}
{\phi}_{\eps}(u,v) := \frac{1}{2} \int_{S_\eps} Q_W^2\Big( e(u) + \frac{1}{2} \nabla v \otimes \nabla v \Big)\, {\rm d}x + \frac{1}{24}\int_{S_\eps}Q_W^2(\nabla^2 v)\, {\rm d}x,
\end{align*}
where  $Q_W^2(F) = F: \ZZZ \C_W^2 \EEE F$ for $F \in \R^{2\times 2}$. Indeed, \eqref{eq: equation-simp-intro2} can  proved to be a (metric) gradient flow for ${\phi}_{\eps}$ for a metric suitably related to $\C_W^2$, see again \cite{MFMKDimension}.  This approach was additionally exploited in \cite{MFMKJV} for numerical approximation of the system \eqref{eq: equation-simp-intro2} via a minimizing movement scheme. We also refer to \cite{AMM} for a related dynamical problem considering inertial terms but no viscosity.

In order  to describe the one-dimensional effective behavior of  thin viscoelastic ribbons, the goal of the present paper is to perform another dimension reduction by letting the width  $\eps$  in  \eqref{eq: equation-simp-intro2} go to zero. In a purely static setting, this problem has been addressed in \cite{Freddi2018} by identifying the $\Gamma$-limit of the sequence $\frac{1}{\eps}\phi_\eps$ in terms of the \BBB non-convex \EEE functional
\begin{align}\label{def:enegeryphi-intro}
\phi_0(\xi_1,\xi_2,w,\theta):= &\frac{1}{2} \int_{-l/2}^{l/2} \ZZZ Q_W^0 \EEE \Big(\xi_1' + \frac{\vert w^{\prime}\vert^2}{2}\Big) \,  {\rm d}x_1 + \frac{1}{24} \int_{-l/2}^{l/2} \big( \ZZZ Q_W^0 \EEE (\xi_2'') + Q_W^1(w^{\prime\prime}, \theta^\prime) \big) \, {\rm d}x_1 
\end{align}
comprising stretching, bending, and twisting, where \ZZZ $Q_W^0$ and $Q_W^1$ are quadratic forms suitably related to $Q_W^2$. \EEE More precisely, the in-plane displacement $u$ can be related to an axial displacement $\xi_1$ and an orthogonal in-plane displacement $\xi_2$. In contrast, the out-of-plane displacement $v$ generates an out-of-plane displacement $w$, and the derivative of $v$ in the direction orthogonal to the axis leads to a  twist function $\theta$.
 
In the framework of viscoelastic ribbons, we relate the nonlinear equations \eqref{eq: equation-simp-intro2} to the following equations for viscoelastic ribbons
\begin{align} \label{eq: equation-simp-intro1}
\begin{cases}
g_1^{1D} = -  \mfrac{\rm d}{{\rm d}x_1} \bigg(C_W^0\Big(\xi_1^\prime+\frac{\vert w^\prime\vert^2}{2}\Big) + C_R^0(\partial_t \xi_1^\prime+w^\prime \partial_t w^\prime )   \bigg) , &\vspace{0.1cm}\\
g_2^{1D} = \mfrac{1}{12}   \mfrac{{\rm d}^2}{{\rm d}x_1^2}  \Big(C_W^0 \xi_2^{\prime\prime} +C_R^0 \partial_t \xi_2^{\prime\prime}   \Big) ,  &\vspace{0.1cm}\\
f^{1D}= - \mfrac{{\rm d}}{{\rm d}x_1} \bigg(\Big(C_W^0\Big(\xi_1^{\prime} + \frac{\vert \wpr \vert^2}{2}\Big) + C_R^0(\partial_t \xi_1^{\prime} + w^\prime \partial_t w^\prime) \Big) w^\prime \bigg)  &\vspace{0.1cm}\\
\ \ \ \ \ \ \ \ \ \ \ +
\frac{1}{24} \mfrac{{\rm d}^2}{{\rm d}x_1^2}  \Big(\partial_1 Q_W^1(w^{\prime\prime},\theta^\prime)+\partial_1Q_R^1(\partial_t\wprpr,\partial_t \theta^\prime )\Big), &\vspace{0.1cm}\\
0 =  \mfrac{{\rm d}}{{\rm d}x_1} \Big(\partial_2 Q_W^1(w^{\prime\prime},\theta^\prime)+\partial_2 Q_R^1(\partial_t \wprpr,\partial_t \theta^\prime)\Big)
\end{cases}
\end{align}
in $[0,T] \EEE  \times (-\frac{l}{2},\frac{l}{2})$, \RRR where the \EEE constants $C_W^0>0$ and $C_R^0>0$, and the quadratic form $Q_R^1$ are again related to $Q_W^2$ and $Q_R^2$, respectively. \RRR Moreover, \EEE $f^{1D}$, $g^{1D}$ are forces derived from $f^{2D}$, $g^{2D}$. \RRR Note that the equations are again given in divergence form. \EEE  More precisely, we prove  existence \BBB of \EEE solutions to \eqref{eq: equation-simp-intro1} and make the dimension reduction rigorous, i.e., we show that solutions to  \eqref{eq: equation-simp-intro2}  converge to solutions of \eqref{eq: equation-simp-intro1} in a specific sense. The solutions have to be understood in a weak sense, see \eqref{eq:weak1d} for the exact definition.  The same property holds for time-discrete approximations, and we provide a corresponding commutativity result, see Figure \ref{diagram}. 

Heuristically, \eqref{eq: equation-simp-intro1} can be understood as the effective equation of a thin-walled beam with reference configuration $\Omega_{\eps,h}$ governed by \eqref{eq:viscoel}, when we first let $h \to 0$ and then $\eps \to 0$. In a forthcoming work, we will make this intuition rigorous by studying simultaneous limits $h \ll \eps \to 0$, \EEE complementing the purely elastic analysis in \cite{Freddi2013}. Let us mention that in \cite{Freddi2018} also other energy regimes have been  considered, leading to a ``linearized'' von K\'arm\'an or a  ``constrained'' von K\'arm\'an  energy. Whereas the former case is completely covered by our analysis, the latter is \ZZZ subtler \EEE due to the nonlinear constraint $\det(\nabla^2 v) = 0$ in the two-dimensional setting.

Our general strategy is to treat the systems \eqref{eq: equation-simp-intro2} and \eqref{eq: equation-simp-intro1} in the abstract setting of metric gradient flows \cite{AGS} for the energies $\phi_\eps$ and $\phi_0$, respectively, where the underlying metric is given by a \emph{dissipation distance} suitably related to  $\C_R^2$, \EEE  $\partial_1 Q_R^1$, $\partial_2 Q_R^1$, \EEE and $C_R^0$.  (We also refer to \cite{MOS} for a thorough explanation to this approach.)  We follow the approach of \emph{evolutionary $\Gamma$-convergence} devised in  \cite{Mielke, Ortner, S1,S2}.  In using this theory, the challenge lies in proving
that the additional conditions needed to ensure convergence of gradient flows are satisfied.

More specifically, to use the  abstract convergence result, lower semicontinuity of the energies, the metrics, and the local slopes is needed. The estimate for the energies and metrics essentially follows from \cite{Freddi2018}, \ZZZ see \EEE Theorems \ref{th: Gamma} and \ref{th: lscD} below.  The lower semicontinuity of the local slopes, however, is more technical and the core of our argument. We briefly present the main idea in the one-dimensional setting.  The \emph{local slope}  of  $\phi_0$ in a metric space with metric $\mathcal{D}_0$ is defined by
$$|\partial \phi_0|_{\mathcal{D}_0}(z): = \limsup_{\hat z \to z} \frac{(\phi_0(z) - \phi_0(\hat z))^+}{\mathcal{D}_0(z,\hat z)}.$$
(To simplify the notation here, we use a single variable $z$ in place of $(\xi_1,\xi_2,w,\theta)$, \ZZZ see \EEE \eqref{def:enegeryphi-intro}.) We use a finer representation of the local slope, based on generalized convexity properties, to show that the local slope coincides with the \emph{global slope}, see \cite[Definition 1.2.4]{AGS}, up to some lower order terms. Essentially, this shows 
$|\partial \phi_0|_{\mathcal{D}_0}(z) \approx \frac{(\phi_0(z) - \phi_0(\tilde{z}))^+}{\mathcal{D}_0(z,\tilde{z})}$ \EEE for some $\tilde{z}$. \EEE Consider a sequence  $(z_n)_{n \in \N}$   converging to a limit $z$.  The main step consists in constructing a \emph{mutual recovery sequence} $(\tilde{z}_n)_{n \in \N}$ such that $\phi_0(z_n) - \phi_0(\tilde{z}_n) \to \phi_0(z) - \phi_0(\tilde{z})$ and $\mathcal{D}_0(z_n, \tilde z_n) \to \mathcal{D}_0(z, \tilde z)$, \BBB see Lemma \ref{lem:mutualrecovery}. \EEE The strategy to prove the lower semicontinuity of local slopes along the sequence $\phi_\eps$ is similar, but the construction of mutual recovery sequences (see Lemma \ref{lem:mutual}) is \ZZZ subtler \EEE as suitable compatibility conditions between the elastic energy and the viscous dissipation are needed. We consider different conditions in that direction, ranging from materials with small Poisson ratio to vanishing dissipation potentials in the direction of the width $\eps$,  see Subsection~\ref{sec: compatibility} for details. \EEE  Let us emphasize that mutual recovery sequences are also crucial to perform the limiting passage on the time-discrete level, see Theorem~\ref{thm:gammaofscheme}.

The plan of the paper is as follows. In Section 2, we introduce the one- and two-dimensional models and state our main results.  The main goal is to prove the existence of solutions to the one-dimensional model, which is based on gradient flows in metric spaces  \cite{AGS}. In particular, Theorem~\ref{thm:curveofmaximalslope:energyidentity}(i) provides the existence of \EEE solutions \EEE to the one-dimensional gradient flow by relying on the theory of generalized minimizing movements. Moreover, Theorem~\ref{thm:curveofmaximalslope:energyidentity}(ii) identifies solutions of the gradient flow as weak solutions of the one-dimensional system of PDEs \EEE \eqref{eq: equation-simp-intro1}. \EEE Finally, Theorem~\ref{thm:relation2d1d} \EEE addresses \EEE the relation to the two-dimensional system. Besides the convergence of two-dimensional solutions to the one-dimensional solutions, we also get analogous results \EEE for the semidiscretized problems. In particular, \EEE convergences for vanishing time step and vanishing width of the \EEE plate commute, see Figure~\ref{diagram}. \EEE Section 3 is devoted to definitions concerning the theory of gradient flows in metric spaces. In particular, we recall the approximation scheme and also collect the necessary existence results for \emph{curves of maximal slopes} \cite{AGS, DGMT}. \EEE While Section 4 collects separately the main  properties of the two- and one-dimensional systems,  Section 5 addresses the relation between the two systems. Finally, proofs of the main results can be found in Section 6, and several technical proofs are postponed to the  Appendix \ref{sec:Appendix}.

\BBB \subsection*{Notation} \EEE
 In what follows, we use standard notation for Lebesgue spaces,  $L^p(\Omega)$, which are measurable maps on $\Omega\subset\R^d$, $d=1,2$, \EEE integrable with the $p$-th power (if $1\le p<+\infty$) or essentially bounded (if $p=+\infty$).    Sobolev spaces, i.e., $W^{k,p}(\Omega)$ denote the linear spaces of  maps  which, together with their weak derivatives up to the order $k\in\N$, belong to $L^p(\Omega)$.  Further,  $W^{k,p}_0(\Omega)$ contains maps from $W^{k,p}(\Omega)$ having zero boundary conditions (in the sense of traces).   Moreover, for a function $ v \in W^{k,p}(\Omega)$, \EEE the set $W^{k,p}_{v}(\Omega)$ contains maps from $W^{k,p}(\Omega)$ attaining $v$ at the boundary  \EEE  (in the sense of traces) up to the \EEE $(k-1)$-th \EEE order. To emphasize the target space $E$, we write $W^{k,p}(\Omega;E)$. If $E = \R$, we write $W^{k,p}(\Omega)$ as usual. \EEE  We refer to \cite{AdamsFournier:05} for more details on Sobolev spaces and their duals.  If the integration variable is clear from the context, we \EEE usually \EEE drop ${\rm d}x$ at the end of integrals. \EEE

\section{The model and main results}\label{sec:Model}

\subsection{The two-dimensional setting}\label{sec: two-d}

In this subsection we describe the two-dimensional von K\' arm\' an plate model. Fixing an interval $I = (-\frac{l}{2}, \frac{l}{2})$, the set $S_\eps:= I \times (-\frac{\eps}{2}, \frac{\eps}{2})$ represents a two-dimensional reference configuration of a two-dimensional plate, where $l>0$ denotes the length and $\eps>0$ the width. For $u \in W^{1,2}(S_\eps;\R^2)$ and $v \in W^{2,2}(S_\eps;\R)$ we define the \emph{von K\'arm\'an} functional by \EEE
\begin{align}\label{def:vK2D}
{\phi}_\eps(u,v) := \frac{1}{\eps}\int_{S_\eps} \frac{1}{2}Q_W^2\Big( e(u) + \frac{1}{2} \nabla v \otimes \nabla v \Big) + \frac{1}{\eps}\int_{S_\eps} \frac{1}{24}Q_W^2(\nabla^2 v) - \frac{1}{\eps} \int_{S_\eps} \big( f^{2D}_\eps v + g^{2D}_\eps \cdot \EEE u \big) \EEE,
\end{align}
where $e(u):= \frac{1}{2} (\nabla u + \nabla u^T)$ denotes the linear strain tensor and $\otimes$ the \EEE Euclidean \EEE tensor product. \EEE Moreover, \EEE $Q_W^2$ denotes a quadratic form on $\R^{2 \times 2}$, related to the \BBB tensor \EEE $\C_W^2$  in \eqref{eq: equation-simp-intro2}  via the mapping \BBB $A \mapsto Q_W^2(A) = A^T\C_W^2 A$. \EEE The quadratic form is derived from  a frame-indifferent energy density $W$. \EEE Therefore, it \ZZZ only depends \EEE on the symmetric part of a matrix  $A \in \R^{2\times 2}$, i.e., on $\frac{1}{2} (A+A^T)$, \EEE and \EEE is positive definite on $\R^{2 \times 2}_{\rm sym}$. \EEE  This functional describes the energy of \BBB a two-dimensional plate deformed in three-dimensional \EEE space, where $u$ and $v$ correspond to in-plane and out-of-plane displacements, respectively. The energy comprises both a \emph{membrane term} depending on $u$ and $v$, as well as a \emph{bending term},   only depending \EEE on $v$. \EEE Eventually, \EEE the functions $f^{2D}_\eps \in L^2(S_\eps)$ and \BBB $g^{2D}_\eps \in L^2(S_\eps;\R^2)$ \EEE correspond to a vertical force density and a horizontal force density, respectively.

We also introduce a \emph{global dissipation distance} $\mathcal{D}_\eps((u,v),(\tilde u, \tilde v))$ between two displacements \EEE $(u,v)$, $(\tilde u , \tilde v) \in W^{1,2}(S_\eps; \R^2) \times W^{2,2}(S_\eps)$  whose square is given by
\begin{align}\label{eq: metriceps}
\mathcal{D}_\eps((u,v),(\tilde u, \tilde v))^2 = \frac{1}{\eps} \int_{S_\eps} Q_R^2\Big(e(\tilde u)- e(u) + \frac{1}{2} \grad \tilde v \otimes \grad \tilde v - \frac{1}{2} \grad v \otimes \grad v \Big)  +  \frac{1}{12\eps} \int_{S_\eps} Q^2_R\big(\grad^2 \tilde v - \grad^2 v \big).
\end{align}
Here, $Q_R^2$ is a quadratic form on $\R^{2 \times 2}$, positive definite on $\R^{2 \times 2}_{\rm sym}$, which corresponds to the tensor \EEE $\C_R^2$  \EEE in \eqref{eq: equation-simp-intro2} and is derived from a \EEE nonlinear \EEE viscous dissipation potential $R$, \EEE see \eqref{eq:viscoel}. \EEE

As shown in \cite{MFMKDimension, MFMKJV}, the metric gradient flow of $\phi_\eps$ with respect to the metric $\mathcal{D}_\eps$, for a given initial datum $(u_0,v_0) \in W^{1,2}(S_\eps; \R^2) \times W^{2,2}(S_\eps)$, corresponds to the \emph{2D equations for viscoelastic  von K\' arm\' an plates} \EEE
\begin{align}\label{eq: equation-simp}
\begin{cases} \EEE - \EEE {\rm div}\Big(\C^2_W\big( e(u)  + \frac{1}{2} \nabla v \otimes \nabla v  \big) +  \C^2_R \big( e(\partial_t u) +   \nabla \partial_t v   \odot \nabla v \big) \Big)   = g^{2D}_\eps, &\vspace{0.1cm} \\ 
-{\rm div}\Big(\Big(\C^2_W\big( e(u)  + \frac{1}{2} \nabla v \otimes \nabla v  \big)  +  \C^2_R \big( e(\partial_t u) +   \nabla \partial_t v   \odot \nabla v \big) \Big) \nabla v \Big)  &\vspace{0.1cm}\\
\quad\quad\quad\quad\quad\quad\quad \quad \ \, \quad \ \  \quad  + \tfrac{1}{12}  {\rm div} \, {\rm div}\Big( \C^2_W \nabla^2 v + \C^2_R \nabla^2 \partial_t v \Big)   =   f^{2D}_\eps  &   \text{in } [0,\infty) \times S_\eps, \\
u(0,\cdot) = u_0, \  v(0,\cdot) = v_0 &   \text{in } S_\eps, \\
\end{cases}
\end{align}
 where $\odot$ is the symmetrized tensor product and $\rm div$ denotes the distributional divergence. The existence of solutions to \eqref{eq: equation-simp} complemented with Dirichlet boundary conditions on $\partial S_\eps$  has been addressed in \cite{MFMKDimension, MFMKJV}. For performing the dimension reduction, we instead only impose boundary conditions on the lateral boundaries $\partial I \times (-\tfrac{\eps}{2}, \tfrac{\eps}{2})$. More precisely, given functions $\hat u_1 \in W^{1,2}(I)$, $\hat u_2 \in W^{2,2}(I)$, and $\hat v \in W^{2,2}(I)$, we \EEE define the space of admissible functions by
\begin{align}\label{admissiblefunctions2Dboundary}
{\mathscr{S}}^{2D}_\eps := \Big\{&(u,v) \in W^{1,2} (S_\eps; \R^2 ) \times W^{2,2} \BBB(S_\eps)\EEE: \quad  \\
&u(x) = (\hat u_1(x_1) - x_2 \hat u_2'(x_1), \tfrac{1}{\eps}\hat u_2(x_1)),\   v(x) = \hat v(x_1), \  \partial_1 v(x) = \partial_1 \hat v(x_1) \text{ on } \partial I \times (-\tfrac{\eps}{2}, \tfrac{\eps}{2})
\Big\}. \nonumber 
\end{align}
Note that the proof in \cite{MFMKDimension, MFMKJV} can still be performed under \eqref{admissiblefunctions2Dboundary} up to minor adjustments, but the equations \eqref{eq: equation-simp} need to be complemented with zero Neumann boundary conditions on the top and the bottom of $S_\eps$. As we will see below, however, they do not affect the effective 1D model. The structure of the conditions on $u$ \EEE is \EEE related to the space of Bernoulli-Navier functions, see \eqref{def:BN} below. 

We say that $(u,v) \in W^{1,2}([0,\infty);\BBB \mathscr{S}^{2D}_\eps) \EEE$ is a \emph{weak solution} of \eqref{eq: equation-simp}  if $u(0,\cdot) = u_0$, $v(0,\cdot) = v_0$ and for a.e.\ $t \ge 0$ we have \NNN
\begin{subequations}\label{eq: weak equation}
	\begin{align}
	& \int_{S_\eps} \Big(\C^2_W\big( e(u)  + \tfrac{1}{2} \nabla v \otimes \nabla v  \big)  +  \C^2_R \big( e(\partial_t u ) +  \nabla \partial_t v   \odot  \nabla v \big) \Big) : \nabla \phi_u    = \int_{S_\eps} g^{2D}_\eps \EEE \cdot \EEE \phi_u,\label{eq: weak equation1} \\
	& \int_{S_\eps} \Big(\C^2_W\big( e(u)  + \tfrac{1}{2} \nabla v \otimes \nabla v  \big)  +  \C^2_R \big( e(\partial_t u ) +  \nabla \partial_t v  \odot\nabla v \big) \Big) : \big(\nabla v  \odot  \nabla \phi_v  \big) \notag \\
	& \quad\quad\quad\quad\quad\quad\quad \ \  \ \  \, \,\quad  + \frac{1}{12}  \int_{S_\eps} \Big(\C^2_W \nabla^2 v + \C^2_R \nabla^2 \partial_t v \Big) : \nabla^2 \phi_v = \int_{S_\eps} f^{2D}_\eps\phi_v \label{eq: weak equation2} 
	\end{align}
\end{subequations} 
for all $(\phi_u, \phi_v) \in \mathscr{S}^0_\eps$, where $\mathscr{S}^0_\eps$ is defined as in \eqref{admissiblefunctions2Dboundary} for $\hat{u}_1 = \hat{u}_2 = \hat{v} = 0$. Note that \eqref{eq: weak equation1} corresponds \second to two scalar equations and \eqref{eq: weak equation2} corresponds to one \second  scalar \rm \EEE equation, respectively.

\smallskip

\subsection{Compactness and limiting variables}\label{topologycompactness} For the limiting passage, \EEE it is more convenient to work on a fixed domain that does not depend on $\eps$. To this end, \EEE we introduce the set $S:= I \times (-\frac{1}{2},\frac{1}{2})$ and define the scaled versions $y \in W^{1,2}(S;\R^2)$ and $w \in \EEE W^{2,2}(S)\EEE$ by
\begin{align}
y_1(x_1,x_2):= u_1(x_1, \eps x_2),\quad y_2(x_1,x_2):= \eps u_2(x_1,\eps x_2), \quad w(x_1,x_2):= v(x_1, \eps x_2), \label{scaledfunctions}
\end{align}
and the scaled differential operators \BBB $E^\eps$, $\grad_\eps$, $\grad^2_\eps$ by \EEE
\begin{align*}
E^\eps y:= \begin{pmatrix}
\partial_1 y_1 & \frac{1}{2\eps}(\partial_1 y_2+ \partial_2 y_1) \\ 
\frac{1}{2\eps}(\partial_2 y_1 + \partial_1 y_2) & \frac{1}{\eps^2}\partial_2 y_2 \\ 
\end{pmatrix},
\end{align*}
\begin{align}
\grad_\eps w := \Big( \partial_1 w, \frac{1}{\eps} \partial_2 w \Big) \quad \text{ and }\quad  \grad^2_\eps w := \begin{pmatrix}
\partial_{11}^2 w & \frac{1}{\eps} \partial_{12}^2 w \\
\frac{1}{\eps} \partial_{21}^2 w & \frac{1}{\eps^2} \partial_{22}^2 w
\end{pmatrix}. \label{scaledoperators}
\end{align}
Thus, by the chain rule we have $$E^\eps y(x) =  e(u) (x_1,\eps x_2),\quad \grad_\eps w (x) = \grad v(x_1,\eps x_2), \quad \grad^2_\eps w (x) = \grad^2 v(x_1,\eps x_2).$$
Similarly, we define the scaled forces $\hat{f}^{2D}_{\eps}\colon S \to \R$ and $\hat{g}^{2D}_{\eps} \colon S \to \R^2$ by
$\hat{f}^{2D}_\eps(x_1,x_2) = f^{2D}_\eps(x_1, \eps x_2)$ and $\hat{g}^{2D}_\eps(x_1,x_2) = g^{2D}_\eps(x_1, \eps x_2)$
and assume that the scaled versions satisfy
\begin{align}\label{eq: forces}
\hat{f}^{2D}_\eps \rightharpoonup f^{1D},\qquad 	\hat{g}^{2D}_{\eps,1} \rightharpoonup g_1^{1D}, \qquad \frac{1}{\eps} \hat{g}^{2D}_{\eps,2} \rightharpoonup g_2^{1D} \qquad \text{weakly in } L^2(S)
\end{align}
for functions $f^{1D} \in L^2(I)$ and $g^{1D} \in L^2(I;\R^2)$. Here and in the following, with a slight abuse of notation, we regard all functions defined on $I$ as functions on $S$ which do not \BBB depend \EEE explicitly on the variable $x_2$.

Starting from the variables $y$ and $w$ in dimension two, we now \EEE introduce \EEE corresponding limiting variables in the one-dimensional setting. \BBB  We will identify the limit of the in-plane displacements with two-dimensional Bernoulli-Navier functions \EEE defined by

%\begin{align}
% BN_\lr(S,\R^2):= &\{ g \in W_\lr^{1,2}(S, \R^2) : (Eg)_{12} = (Eg)_{21} = (Eg)_{22} = 0 \} \label{def:BN} \\
%= &\{ g \in W_\lr^{1,2}(S, \R^2) : \exists \xi_1 \in W_\lr^{1,2}(I) \text{ and } \xi_2 \in W_\lr^{1,2}(I) \cap W^{2,2}(I)  \text{ such that}  \nonumber\\
% &\quad g_1(x) = \xi_1(x_1) - x_2 \xi_2^\prime (x_1),\quad  g_2(x) = \xi_2(x_1)\}, \nonumber
%\end{align}
%where the second characterization can be derived as in \cite{???}.

\begin{align*}
BN_{ (\hat u_1,\hat u_2)}(S;\R^2):= &\Big\{ y \in W^{1,2}(S; \R^2): e(y)_{12} = e(y)_{22} = 0 ,  \quad y_1 = \hat u_1 - x_2 \hat u_2', \quad y_2 = \hat u_2  \ \text{ on } \Gamma \Big\},
\end{align*}
where for shorthand, we set $\Gamma:=\partial I \times (-\tfrac{1}{2}, \tfrac{1}{2})$. Compared to \cite{Freddi2018}, this function space is different since in our analysis we consider functions with boundary values instead of functions with vanishing mean. By arguing analogously to \cite[Theorem 4.1]{ABJMV}, the space of Bernoulli-Navier functions can be identified with \EEE  functions defined \EEE on \EEE   $I$, namely
\begin{align}\label{def:BN}
BN_{(\hat u_1,\hat u_2)}(S;\R^2)= & \Big\{ y \in W^{1,2}(S; \R^2) : \exists \,   \xi_1 \in W_{\hat u_1}^{1,2}(I), \, \exists \,  \xi_2 \in  W_{\hat u_2}^{2,2}(I) \, \text{  such that}  \nonumber\\
& \ \  \quad \quad \quad \quad \quad \quad \quad   y_1(x) = \xi_1(x_1) - x_2 \xi_2^\prime (x_1), \ \  y_2(x) = \xi_2(x_1) \Big\}. 
\end{align}
Here, it is worth noting that the second component has a higher regularity and that $\xi_2' = \hat{u}_2'$ on $\partial I$ as $ \xi_1 - x_2 \xi_2^\prime = \hat u_1 - x_2 \hat u_2'$ on $\partial I$ for a.e.\ $x_2 \in (-\frac{1}{2},\frac{1}{2})$. We recall the scalings \eqref{scaledfunctions} and \eqref{scaledoperators}\ZZZ, and denote by ${\mathscr{S}}^{2D}_{\eps,M} := \{(u,v) \in \BBB \mathscr{S}^{2D}_\eps \EEE \colon \, {\phi}_\eps(u,v) \leq M \}$ the sublevel sets of the energy.\EEE 

\begin{proposition}[Compactness]\label{thm:compactness} Let $(u_\eps,v_\eps)_\eps$ be a sequence such that $(u_\eps,v_\eps) \in \mathscr{S}^{2D}_{\eps,M}$ for $M>0$. Let $(y_\eps,w_\eps)_\eps$ be the scaled sequence in the sense of \eqref{scaledfunctions}. 
Then, up to a subsequence, there exists a vertical displacement $w \in W_{\hat v}^{2,2}(I)$, a twist function $\theta \in W_{0}^{1,2}(I)$, and a horizontal displacement $y \in BN_{(\hat u_1,\hat u_2)}(S;\R^2)$ such that
\begin{align*}
&w_\eps \rightharpoonup w \text{ in } W^{2,2}(S), \quad \grad_\eps w_\eps \rightharpoonup (w',\theta) \text{ in } W^{1,2}(S;\R^2),\\
&\grad_\eps^2 w_\eps \rightharpoonup \begin{pmatrix}
w'' & \theta' \\
\theta' & \gamma\\
\end{pmatrix} \text{ in } L^2(S; \R^{2\times 2}_{{\rm sym}})
\end{align*}
for a suitable $\gamma \in L^2(S)$ and
\begin{align*}
y_\eps \rightharpoonup y \text{ in } W^{1,2}(S;\R^2) \quad \text{ and } \quad E^\eps y_\eps \rightharpoonup E \text{ in } L^2(S;\R^{2 \times 2}_{\rm sym})
\end{align*}
for a suitable $E\in L^2(S;\R^{2\times 2}_{\rm sym})$ such that $E_{11} = \partial_1 y_1$.
\end{proposition}
 The proof is omitted as it closely follows the lines of \cite[Lemma 2.1]{Freddi2018} where functions with vanishing mean have been considered. In fact, the only difference lies in using suitable versions of Poincaré's and Korn's  inequality for functions with given trace, and in checking that the limits satisfy the boundary conditions. To see the latter, it suffices to observe that $y_\eps(x) = (\hat u_1(x_1) - x_2 \hat u_2'(x_1), \EEE \hat{u}_2 \EEE (x_1) )$, $w_\eps(x) = \hat v(x_1)$,  and $\frac{1}{\eps}\partial_2 w_\eps =0$ on $\Gamma$, \EEE see \eqref{admissiblefunctions2Dboundary} and \eqref{scaledfunctions}. \EEE Here, our choice of the boundary values in \eqref{admissiblefunctions2Dboundary} becomes apparent since it guarantees that the limit of the in-plane \BBB displacements can \EEE be identified with functions in \eqref{def:BN}. We also refer to \cite{lecumberry}, where clamped boundary conditions in a related context are \EEE considered. \EEE Later we  will see that the compactness result also holds in the time-dependent setting along solutions to \eqref{eq: weak equation}. \EEE

\subsection{Effective quadratic forms and compatibility conditions}\label{sec: compatibility} As a preparation for the formulation of the one-dimensional model, we introduce effective quadratic forms related to $Q^2_S$, $S=W,R$, introduced in Subsection \ref{sec: two-d}.  Recall that \EEE  $Q_W^2$ and $Q_R^2$ are defined on $\R^{2 \times 2 }$ which depend only on symmetric matrices and can thus be identified   with functions defined on $\R^3$ via $(q_{11},q_{12},q_{22}) \simeq {\footnotesize \begin{pmatrix}
q_{11} & q_{12} \\ 
q_{12} & q_{22}\end{pmatrix}}$. For the sake of readability, we use both types of notation in the sequel. We define reduced quadratic forms by minimizing of the second and third entry. More precisely, we let
\begin{align}
&Q_S^1(q_{11}, q_{12}) := \min\limits_{\alpha \in \R}  Q_S^2(q_{11}, \EEE  q_{12}, \EEE \alpha) \label{quadraticform1d}
\end{align}
for $(q_{11}, q_{12}) \in \R^2$ and 
\begin{align}
Q_S^0( q_{11} ):= \min\limits_{z \in \R} Q_S^1(q_{11},z) = C_S^0 q_{11}^2 \label{quadraticform0d}
\end{align}
for $q_{11} \in \R$ and a suitable constant $C_S^0>0$. We denote by $ \C_W^1$ and $ \C_R^1$ the corresponding second-order tensors. 

To perform a rigorous evolutionary dimension reduction, we require some compatibility conditions of the quadratic forms $Q_W^2$ and $Q_R^2$ as we need to construct  \emph{mutual \EEE recovery sequences}, compatible for the \EEE elastic energy and the viscous dissipation at the same time (see Theorems \ref{theorem: lsc-slope} and \ref{thm:gammaofscheme})\EEE. A first possibility is given by the assumption that
\begin{align}
\argmin_{\alpha \in \R}  Q_S^2(q_{11}, q_{12} , \alpha)  = 0 , \quad \quad  \argmin_{z \in \R} Q_S^1(q_{11},z) = 0 \quad \quad  \text{for all  $q_{11},q_{12} \in \R$}. \tag{\textbf{H1}} \label{quadraticforms}
\end{align}
This induces a restriction from \EEE a \EEE  modeling point of view since it  particularly corresponds to  materials with Poisson ratio zero, \ZZZ such as cork.  \EEE A reasonable generalization is to assume that   $Q_S^2$ are $\eps$-dependent, denoted by $Q_{S,\eps}^2$, such that $Q_{S,\eps}^2 = Q_S^2 + \RRR{\rm o}(1)\EEE \hat{Q}_{S}$ \EEE for $\eps \to 0$, \EEE where $Q_S^2$ satisfies \eqref{quadraticforms} and $\hat{Q}_S$ is any positive definite quadratic form. (For simplicity, we \BBB did \EEE not include this explicitly in the notation in \eqref{def:vK2D} and \eqref{eq: metriceps}.)   Another sound option is to consider thin materials with general Poisson ratio, but with a vanishing dissipation effect in the $e_2$ direction. In this case, the assumptions are given by  
\begin{align}
\lim_{\eps \to 0}  Q_{R,\eps}^2(q_{11}, \EEE q_{12} \EEE , \alpha)  = Q_{R}^1(q_{11}, \EEE q_{12}\EEE ), \quad \quad  \argmin_{z \in \R} Q_S^1(q_{11},z) = 0  \text{ for $S=W,R$}  \tag{\textbf{H2}} \label{eq:quadraticform:minimumnot0}
\end{align} 
\EEE for \EEE  all $q_{11},q_{12}, \EEE \alpha \EEE \in \R$. This setting includes, but is not restricted to, the case of  linear isotropic elastic materials with corresponding quadratic form
\begin{align*}
	Q_{W}^2(q_{11},q_{12}, q_{22}):= 2 \mu (q_{11}^2 + 2 q_{12}^2 + q_{22}^2) + \lambda (q_{11}+ q_{22})^2,
\end{align*}
where $\mu>0$ and $\lambda \in \R$ are suitable Lam\'e parameters. In particular, \eqref{quadraticforms} corresponds to $\lambda = 0$. In this paper, we cover both cases described above. We remark that \eqref{quadraticforms} and \eqref{eq:quadraticform:minimumnot0} are only needed for the proof of Lemma \ref{lem:mutual}.

%%For instance, this can be achieved by assuming that the minima of $Q_S^i$, $i = 1,2$, $S = W,R$ are attained at 0. Then, the quadratic forms take the form
%\begin{align*}
%Q_S^2(q_{11},q_{12},q_{22}) = a_S q_{11}^2 + b_S q_{12}^2 + c_S q_{22}^2
%\end{align*}
%for suitable coefficients $a_S,b_S,c_S >0$. 
% Therefore, it is reasonable to assume that $Q_S^2$ additionally depends on quadric forms of order $o_\eps(1)$. More precisely, the quadratic forms take the form
%\begin{align}
%Q_{S,\eps}^2(q_{11},q_{12},q_{22}) =  Q_S^2(q_{11},q_{12},q_{22}) + o_\eps(1) \hat{Q}(q_{11},q_{12},q_{22}), \tag{\textbf{H1}} \label{quadraticforms}
%\end{align}
%where $ Q_S^2$ attains its minima at $0$ and $\hat{Q}$ is an arbitrary positive definite quadratic form.
%Another sound assumption is to consider thin materials with a positive or negative Poissons ratio, but with a vanishing dissipation effect into the $e_2$ direction, which can be expressed by the quadratic forms 
%\begin{align}
%&Q_W^2(q_{11},q_{12},q_{22}) := a q_{11}^2 + b q_{12}^2 + c q_{22}^2 + ( q_{11} +d q_{22})^2 \quad\text{and} \nonumber \\
%&Q_R^2(q_{11},q_{12},q_{22}) := e q_{11}^2 +f q_{12}^2 + o_\eps(1) \tilde Q(q_{22}) \tag{\textbf{H2}} \label{eq:quadraticform:minimumnot0}
%\end{align} for positive coefficients $a,b,c,d,e,f>0$ and a positive definite quadratic form $\tilde Q$. 
%
%

\subsection{Equations of viscoelastic vK ribbons in 1D} We now present the effective 1D equations. To this end, define the set of admissible functions by 
\begin{align}\label{eq: mathcalK}
\mathcal{K}:= W_{\hat u_1}^{1,2}(I) \times  W_{\hat u_2}^{2,2}(I)   \times W_{\hat v}^{2,2}(I) \times W_0^{1,2}(I) .
\end{align}
In the following, we write  $^\prime$ for spatial- and $\partial_t$ for time derivatives. Recall the definition of the forces in \eqref{eq: forces}. \EEE 
Given $(\xi^0_1,\xi^0_2(t), w^0,\theta^0) \in \mathcal{K}$, we consider the \EEE system of equations  
\begin{align*}
g^{1D}_1 = &- \ZZZ \mfrac{\rm d}{{\rm d}x_1}\EEE \bigg(C_W^0\Big(\xi_1^\prime+\frac{\vert w^\prime\vert^2}{2}\Big) + C_R^0(\partial_t \xi_1^\prime+w^\prime \partial_t w^\prime )   \bigg) ,\\
g^{1D}_2 = &\ZZZ \mfrac{1}{12}\mfrac{{\rm d}^2}{{\rm d}x_1^2}\EEE\Big(C_W^0 \xi_2^{\prime\prime} +C_R^0 \partial_t \xi_2^{\prime\prime}   \Big) ,  \\
f^{1D}=& - \mfrac{\rm d}{{\rm d}x_1} \Bigg(\bigg(C_W^0\Big(\xi_1^{\prime} + \frac{\vert \wpr \vert^2}{2}\Big) + C_R^0(\partial_t \xi_1^{\prime} + w^\prime \partial_t w^\prime) \bigg) w^\prime \Bigg) \nonumber \\
&+
\mfrac{1}{24} \mfrac{{\rm d}^2}{{\rm d}x_1^2} \Big(\partial_1 Q_W^1(w^{\prime\prime},\theta^\prime)+\partial_1Q_R^1(\partial_t\wprpr,\partial_t \theta^\prime )\Big), \\
0 =&   \, \mfrac{\rm d}{{\rm d}x_1}\Big(\partial_2 Q_W^1(w^{\prime\prime},\theta^\prime)+\partial_2 Q_R^1(\partial_t \wprpr,\partial_t \theta^\prime)\Big) &&\text{in } [0,\infty) \times I \EEE
\end{align*}
such that $\xi_1(0,\cdot) = \xi_1^0$,  $\xi_2(0,\cdot) = \xi_2^0$,  $w(0,\cdot) = {w}^0$,  $\theta(0,\cdot) = {\theta}^0$ in $I$ and $(\xi_1(t),\xi_2(t), w(t),\theta(t)) \in \mathcal{K}$ for $t \in [0,\infty)$. We also say that $(\xi_1,\xi_2,w, \theta) \in W^{1,2}([0, \infty); \mathcal{K})$ is a \textit{weak solution} if $\xi_1(0,\cdot) = \xi_1^0$,  $\xi_2(0,\cdot) = \xi_2^0,$  $w(0,\cdot) = {w}^0,$  $\theta(0,\cdot) = {\theta}^0$ and for a.e. $t \ge 0$, we have
\begin{subequations}\label{eq:weak1d}
\begin{align}
0 = &\int_I  \bigg(C_W^0\Big(\xi_1^\prime+\frac{\vert w^\prime\vert^2}{2}\Big) + C_R^0(\partial_t \xi_1^\prime+w^\prime \partial_t w^\prime )   \bigg) \phi_{\xi_1}^\prime - \int_I g^{1D}_1 \phi_{\xi_1}, \label{equation:1}\\
0 = &\int_I \frac{1}{12}\Big(C_W^0 \xi_2^{\prime\prime} +C_R^0 \partial_t \xi_2^{\prime\prime}   \Big) \phi_{\xi_2}^{\prime\prime} - \int_I g^{1D}_2 \phi_{\xi_2},  \label{equation:2} \\
0=& \int_I \bigg(C_W^0\Big(\xi_1^{\prime} + \frac{\vert \wpr \vert^2}{2}\Big) + C_R^0(\partial_t \xi_1^{\prime} + w^\prime \partial_t w^\prime) \bigg) w^\prime \phi_w^\prime \nonumber \\
&+
\frac{1}{24} \int_I \Big(\partial_1 Q_W^1(w^{\prime\prime},\theta^\prime)+\partial_1Q_R^1(\partial_t\wprpr,\partial_t \theta^\prime )\Big) \phi_w^{\prime\prime}- \int_I f^{1D}\phi_w , \label{equation:3}\\
0 =& \int_I \Big(\partial_2 Q_W^1(w^{\prime\prime},\theta^\prime)+\partial_2 Q_R^1(\partial_t \wprpr,\partial_t \theta^\prime)\Big) \phi_\theta^{\prime}  \label{equation:4}
\end{align}
\end{subequations}
for all $\phi_{\xi_1} \in W_0^{1,2}(I)$, $\phi_{\xi_2} \in W^{2,2}_0(I)$, ${\phi_{w}} \in  W_0^{2,2}(I)$, and $\phi_\theta \in W_0^{1,2}(I)$.

We point out that  \eqref{equation:2} describing the orthogonal in-plane displacement $\xi_2$ is completely decoupled from the other equations, whereas the axial in-plane displacement $\xi_1$ is always coupled to the vertical displacement $w$ by \eqref{equation:1} and \eqref{equation:3}. Interestingly, under assumption \EEE \eqref{quadraticforms}, \EEE \BBB one can check that the twist function only appears in \eqref{equation:4}\BBB, \EEE and \eqref{equation:4} is also independent of $w$ in this setting, i.e., \eqref{equation:4} decouples completely, as well.

 Our goal will be to show existence of weak solutions to \eqref{eq:weak1d} and that weak \BBB solutions \EEE \eqref{eq: weak equation} converge to weak solutions \eqref{eq:weak1d} in a suitable sense. In particular, we will relate \eqref{eq:weak1d} \EEE to a metric gradient flow \EEE with respect to \EEE an energy $\phi_0$ in the space \EEE 
\begin{align}\label{eq: s1D}
 \mathscr{S}^{1D} := BN_{(\hat u_1, \hat u_2)}(S;\R^2) \times W_{\hat v}^{2,2}(I) \times W^{1,2}_0(I),
 \end{align}
 endowed with a metric $\mathcal{D}_0$ \ZZZ whose square is given by \EEE
\begin{align}
\mathcal{D}_0((y,w,\theta),(\tilde y, \tilde w, \tilde \theta))^2 :=&\int_S Q_R^0\Big(\partial_1 y_1 - \partial_1 \tilde y_1 + \frac{\vert w^\prime\vert^2}{2} - \frac{\vert \tilde w^{\prime}\vert^2}{2}\Big)  + \frac{1}{12}\int_I Q_R^1(w^{\prime\prime}- \tilde w^{\prime\prime}, \theta^\prime - \tilde \theta^{\prime})\label{def:metric}
\end{align}
for $(y,w,\theta), (\tilde y, \tilde w, \tilde \theta) \in \mathscr{S}^{1D} $ and the energy $\phi_0$ is given by
\begin{align}
\phi_0(y,w,\theta):= &\frac{1}{2} \int_S Q_W^0\Big(\partial_1 y_1 + \frac{\vert w^{\prime}\vert^2}{2}\Big)  + \frac{1}{24} \int_I Q_W^1(w^{\prime\prime}, \theta^\prime)  -\int_I f^{1D} w + g^{1D}\cdot y  \label{def:enegeryphi}
\end{align}
for $(y,w,\theta) \in \mathscr{S}^{1D}$. Note that \eqref{def:enegeryphi} coincides with \eqref{def:enegeryphi-intro} (for $f^{1D} =0$, $ g^{1D} = 0$) by using \eqref{def:BN} \EEE and by performing \ZZZ an \EEE integration over $x_2$, where one uses $\int_{-1/2}^{1/2} x_2 = 0$ and $\int_{-1/2}^{1/2} x_2^2 = 1/12$.  \EEE  For notational convenience, we  work with $\mathscr{S}^{1D}$ instead of the (equivalent) space $\mathcal{K}$, i.e., we identify $(\xi_1,\xi_2)$ with $y$ via \eqref{def:BN}. \EEE

\subsection{Main results}\label{sec: Main results}
To show existence and convergence of solutions, we will use the abstract theory of gradient flows \cite{AGS} and evolutionary $\Gamma$-convergence \cite{Mielke, S1, S2}. In particular, our approach to prove existence of 1D solutions  is twofold as we derive \MMM it \EEE both \EEE  by time-discrete approximations and also by limits of \EEE two-dimensional solutions. 

Our first main result addresses the existence of time-discrete solutions to the one-dimensional problem and their \EEE convergence to a curve of maximal slope for $\phi_0$ \EEE with respect to $\mathcal{D}_0$ and $\vert \partial \phi_0 \vert_{\mathcal{D}_0}$. For the main definitions and notation for curves of maximal slope and strong upper gradients we refer to  Subsection~\ref{sec: defs}.  In particular, we write $\vert \partial \phi_\eps \vert_{\mathcal{D}_\eps}$ and $\vert \partial \phi_0 \vert_{\mathcal{D}_0}$ for the local slopes, where the energies and metrics are defined in \eqref{def:vK2D}, \eqref{eq: metriceps}, \eqref{def:metric}, and \eqref{def:enegeryphi}. The definition of \EEE time-discrete solutions is given in Subsection~\ref{sec: auxi-proofs}. \EEE The relevant results about existence of \EEE curves of maximal slope are recalled in Subsections \ref{sec: auxi-proofs} and \ref{sec: auxi-proofs2}.

%\begin{theorem}[Solutions in the one-dimensional setting]
%\label{thm:curveofmaximalslope:energyidentity}
%Suppose ${\xi_1}_0 \in W_\lr^{1,2}(I)$, ${\xi_2}_0 \in W_\lr^{1,2}(I) \cap  W^{2,2}(I)$, $w_0 \in W^{2,2}_\lr(I)$ and $\theta_0 \in W^{1,2}_\lr(I)$ such that $({\xi_1}_0 - x_2{\xi_2}_0', {\xi_2}_0 ) \in BN_\lr(S,\R^2)$. Then,
%
%\begin{itemize}
%\item[(i)] (Existence) there exists a curve of maximal slope $(y,w,\theta) \in AC^2([0,\infty);(\mathscr{S}^{1D},D))$ such that $y(0) = y_0 $, $w(0) = w_0$ and $\theta(0) = \theta_0$, where ${y_0}_1(x):= {\xi_1}_0(x_1) - x_2 {\xi_2}_0'(x_1)$ and ${y_0}_2(x):= {\xi_0}_2(x_1)$. In particular, it holds for all $T >0$ that
%\begin{align*}
%\frac{1}{2} \int_0^T \vert (y,w,\theta)'\vert_D^2(t) dt + \frac{1}{2} \int_0^T \vert \partial \phi \vert^2_D\big( (y,w,\theta)(t) \big) dt + \phi \big( (y,w,\theta)(T) \big) = \phi \big( (y_0,w_0,\theta_0) \big).
%\end{align*}
%\item[(ii)] (Approximation) there exists a subsequence $(\tau_k)_{k \in \N}$ such that a subsequence of discrete solutions $(\bar U_{\tau_k})_{k \in \N}$ defined in \eqref{eq: ds-new2} satisfies $\bar U_{\tau_k}(t) \to (y(t),w(t),\theta(t))$ strongly for all $t \in [0,\infty)$.
%\end{itemize}
%\end{theorem}

\begin{theorem}[Solutions in the one-dimensional setting] 
\label{thm:curveofmaximalslope:energyidentity}
Consider  $(\xi_1^0, \xi_2^0, w^0,\theta^0) \in \mathcal{K}$ and define $y^0 := (\xi_1^0 - x_2 (\xi^0_2)', \xi_2^0)$, i.e., $(y^0,w^0,\theta^0) \in \mathscr{S}^{1D}$.  
\begin{itemize}
\item[(i)] {\rm{(Approximation and existence)}}  For each null sequence $(\tau_l)_{l\in \N}$ and each sequence of discrete solutions $(\bar U_{\tau_l})_l$ as in \eqref{eq: ds-new2} below, there exists an absolutely continuous function $(y,w,\theta)\colon [0,\infty) \to \mathscr{S}^{1D}$ with respect to the metric $\mathcal{D}_0$ satisfying $(y,w,\theta)(0) = (y^0,w^0,\theta^0)$ such that, up to a subsequence, not relabeled,
$$\bar U_{\tau_l}(t) \to (y(t),w(t),\theta(t)) \quad  \quad \text{for all $t \in [0,\infty)$ as $l \to \infty$}$$ 
\EEE with respect to the topology induced by $\mathcal{D}_0$, \EEE  and $(y,w,\theta)$ is a curve of maximal slope for $\phi_0$ with respect to $\vert \partial \phi_0 \vert_{\mathcal{D}_0}$.
\item[(ii)] {\rm{(Identification)}}  Each \EEE curve of maximal slope $(y,w,\theta)$ for $\phi_0$ with respect to $\vert \partial \phi_0 \vert_{\mathcal{D}_0}$ can be   identified via \eqref{def:BN}  \EEE with a curve 
\begin{align}
(\xi_1,\xi_2,w,\theta) \in W^{1,2}\big([0,\infty);\mathcal{K} \big) \label{regularitylimiting}
\end{align}
such that $(\xi_1,\xi_2,w,\theta)$ is  a weak solution of the system  \eqref{eq:weak1d}.
\end{itemize}
\end{theorem}

%\begin{align*}
%\mathcal{K}:= W_{\hat u_1}^{1,2}(I) \times \big( W^{1,2}(I) \cap W_{\hat u_2}^{2,2}(I) \big) \times W_{\hat v}^{2,2}(I) \times W_0^{1,2}(I) .
%\end{align*}
%The next theorem identifies curves of maximal slopes in the one-dimensional setting with solutions of the one-dimensional equations \eqref{eq:weak1d}.

%\begin{theorem}[Relation to one-dimensional PDE's]\label{identification} 
%	A curve of maximal slope $(y,w,\theta)$ for $\phi_0$ with respect to $\vert \partial \phi_0 \vert_{\mathcal{D}_0}$ can be identified with a curve \begin{align}
%	(\xi_1,\xi_2,w,\theta) \in W^{1,2}\big([0,\infty); W_\lr^{1,2}(I) \times W_\lr^{1,2}(I) \cap W^{2,2}(I) \times W_\lr^{2,2}(I) \times W_\lr^{1,2}(I) \big) \label{regularitylimiting}
%	\end{align}
%	such that $(\xi_1,\xi_2,w,\theta)$ is indeed a weak solution of the system of partial differential equations \eqref{eq:weak1d}.
%\end{theorem}

Now,  we study the relation of weak solutions  \eqref{eq: weak equation} and weak solutions \eqref{eq:weak1d}. \EEE To this end, we need to specify the topology of the convergence. \EEE  We define mappings $\pi_\eps\colon {\mathscr{S}}^{2D}_\eps \to \BBB {\mathscr{S}}\EEE$ by $\BBB \pi_\eps(u,v)\EEE:= (y,w, \tfrac{\partial_2 w}{\eps})$, where
$y$ and $w$ are the scaled in-plane and out-of-plane displacements corresponding to
$u$ and $v$, see \eqref{scaledfunctions}, \BBB and  $\mathscr{S} = \pi_\eps(\mathscr{S}^{2D}_\eps)$. \EEE We say that $\pi_\eps (u_\eps,v_\eps) = (y_\eps , w_\eps, \tfrac{\partial_2 w_\eps}{\eps}) \stackrel{\sigma}{\to} (y,w,\theta)$ if we have the convergence in the sense of Proposition \ref{thm:compactness}. Furthermore, we say that $(u_\eps,v_\eps) \stackrel{\pi\sigma}{\to} (y,w,\theta)$ if $\pi_\eps (u_\eps,v_\eps) \stackrel{\sigma}{\to} (y,w,\theta)$.
The sequence $(u_\eps, v_\eps)_\eps$ converges strongly to $(y,w,\theta)$, written $(u_\eps,v_\eps) \stackrel{\pi\rho}{\to} (y,w,\theta)$, if the convergence in Proposition \ref{thm:compactness} also holds with respect to the strong in place of the weak topology. We remark that the limiting variables $(y,w,\theta)$ are contained in the space \EEE $\mathscr{S}^{1D} \subset \mathscr{S}$ \EEE defined in \eqref{eq: s1D}.\EEE

\begin{theorem}[Relation between two-dimensional and one-dimensional system\EEE]\label{thm:relation2d1d} \ZZZ Suppose that \eqref{quadraticforms} or \eqref{eq:quadraticform:minimumnot0} holds. 
Consider a null sequence $(\eps_l)_{ l\in \N}$ and a sequence of initial data $(u^0_{\eps_l},v^0_{\eps_l})_{\eps_l}$ with $(u^0_{\eps_l},v^0_{\eps_l})\in \mathscr{S}_{\eps_l,M}^{2D}$ such that $(u^0_{\eps_l},v^0_{\eps_l}) \stackrel{\pi\sigma}{\to} (y^0,w^0,\theta^0) \in \mathscr{S}^{1D}$ and $\phi_\eps(u^0_{\eps_l},v^0_{\eps_l}) \to \phi_0(y^0,w^0,\theta^0)$. \EEE

\begin{itemize}
\item[(i)] {\rm{(Convergence of continuous solutions)}} Let $(u_{\eps_l},v_{\eps_l})_l$ be a sequence of curves of maximal slopes for $\phi_{\eps_l}$ with respect to $\vert \partial \phi_{\eps_l} \vert_{\mathcal{D}_{\eps_l}}$ satisfying $(u_{\eps_l}(0),v_{\eps_l}(0)) = (u^0_{\eps_l}, v^0_{\eps_l})$. Then, there exists an
absolutely continuous function $(y,w,\theta)\colon [0,\infty) \to \mathscr{S}^{1D}$ with respect to the metric $\mathcal{D}_0$ satisfying $(y,w,\theta)(0) = (y^0,w^0,\theta^0)$
such that, up to a subsequence, not relabeled,
\begin{align*}
\BBB(u_{\eps_l}(t),v_{\eps_l}(t))\EEE \stackrel{\pi\rho}{\to} (y(t),w(t),\theta(t))  \quad \quad \text{for all $t \in [0,\infty)$ as $l \to \infty$}
\end{align*}
and $(y,w,\theta)$ is a curve of maximal slope for $\phi_0$ with respect to $\vert \partial \phi_0\vert_{\mathcal{D}_0}$.
\item[(ii)] {\rm{(Convergence of discrete solutions)}} For all $\tau>0$ and all discrete solutions $\bar Y_{\eps_l,\tau}$ as in \eqref{eq: ds-new2} below there exists a discrete solution \BBB $\bar U_\tau$ \EEE \BBB of \EEE the one-dimensional system such that, up to a subsequence, not relabeled,\begin{align*}
\bar Y_{\eps_l, \tau} \stackrel{\pi\rho}{\to} \BBB\bar U_\tau \EEE\quad \quad \text{for all $t\in [0,\infty)$ as $l \to \infty$.}
\end{align*} 
\item[(iii)] {\rm{(Convergence at specific scales)}} For each null sequence $(\tau_l)_{l\in \N}$ and each sequence of discrete solutions $\bar Y_{\eps_l, \tau_l}$ as in \eqref{eq: ds-new2} below, there exists an absolutely continuous function $(y,w,\theta)\colon [0,\infty) \to \mathscr{S}^{1D}$ with respect to the metric $\mathcal{D}_0$ satisfying $(y,w,\theta)(0) = (y^0,w^0,\theta^0)$ such that, up to a subsequence, not relabeled,
 \begin{align*}
\bar Y_{\eps_l,\tau_l}(t) \stackrel{\pi\rho}{\to} (y(t),w(t),\theta(t)) \quad \quad \text{for all $t\in [0,\infty)$ as
$l \to \infty$}
\end{align*}  
and $(y,w,\theta)$ is a curve of maximal slope for $\phi_0$ with respect to $\vert \partial \phi_0\vert_{\mathcal{D}_0}$.
\EEE
\end{itemize}

\end{theorem}
Note that in the two-dimensional setting the existence of curves of maximal slope in (i) and  discrete solutions in (ii) \BBB and \EEE  (iii)  is guaranteed by  \cite[Theorem 2.2]{MFMKDimension} and  \cite[Theorem 4.1]{MFMKJV}. \EEE
In particular, weak solutions in the two-dimensional setting converge (in the sense above) to solutions of the one-dimensional equations. We refer to Figure \ref{diagram} for an illustration. \EEE  We point out that \EEE existence of solutions to the one-dimensional equations   follows without \eqref{quadraticforms} or \eqref{eq:quadraticform:minimumnot0}, \EEE see \EEE Theorem~\ref{thm:curveofmaximalslope:energyidentity}. These compatibility \EEE conditions are only needed to prove Theorem~\ref{thm:relation2d1d}.

 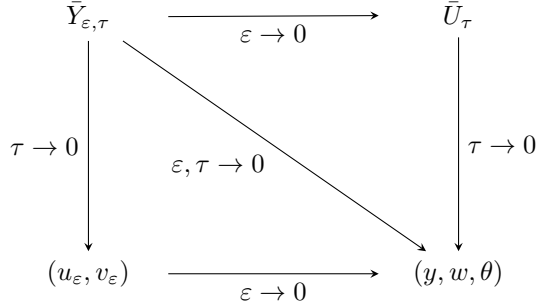
\begin{figure}[H]
 	\centering
 	
 	%\resizebox{300pt}{200pt}{%
 	
 	\begin{tikzpicture}
 	\matrix (m) [matrix of math nodes,row sep=8em,column sep=8em,minimum width=6em, ampersand replacement=\&]
 	{ \bar Y_{\eps,\tau}    \& \bar U_\tau \\
 		(u_\eps,v_\eps) \&  (y,w,\theta) \\
 	};
 	\path[-stealth]
 	(m-1-1) edge node [left] {$\tau \to 0$} (m-2-1)
 	edge   node [below] {$\eps \to 0$} (m-1-2)
 	(m-2-1) edge node [below] {${\eps \to 0}$}               (m-2-2)
 	(m-1-2) edge node [right] {$\tau \to 0$} (m-2-2) 
 	(m-1-1)      
 	edge node [below] {\hspace{-1.5cm}$\eps,\tau\to 0$} (m-2-2)  ;
 	\end{tikzpicture}
 	
 	\caption{Illustration of the commutativity result given in Theorem \ref{thm:curveofmaximalslope:energyidentity} and Theorem~\ref{thm:relation2d1d}, where \EEE $\tau$ indicates the timestep and $\eps$ the width of the ribbon. \EEE The horizontal and diagonal arrows are addressed in Theorem \ref{thm:relation2d1d}. The left vertical arrow is a consequence of \cite[Theorem 4.1]{MFMKJV}, up to minor adaptions due to boundary conditions. The remaining vertical arrow \EEE corresponds to Theorem~\ref{thm:curveofmaximalslope:energyidentity}. } \label{diagram}
 	%	}
 	
 \end{figure}

\section{Preliminaries: Curves of maximal slope}\label{sec3}

In this section, we  recall the relevant definitions about curves of maximal slope and present  abstract theorems concerning the convergence of time-discrete solutions and continuous solutions to curves of maximal slope.   
\subsection{Definitions}\label{sec: defs}

We consider a   complete metric space $(\mathscr{S},\mathcal{D})$. We say a curve $y\colon (a,b) \to \mathscr{S}$ is \emph{absolutely continuous} with respect to $\mathcal{D}$ if there exists $m \in L^1(a,b)$ such that
\begin{align*}
\mathcal{D}(y(s),y(t)) \le \int_s^t m(r) \,\ZZZ {\rm d}r \EEE \ \ \   \text{for all} \ a \le s \le t \le b.
\end{align*}
The smallest function $m$ with this property, denoted by $|y'|_{\mathcal{D}}$, is called \emph{metric derivative} of  $y$  and satisfies  for a.e.\ $t \in (a,b)$   (see \cite[Theorem 1.1.2]{AGS} for the existence proof)
$$|y'|_{\mathcal{D}}(t) := \lim_{s \to t} \frac{\mathcal{D}(y(s),y(t))}{|s-t|}.$$
We  define the notion of a \emph{curve of maximal slope}. We only give the basic definition here and refer to \cite[Section 1.2, 1.3]{AGS} for motivations and more details.  By  $h^+:=\max(h,0)$ we denote the positive part of a function  $h$.

\begin{definition}[Upper gradients, slopes, curves of maximal slope]\label{main def2} 
 We consider a   complete metric space $(\mathscr{S},\mathcal{D})$ with a functional $\phi\colon \mathscr{S} \to (-\infty,+\infty]$.

{\rm(i)} A function $g\colon \mathscr{S} \to [0,\infty]$ is called a strong upper gradient for $\phi$ if for every absolutely continuous curve $ y\colon  (a,b) \to \mathscr{S}$ the function $g \circ y$ is Borel and 
$$|\phi(y(t)) - \phi(y(s))| \le \int_s^t g( y(r)) |y'|_{\mathcal{D}}(r)\,\EEE {\rm d}r \EEE \  \ \  \text{for all} \ a< s \le t < b.$$

{\rm(ii)} For each $y \in \mathscr{S}$ the local slope of $\phi$ at $y$ is defined by 
$$|\partial \phi|_{\mathcal{D}}(y): = \limsup_{z \to y} \frac{(\phi(y) - \phi(z))^+}{\mathcal{D}(y,z)}.$$

{\rm(iii)} An absolutely continuous curve $y\colon (a,b) \to \mathscr{S}$ is called a curve of maximal slope for $\phi$ with respect to the strong upper gradient $g$ if for a.e.\ $t \in (a,b)$
$$\frac{\rm d}{ {\rm d} t} \phi(y(t)) \le - \frac{1}{2}|y'|^2_{\mathcal{D}}(t) - \frac{1}{2}g^2(y(t)).$$
\end{definition}

\subsection{Curves of maximal slope as limits of time-discrete solutions}\label{sec: auxi-proofs}
 %It is often convenient to consider a weaker Hausdorff topology $\sigma$ on $\mathscr{S}$ to have more flexibility in the derivation of compactness properties (see \cite[Remark 2.0.5]{AGS}). 

We consider a sequence of complete metric spaces $(\mathscr{S}_k, \mathcal{D}_k)_k$, as well as a limiting complete metric space $(\mathscr{S}_0,\mathcal{D}_0)$. Moreover, let $(\phi_k)_k$ be a sequence of functionals with $\phi_k\colon \mathscr{S}_k \to [0,\infty]$ and $\phi_0\colon \mathscr{S}_0 \to [0,\infty]$.

We introduce time-discrete solutions for the energy $\phi_k$ and the metric $\mathcal{D}_k$ by solving suitable time-incremental minimization problems: consider a fixed time step $\tau >0$ and suppose that an initial datum $Y^0_{k,\tau}$ is given. Whenever $Y_{k,\tau}^0, \ldots, Y^{n-1}_{k,\tau}$ are known, $Y^n_{k,\tau}$ is defined as (if existent)
\begin{align}\label{eq: ds-new1}
Y_{k,\tau}^n = {\rm argmin}_{v \in \mathscr{S}_k} \mathbf{\Phi_k}(\tau,Y^{n-1}_{k,\tau}; v), \ \ \ \mathbf{\Phi_k}(\tau,u; v):=  \frac{1}{2\tau} \mathcal{D}_k(v,u)^2 + \phi_k(v). 
\end{align}
We suppose that for a choice of $\tau$ a sequence $(Y_{k,\tau}^n)_{n \in \N}$ solving \eqref{eq: ds-new1} \EEE exists. Then we define the  piecewise constant interpolation by
\begin{align}\label{eq: ds-new2}
  \bar{Y}_{k,\tau}\EEE(0) = Y^0_{k,\tau}, \ \ \ \bar{Y}_{k,\tau}(t) = Y^n_{k,\tau}  \ \text{for} \ t \in ( (n-1)\tau,n\tau], \ n\ge 1.  
\end{align}
We call  $\bar{Y}_{k,\tau}$  a \emph{time-discrete solution}. Note that the existence of such  solutions is usually guaranteed by the direct method of the calculus of variations under suitable compactness, coercivity, and lower semicontinuity assumptions.

%\BBB To this end, introduce a projection of $\mathscr{S}_k$ onto $\mathscr{S}_0$. We suppose that there exists  $\mathscr{S}$ with  $\mathscr{S} \supseteq \mathscr{S}_0$ and that \EEE for each $k \in \N$ there exists a map $\pi_k\colon \mathscr{S}_k \to \mathscr{S}$.  \BBB (Note that a usual assumption is $\mathscr{S} = \mathscr{S}_0$, see e.g.\  \cite{S2},  but for our application we need a slightly more general setting.)  \EEE

Our goal is to study the limit of time-discrete solutions as $k \to \infty$. To this end, we need to introduce a suitable topology for the convergence.  We suppose that there exists a set  $\mathscr{S}$ with  $\mathscr{S} \supseteq \mathscr{S}_0$ and a projection $\pi_k\colon \mathscr{S}_k \to \mathscr{S}$.  \BBB (Note that a usual assumption is $\mathscr{S} = \mathscr{S}_0$, see e.g.\  \cite{S2},  but for our application we need a slightly more general setting.)  \EEE

%
%We assume that there is a possibly weaker \EEE topology $\sigma$ on $\mathscr{S}$. Given a sequence $(z_k)_k$, $z_k \in \mathscr{S}_k$, and $z \in \mathscr{S}$,  we say 
%\begin{align}\label{eq sigma'}
%z_k \stackrel{\pi\sigma}{\to} z \  \ \ \ \  \text{if} \ \ \ \pi_k(z_k) \stackrel{\sigma}{\to} z. 
%\end{align}
%We suppose that the topology $\sigma$ satisfies  
%\begin{align}\label{compatibility}
%\begin{split}
%z_k \stackrel{\pi\sigma}{\to} z, &\ \  \bar{z}_k \stackrel{\pi\sigma}{\to} \bar{z}  \ \ \  \Rightarrow \ \ \ \liminf_{k \to \infty} \mathcal{D}_k(z_k,\bar{z}_k) \ge  \mathcal{D}_0(z,\bar{z})
%\end{split}
%\end{align}
%for all $z, \bar z \in \mathscr{S}_0$.

We assume that there is a possibly weaker \EEE topology $\sigma$ on $\mathscr{S}$. \BBB Given a sequence $(z_k)_k$, $z_k \in \mathscr{S}_k$, and $z \in \mathscr{S}$,  we say $z_k \stackrel{\pi\sigma}{\to} z$   if $\pi_k(z_k) \stackrel{\sigma}{\to} z$.
We suppose that the topology $\sigma$ satisfies  \EEE
\begin{align}\label{compatibility}
\begin{split}
z_k \stackrel{\pi\sigma}{\to} z, &\ \  \bar{z}_k \stackrel{\pi\sigma}{\to} \bar{z}  \ \ \  \Rightarrow \ \ \ \liminf_{k \to \infty} \mathcal{D}_k(z_k,\bar{z}_k) \ge  \mathcal{D}_0(z,\bar{z})
\end{split}
\end{align}
for all $z, \bar z \in \mathscr{S}_0$.
Moreover, we assume that for every sequence $(z_k)_k$, $z_k \in \mathscr{S}_k$, and $N \in \N$ we have
\begin{align}\label{basic assumptions2}
\begin{split}
	\phi_k(z_k) \leq N \quad \Rightarrow \ \ \ z_k \stackrel{\pi\sigma}{\to} z \in \mathscr{S}_0 \quad \text{(up to a subsequence)}.
\end{split}
\end{align}
 \BBB Further, we \EEE suppose lower semicontinuity of the energies and the slopes in the following sense: for all $z \in \mathscr{S}_0$ and $(z_k)_k$, $z_k \in \mathscr{S}_k$, we have
\begin{align}\label{eq: implication}
\begin{split}
z_k \stackrel{\pi\sigma}{\to}  z \ \ \ \ \  \Rightarrow \ \ \ \ \  \liminf_{k \to \infty} |\partial \phi_{k}|_{\mathcal{D}_{k}} (z_{k}) \ge |\partial \phi_0|_{\mathcal{D}_0} (z), \ \ \ \ \  \liminf_{k \to \infty} \phi_{k}(z_{k}) \ge \phi_0(z).
\end{split}
\end{align}

We now formulate the main convergence result of time-discrete solutions to curves of maximal slope, proved in \cite[Section 2]{Ortner}.

\begin{theorem}\label{th:abstract convergence 2}
Suppose that   \eqref{compatibility}--\eqref{eq: implication} hold. Moreover, assume that    $|\partial \phi_0|_{\mathcal{D}_0}$ is a  strong upper gradient for $ \phi_0 $.  Consider a  null sequence $(\tau_k)_k$. Let   $(Y^0_{k,\tau_k})_k$ with $Y^0_{k,\tau_k} \in \mathscr{S}_k$  and $\bar{z}_0 \in \mathscr{S}_0$ be initial data satisfying 
\begin{align*}
& \ \  Y^0_{k,\tau_k} \stackrel{\pi\sigma}{\to} \bar{z}_0 , \ \ \ \ \  \phi_k(Y^0_{k,\tau_k}) \to \phi_0(\bar{z}_0).
\end{align*}
Then, for each sequence of discrete solutions $(\bar{Y}_{k,\tau_k})_k$  starting from $(Y^0_{k,\tau_k})_k$, there exists a limiting function $z\colon [0,+\infty) \to \mathscr{S}_0$ such \BBB that, up to a  subsequence, not relabeled, \EEE
$$\bar{Y}_{k,\tau_k}(t) \stackrel{\pi\sigma}{\to} z(t), \ \ \ \ \ \phi_k(\bar{Y}_{k, \tau_k}(t)) \to \phi_0(z(t)) \ \ \  \ \ \ \ \ \forall t \ge 0$$
as $k \to \infty$, and $z$ is a curve of maximal slope for $\phi_0$ with respect to $|\partial \phi_0|_{\mathcal{D}_0}$. In particular,  $z$ satisfies the energy identity 
\begin{align}\label{maximalslope}
\frac{1}{2} \int_0^T |z'|_{\mathcal{D}_0}^2(t) \, \EEE  {\rm d}t \EEE + \frac{1}{2} \int_0^T |\partial \phi_0|_{\mathcal{D}_0}^2(z(t)) \, \EEE {\rm d}t \EEE + \phi_0(z(T)) = \phi_0(\bar{z}_0) \ \  \ \ \ \forall \, T>0. 
\end{align} 

\end{theorem}

The statement is a combination of convergence results for curves of maximal slope \cite{S2}  with their approximation by time-discrete solutions via the minimizing movement scheme. For a more detailed discussion of similar statements, we refer \EEE for example \EEE to \cite[Section 3]{MFMKDimension}.

 \subsection{Curves of maximal slope as limits of continuous solutions}\label{sec: auxi-proofs2}

%We refer to \RRR \cite[Theorem 3.6]{MFMK} \EEE for an abstract convergence result for curves of maximal slope in a setting \BBB where \EEE conditions  \eqref{compatibility}-\eqref{eq: implication} hold. 

As before, $(\mathscr{S}_k, \mathcal{D}_k)_k$ and  $(\mathscr{S}_0,\mathcal{D}_0)$ denote complete metric spaces, with \ZZZ corresponding functionals \EEE $(\phi_k)_k$ and  $\phi_0$.  \EEE For the relation of the two- and one-dimensional systems, we will use the following result.

\begin{theorem}\label{thm: sandierserfaty}
	
%Consider sets $\mathscr{S}_k$, $\mathscr{S} \supset \mathscr{S}_0$, metrics $(\mathcal{D}_n)_{n \in \N}$ and functionals $\phi_n: \mathscr{S}_k \to [0,\infty]$, $n \in \N$, as well as $\mathcal{D}_0$ and $\phi_0: \mathscr{S}_0 \to [0,\infty]$.

\EEE Suppose that   \eqref{compatibility}--\eqref{eq: implication} hold. \EEE Moreover, assume that $\vert \partial \phi_n \vert_{\mathcal{D}_n}$, $\vert \partial \phi_0 \vert_{\mathcal{D}_0}$ are strong upper gradients for $\phi_n$, $\phi_0$ with respect to $\mathcal{D}_n$, $\mathcal{D}_0$, respectively.
Let $\bar u \in \mathscr{S}_0$. For all $n \in \N$, let $u_n$ be a curve of maximal slope for $\phi_n$ with respect to $\vert \partial \phi_n \vert_{\mathcal{D}_n}$ such that
\begin{align*}
{\rm(i)} \quad & \sup\limits_{n \in \N} \sup \limits_{t \geq 0}  \, \phi_n(u_n(t))  < \infty \\
{\rm(ii)} \quad & u_n(0) \stackrel{\pi\sigma}{\to} \bar u, \quad \phi_n(u_n(0)) \to \phi_0(\bar u).
\end{align*}
Then, there exists a limiting function $u\colon [0,\infty) \to \mathscr{S}_0$ such that up to a subsequence, not relabeled,
\begin{align*}
u_n(t) \stackrel{\pi\sigma}{\to} u(t), \quad \phi_n(u_n(t)) \to \phi_0 (u(t)) \quad \forall t \geq 0
\end{align*}
as $n \to \infty$ and $u$ is a curve of maximal slope for $\phi_0$ with respect to $\vert \partial \phi_0 \vert_{\mathcal{D}_0}$.
\end{theorem}

  The result is a variant of \cite{S2} and \EEE is given in \cite[Theorem 3.6]{MFMK}, with the only difference being that here we consider a sequence of spaces instead of a fixed space. The generalization is straightforward and follows from standard adaptions.

\section{Properties of energies and dissipation distances}\label{sec:energy-dissipation}

In this section, we collect   basic   properties of the energies and dissipation distances,   and we establish properties for the local slopes. We  recall the definition of the energy and the dissipation distance in \eqref{def:vK2D}, \eqref{def:enegeryphi} and \eqref{eq: metriceps}, \eqref{def:metric}, respectively. We also recall the notation for the sublevel sets  $\mathscr{S}^{2D}_{\eps,M} = \lbrace (u,v) \in \mathscr{S}_\eps^{2D}: \phi_\eps(u,v) \le   M \rbrace$.  In what follows, we assume that $f^{2D},g^{2D}=0$ for the sake of simplicity, see \eqref{def:vK2D}. Indeed, \EEE the force terms can be included in the analysis by minor, standard modifications. \EEE

\subsection{Properties in 2D}

In this subsection, we state the relevant properties of the local slopes in the two-dimensional setting, which are provided by the following lemma.

%In this subsection, we collect the relevant properties of the 2D model which will be instrumental to \EEE use the theory in \cite{AGS}. 

\begin{lemma}[Properties of the two-dimensional \BBB setting\EEE]\label{th: metric space-lin} \label{lem: rep2d} \label{thm:stronguppergradient2d}
	Let $M>0$. We have:
	\begin{itemize}
	\item[(i)] $({\mathscr{S}}^{2D}_{\eps, M},  {\mathcal{D}}_\eps)$ is a complete metric space. \EEE
		\item[(ii)] Let $\Phi^1(t) := \sqrt{t^2 + C  t^3 + C  t^4}$ and $\Phi_{M}^2(t):= C\sqrt{ M} t^2 + C t^3 +  Ct^4$ for any $C>0$ large enough. \EEE Suppose that $(u,v) \in \mathscr{S}^{2D}_{\eps,M}$. Then, the local slope for the energy $\phi_\eps$ admits the representation
		\begin{align*}
		\vert \partial \phi_\eps \vert_{\mathcal{D}_\eps}(u,v) := \sup\limits_{\stackrel{(\tilde u, \tilde v) \in \mathscr{S}_\eps^{2D}}{(u,v)\neq(\tilde u, \tilde v)}} \frac{\big(\phi_\eps(u,v) - \phi_\eps (\tilde u , \tilde v) - \Phi_{M}^2\big(\mathcal{D}_\eps\big((u,v),(\tilde u, \tilde v )\big)\big)\big)^+}{\Phi^1\big(\mathcal{D}_\eps\big((u,v),(\tilde u, \tilde v)\big)\big)}.
		\end{align*}
		\item[(iii)] The local slope $\vert \partial \phi_\eps \vert_{\mathcal{D}_\eps}$ is a strong upper gradient for $\phi_\eps$.
	\end{itemize}
\end{lemma}

\begin{proof}
Item (i) is proved in \cite[Lemma 4.6]{MFMKDimension}. \EEE  For the \EEE proof of (ii), we refer to \cite[Lemmas 4.8 and 4.9]{MFMKDimension}. One only needs to ensure \EEE that the constant $C$ can be chosen \emph{independently of $\eps$}. The crucial step is the scaling of the embedding $W^{1,2}(S_\eps) \subset \subset L^4(S_\eps)$ on the thin domain $S_\eps$ in order to deal with the non-linearity $\nabla v \otimes \nabla v$ in the energy \eqref{def:vK2D} and the metric \eqref{eq: metriceps}. More precisely, \EEE by a scaling argument, \cite[(4.15)]{MFMKDimension} can  be replaced by  
\begin{align*}
\Vert \tfrac{1}{2}\BBB \nabla(v_0 - v_1) \otimes \nabla(v_0 - v_1)  \Vert_{L^2(S_\eps)} &\le \BBB C\Vert \nabla (v_0 - v_1)  \Vert^2_{L^4(S_\eps)} \leq C \eps^{-1/2}  \Vert v_0 - v_1 \Vert^2_{W^{2,2}(S_\eps)}\\ &\le C\sqrt{\eps} \mathcal{D}^2_\eps\big((u_0,v_0),(u_1, v_1)\big)
\end{align*}
%\RRR vielleicht koennten wir die Gleichung auch folgendermassen schreiben und der Zussamnhang zum anderen Paper ist noch klarer. Entscheide du einfach, was du besser findest. 
%\begin{align*}
%\eps^{-1}\Vert \tfrac{1}{2} \nabla(v_0 - v_1) \otimes \nabla(v_0 - v_1)  \Vert^2_{L^2(S_\eps)} &\le  C\eps^{-1}\Vert \nabla (v_0 - v_1)  \Vert^4_{L^4(S_\eps)} \leq C \eps^{-2}  \Vert v_0 - v_1 \Vert^4_{W^{2,2}(S_\eps)}\\ &\le C \mathcal{D}^4_\eps\big((u_0,v_0),(u_1, v_1)\big)
%\end{align*}
\EEE 	for $(u_0,v_0), (u_1,v_1) \in \BBB \mathscr{S}^{2D}_{\eps,M}\EEE$, where the last inequality follows from \eqref{eq: metriceps}, \BBB Poincar\'e's inequality, \EEE and the \EEE positivity \EEE of $Q^2_R$. This  \EEE is sufficient to adapt the proof of \cite[Lemma 4.8]{MFMKDimension}.  The mappings $\Phi^1$ and $\Phi^2_M$ have been introduced after \cite[(4.19)]{MFMKDimension} and before  \cite[(4.21)]{MFMKDimension}, respectively. Item \EEE (iii) is also a consequence of \cite[Lemma 4.9]{MFMKDimension}.
\end{proof}

\subsection{Properties in 1D}\label{sec:prop1d}
\BBB
In this subsection, we derive properties in the one-dimensional setting. We mainly need a one-dimensional version of Lemma \ref{lem: rep2d}  \EEE as well as some basic properties of the metric space $(\mathscr{S}^{1D},\mathcal{D}_0)$ and the energy $\phi_0$. \EEE 	\ZZZ The proofs are analogous to the corresponding  proofs in \cite{MFMK, MFMKDimension}. We refer to  Appendix~\ref{sec:Appendix} where we give the main arguments for the reader's convenience.\EEE
%
%the facts that the one-dimensional local slope $\vert \partial \phi_0 \vert_{\mathcal{D}_0}$ is a strong upper gradient (see Theorems \ref{th:abstract convergence 2} and \ref{thm: sandierserfaty}). 
\begin{lemma}[Properties of $(\mathscr{S}^{1D},\mathcal{D}_0)$ and $\phi_0$]\label{lem:complete1d}
Consider the canonical norm $\Vert \cdot \Vert_{can}$ on $\mathscr{S}^{1D}$ which is defined as  
\begin{align*}
\Vert (y,w,\theta) \Vert_{can}:= \Vert y \Vert_{W^{1,2}(S;\R^2)} + \Vert w \Vert_{W^{2,2}(I)} + \Vert \theta \Vert_{W^{1,2}(I)}.
\end{align*} 
\begin{itemize}
\item[(i)] {\rm Completeness:} $(\mathscr{S}^{1D},\mathcal{D}_0)$ is a complete metric space.
\item[(ii)]  {\rm Topology:} The topology induced by $\Vert \cdot \Vert_{can}$ coincides with the topology induced by $\mathcal{D}_0$. In particular, there \EEE exists \EEE a constant \EEE $C > 0$ \EEE such that   
\begin{align}
&  \Vert{w- \tilde w}\Vert_{W^{2,2}(I)} + \Vert{\theta- \tilde \theta}\Vert_{W^{1,2}(I)} \le C \mathcal{D}_0((y,w,\theta),(\tilde y,\tilde w,\tilde\theta))  \label{ineq:lowerbound} \quad \text{and}\\
&\Vert y-\tilde y \Vert_{W^{1,2}(S;\R^2)} \leq C \mathcal{D}_0((y,w, \theta),(\tilde y, \tilde w, \tilde \theta)) + C\Vert w^\prime +  \tilde w^\prime \Vert_{L^4(I)} \Vert w^\prime  -  \tilde w^\prime \Vert_{L^4(I)}  \label{ineq:lowerboundy}
\end{align}
for $(y,w,\theta), (\tilde y, \tilde w, \tilde \theta) \in \mathscr{S}^{1D}$. 
\item[(iii)] {\rm Compactness: }Let $(y_k,w_k,\theta_k)_k$ be a sequence in $\mathscr{S}^{1D}$ with $\sup_{k\in \N} \BBB \phi_0(y_k,w_k,\theta_k)  \EEE < \infty$. Then,  $\sup_{k\in \N} \Vert(y_k,w_k,\theta_k)\Vert_{can} <+\infty$ and, up to a subsequence, $(y_k,w_k,\theta_k)_{k}$ \EEE  converges weakly   in $(\mathscr{S}^{1D}, \Vert \cdot \Vert_{can})$ \EEE to some $(y,w,\theta) \in \mathscr{S}^{1D}$.
\item[(iv)] {\rm Lower semicontinuity: }If $(y_k,w_k,\theta_k)$ and $(\tilde y_k,\tilde w_k,\tilde \theta_k)$ converge weakly in $(\mathscr{S}^{1D}, \Vert \cdot \Vert_{can})$  to $(y,w,\theta)$ and $(\tilde y, \tilde w, \tilde \theta)$, respectively, we have $\liminf_{k \to \infty} \phi_0(y_k,w_k,\theta_k) \geq \phi_0(y,w,\theta)$ and $\liminf_{k \to \infty} \mathcal{D}_0((y_k,w_k,\theta_k),(\tilde y_k,\tilde w_k,\tilde \theta_k)) \geq \mathcal{D}_0((y,w,\theta),(\tilde y, \tilde w, \tilde \theta)).$ \EEE
\end{itemize}
\end{lemma}

%
%The following theorem provides the fact that the one-dimensional slope is a strong upper gradient.

The following \EEE lemma \EEE is the \EEE one-dimensional version of Lemma \ref{lem: rep2d}.
\begin{lemma}[Properties of the one-dimensional slope $\vert \partial{\phi}_0\vert_{\mathcal{D}_0}$]\label{thm:stronguppergradient1d} \label{Thm:representationlocalslope1}
	Let $M>0$. We have:
	\begin{itemize} 
		\item[(i)] The local slope for the energy $\phi_0$ admits the representation
		\begin{align*}
		\vert \partial \phi_0 \vert_{\mathcal{D}_0}(y,w,\theta) := \sup\limits_{(y,w,\theta)\neq(\tilde y, \tilde w, \tilde \theta) \in \mathscr{S}^{1D}} \frac{\big(\phi_0(y,w,\theta) - \phi_0 (\tilde y , \tilde w, \tilde \theta) - \Phi_M^2(\mathcal{D}_0((y,w,\theta),(\tilde y, \tilde w, \tilde \theta))\EEE ) \EEE\big)^+}{\Phi^1(\mathcal{D}_0((y,w,\theta),(\tilde y, \tilde w, \tilde \theta))\EEE )\EEE}
		\end{align*}
		for all $(y,w,\theta)\in \mathscr{S}^{1D}$  satisfying $\phi_0(y,w,\theta) \leq M$, where $\Phi^1$ and $\Phi^2_M$ are defined \BBB in Lemma~\ref{lem: rep2d}(ii).
			\EEE
		\item[(ii)] 	
		The local slope $\vert \partial \phi_0 \vert_{\mathcal{D}_0}$ is a strong upper gradient for $\phi_0$.
		\item[(iii)] {\rm Lower semicontinuity:} If $(y_k,w_k,\theta_k)_k\EEE \subset \EEE \mathscr{S}^{1D}$ converges weakly in $(\mathscr{S}^{1D}, \Vert \cdot \Vert_{can})$ to $(y,w,\theta)\in \mathscr{S}^{1D}$, we have $\liminf_{k \to \infty} \vert \partial \phi_0\vert_{\mathcal{D}_0}(y_k,w_k,\theta_k) \geq \vert \partial \phi_0\vert_{\mathcal{D}_0}(y,w,\theta)$. \EEE
	\end{itemize}
\end{lemma}

\section{Relation between 2D and 1D setting}\label{Sec:relation2D1D}

In this section, we briefly recall the convergence results in the static case \cite{Freddi2018} and then \EEE we prove lower semicontinuity for the local slopes along the passage from the 2D to the 1D setting. \EEE Recall \EEE the projection mapping $\pi_\eps$, the space $\mathscr{S}= \pi_\eps(\mathscr{S}^{2D}_\eps)$, and the convergences $\pi\sigma$ and $\pi\rho$ introduced before Theorem \ref{thm:relation2d1d}. \EEE In the sequel, it is convenient to express $\phi_\eps$ and $\mathcal{D}_\eps$ in terms of the scaled functions $y$ and $w$ introduced in \eqref{scaledfunctions}. By a change of variables we have (recall $f^{2D}=g^{2D}=0$) \EEE 
\begin{align}\label{def:vK2D-scaled}
\phi_\eps(y,w)=& \frac{1}{2} \int_S Q_W^2\big( E^\eps y + \frac{1}{2} \nabla_\eps w \otimes \nabla_\eps w \big)  + \frac{1}{24} \int_S Q_W^2(\nabla_\eps^2 w) 
\end{align}
and
\begin{align}\label{eq: metriceps-scaled}
	\mathcal{D}_\eps^2((y,w),(\tilde y, \tilde w))= & \int_S Q_R^2\big( E^\eps y -  E^\eps \tilde y + \frac{1}{2} \nabla_\eps w \otimes \nabla_\eps w - \frac{1}{2} \nabla_\eps \tilde w \otimes \nabla_\eps \tilde w \big)  + \frac{1}{12} \int_S Q_R^2(\nabla_\eps^2 w - \nabla_\eps^2 \tilde w) 
\end{align}
for all \EEE $(y,w), (\tilde y, \tilde w)\ZZZ \in \mathscr{S} \EEE$. \EEE

\subsection{$\Gamma$-convergence}\label{sec: gammastatic}
We briefly recall the $\Gamma$-convergence result in \cite{Freddi2018} which \BBB particularly \EEE yields the lower semicontinuity of the energies and the dissipation.

\begin{theorem}[$\Gamma$-convergence of energies]\label{th: Gamma}
$\phi_\eps$ converges to $\phi_0$ in the sense of \ZZZ$\,\Gamma$\EEE-convergence. More precisely,

\noindent {\rm (i) (Lower bound)} For all $(y,w,\theta) \in {\mathscr{S}^{1D}}$ and all sequences $(u_\eps, v_\eps)_\eps$\BBB , $(u_\eps,v_\eps) \in \mathscr{S}^{2D}_\eps$, \EEE such that  $(u_\eps, v_\eps) \stackrel{\pi\sigma}{\to} (y,w,\theta)$ we find
$$\liminf_{\eps \to 0}\phi_\eps(u_\eps, v_\eps) \ge {\phi}_0(y,w,\theta). $$

\noindent {\rm (ii) (Optimality of lower bound)} For all $(y,w,\theta) \in {\mathscr{S}^{1D}}$ there exists a sequence $(u_\eps, v_\eps)_\eps$ \ZZZ with $(u_\eps, v_\eps) \in \BBB\mathscr{S}^{2D}_\eps\EEE$  for all $\eps > 0$ \EEE  such that  $(u_\eps, v_\eps) \stackrel{\pi\rho}{\to} (y,w,\theta)$ and  
$$\lim_{\eps \to 0}\phi_\eps(u_\eps, v_\eps) = {\phi}_0(y,w,\theta).$$
\end{theorem}
In \cite[Theorem 2.3(i),(ii)]{Freddi2018}, the proof has been given for functions   with vanishing mean. The adaptions to the present setting, however, are minor. In particular,  the construction of recovery sequences needs to be slightly adjusted to comply with the imposed boundary conditions. We defer details to  Lemma \ref{lem:mutual} and Remark \ref{rem: recov}\ZZZ\textbf{(a)}\,\EEE below where we use a similar ansatz \EEE as in \EEE \cite{Freddi2018}. \EEE

% we refer to .   minor adaptions, one can fit the proof to our setting. Especially, the recovery sequence for the in-plane displacement $(u_\eps)_\eps$ is slightly different as we add lower-order-terms to $(u_\eps)_1$ instead of $(u_\eps)_2$. Let us just mention that we use the same ansatz of the recovery sequence in

\begin{theorem}[Lower semicontinuity of dissipation distances]\label{th: lscD}
Let $M>0$. Then, for sequences $(u_\eps, v_\eps)_\eps$ and  $(\tilde u_\eps, \tilde v_\eps)_\eps$, $(u_\eps, v_\eps),(\tilde u_\eps, \tilde v_\eps) \in \mathscr{S}^{2D}_{\eps,M}$, with  $(u_\eps, v_\eps) \stackrel{\pi\sigma}{\to} (y,w,\theta)$ and $(\tilde u_\eps, \tilde v_\eps) \stackrel{\pi\sigma}{\to} (\tilde y, \tilde w, \tilde \theta)$  we have
$$\liminf_{\eps \to 0} \mathcal{D}_\eps((u_\eps, v_\eps),(\tilde u_\eps, \tilde v_\eps) ) \ge {\mathcal{D}_0} \big ((y,w,\theta),(\tilde y , \tilde w, \tilde \theta) \big).$$
\end{theorem}

\begin{proof}
The argument is analogous to the one in Theorem \ref{th: Gamma}(i) (cf.~\cite[Theorem 2.3(i)]{Freddi2018}) \BBB as the structure of $\mathcal{D}_\eps$ and $\mathcal{D}$ is similar to $\phi_\eps$ and $\phi$, see \eqref{def:metric},  \eqref{def:enegeryphi},  \eqref{def:vK2D-scaled},  and \eqref{eq: metriceps-scaled}.   \EEE
\end{proof}

\subsection{Lower semicontinuity of slopes}\label{sec: gamma}
In contrast to \EEE Theorems \EEE \ref{th: Gamma} and \ref{th: lscD}, the derivation of lower semicontinuity for the local slopes is more challenging as \BBB in its definitions \EEE metrics appear in the denominator and energies are subtracted  in the enumerator. Roughly speaking, \EEE we need a reverse inequality in Theorem \ref{th: Gamma}(i) and \EEE Theorem \ref{th: lscD} which in general is false. The representations of the slopes in Lemmas \ref{th: metric space-lin} and \ref{Thm:representationlocalslope1} allow  us to construct recovery sequences \EEE such that the ``reverse inequality'' is true. More precisely, we have the following.

\begin{lemma}[Mutual recovery sequence]\label{lem:mutual}
	Suppose that \eqref{quadraticforms} or \eqref{eq:quadraticform:minimumnot0} holds. \ZZZ
	Consider a sequence $\ZZZ ( u_\eps, v_\eps)_\eps, ( u_\eps, v_\eps)  \in \mathscr{S}^{2D}_{\eps,M}$, such that $(u_\eps, v_\eps) \stackrel{\pi\sigma}{\to}  (y,w,\theta)$. 	Let $(\tilde y, \tilde w, \tilde \theta) \in \mathscr{S}^{1D}$. Then, there exists a mutual recovery sequence $(\tilde u_\eps, \tilde v_\eps)_\eps$, \MMM $(\tilde u_\eps, \tilde v_\eps) \in \mathscr{S}^{2D}_{\eps}$ for all $\eps >0$, \EEE such that
	\begin{align}
	\liminf\limits_{\eps \to 0} \BBB \big( \EEE \phi_\eps(u_\eps,v_\eps) - \phi_\eps(\tilde u_\eps, \tilde v_\eps) \big)\geq \phi_0(y,w,\theta) - \phi_0 (\tilde y, \tilde w, \tilde \theta) \label{mutualenergy}
	\end{align}
	and
	\begin{align}
	\lim\limits_{\eps \to 0} \mathcal{D}_\eps((u_\eps,v_\eps),(\tilde u_\eps, \tilde v_\eps)) \BBB = \EEE \mathcal{D}_0((y,w,\theta),(\tilde y, \tilde w, \tilde \theta)).\label{mutualmetric}
	\end{align}
\end{lemma}

\begin{rem}
{\normalfont
 In the following and in later proofs, \EEE we will frequently use the elementary expansion
\begin{align}
Q(a)-Q(b) = Q(a-b) + 2 \C[a-b,b]\label{expansion}
\end{align}
for all $a,b$, where $Q$ \EEE is \EEE a quadratic form with associated bilinear form $\C$.
}
\end{rem}

%Here, it is fundamental to consider a compatibility condition of the quadratic forms. As mentioned in Theorem \ref{th: Gamma}, the construction of the mutual recovery sequence is more simple as we assume that $\tilde Q_W^2$ and $\tilde Q_R^2$ attain its minima at $0$. The proof also works for materials with vanishing dissipation effect into the $e_2$ direction. The recovery sequence can be constructed as in Theorem \ref{th: Gamma} (ii) and the assumption for $Q_R^2$ allows us to neglect coordinates that are linked to the $e_2$ direction.

\begin{proof}[Proof of Lemma \ref{lem:mutual}] 
The proof is based on the ansatz of the recovery sequence in \cite{Freddi2018}, slightly modified to comply with the imposed boundary conditions. Recall the definition of the quadratic forms in \eqref{quadraticform1d}--\eqref{quadraticform0d}. \EEE We first suppose that \eqref{eq:quadraticform:minimumnot0} holds as this case is more delicate. At the end of the proof, we briefly present the adaptions for \eqref{quadraticforms}.

\noindent \textit{Step 1: Definition of recovery sequences.} \BBB Given $(u_\eps,v_\eps)$, let $(y_\eps,w_\eps)$ as in \eqref{scaledfunctions}. \EEE Let $\gamma$, $E_{12}$, and $E_{22}$ be the functions given by Proposition~\ref{thm:compactness} such that $\tfrac{1}{\eps^2} \partial_{22} w_\eps \rightharpoonup \gamma$, $(E^\eps y_\eps)_{12} \rightharpoonup E_{12}$ and $(E^\eps y_\eps)_{22} \rightharpoonup E_{22}$ in $L^2(S)$. 
We let \EEE $\tilde \gamma\colon I \to \R$ and $z\colon S \to \R$ be functions \EEE such that 
\begin{align}\label{eq: Wredu}
Q_W^2(\tilde w'',\tilde \theta',\tilde \gamma) = Q_W^1(\tilde w'', \tilde \theta') \quad \quad \text{ and } \quad \quad Q_W^0(\partial_1 \tilde y_1+ \tfrac{1}{2} \tilde w''^2) = Q_W^2(\partial_1 \tilde y_1+ \tfrac{1}{2} \tilde w''^2, 0, z).
\end{align}
As $Q_W^2$ is positive definite on $\R^{2\times 2}_{\rm sym}$ and $(\tilde y, \tilde w, \tilde \theta) \in \mathscr{S}^{1D}$, \EEE we find that $\tilde \gamma \in L^2(I)$ and $z \in L^2(S)$. \EEE By standard density arguments for $L^2$- and $W^{1,2}$-spaces, there exist functions $\theta_\eps \in C_c^\infty(I)$ and $\ZZZ \tilde\gamma_\eps, \EEE \zeta_\eps \in C_c^\infty(S)$ such that $\theta_\eps \to \tilde \theta - \theta$ in $W^{1,2}(I)$, $\MMM \eps \theta'_\eps, \eps\theta''_\eps \EEE \to 0$ in $L^2(I)$, $\tilde \gamma_\eps \to \tilde \gamma - \gamma$, \ZZZ $\eps  \partial_1 \tilde\gamma_\eps \to 0$, $\eps^2 \RRR \partial^2_{11} \EEE\tilde \gamma_\eps \to 0$ \EEE in $L^2(S)$, and $\zeta_\eps \to z- E_{22}$, $\eps  \partial_1\zeta_\eps \EEE \to 0$ in $L^2(S)$. We define the recovery sequence for the vertical displacement $\tilde w_\eps$ \EEE as
	\begin{align}\label{eq: rec1}
	\tilde w_\eps (x) = w_\eps (x) + \tilde w(x_1) - w(x_1) + \eps x_2 \theta_\eps (x_1) + \ZZZ \eps^2 \int_{-1/2}^{x_2} \int_{-1/2}^t \tilde \gamma_\eps(x_1,s) \, {\rm d}s \, {\rm d}t\EEE.
	\end{align}
\ZZZ By recalling \eqref{scaledoperators}, we compute \EEE
\begin{align*}
\nabla_\eps w_\eps(x) = \begin{pmatrix}
\partial_1 w_\eps(x) + \tilde w' (x_1) - w'(x_1) + \eps x_2 \theta_\eps' (x_1) + \ZZZ \eps^2 \int_{-1/2}^{x_2} \int_{-1/2}^t \partial_1 \tilde \gamma_\eps(x_1,s) \, {\rm d}s \, {\rm d}t\EEE  \\
	\EEE \tfrac{1}{\eps} \partial_2 w_\eps (x) + \theta_\eps (x_1)   + \ZZZ \eps \int_{-1/2}^{x_2} \tilde \gamma_\eps(x_1,s) \, {\rm d}s \EEE
\end{pmatrix}
\end{align*}
%	\begin{align*}
%	 \partial_1 w_\eps (x) =& \partial_1 w_\eps(x) + \tilde w' (x_1) - w'(x_1) + \eps x_2 \theta_\eps' (x_1) + \ZZZ \eps^2 \int_{-1/2}^{x_2} \int_{-1/2}^t \partial_1 \tilde \gamma_\eps(x_1,s) \, {\rm d}s \, {\rm d}t\EEE,  \\
%	\tfrac{1}{\eps} \partial_2 w_\eps(x) =&
%	\EEE \tfrac{1}{\eps} \partial_2 w_\eps (x) + \theta_\eps (x_1)   + \ZZZ \eps \int_{-1/2}^{x_2} \tilde \gamma_\eps(x_1,s) \, {\rm d}s \EEE
%	\end{align*}
	and further see \ZZZ that the entries of the scaled Hessian $\grad_\eps^2 \tilde w_\eps(x)$ are given by \EEE
	\begin{align*}
	\partial^2_{11} w_\eps(x) &=\textstyle    \partial_{11} w_\eps(x) + \tilde w''(x_1) - w''(x_1)  + \eps x_2 \theta_\eps'' (x_1) + \ZZZ \eps^2 \int_{-1/2}^{x_2} \int_{-1/2}^t \partial^2_{11} \tilde \gamma_\eps(x_1,s) \, {\rm d}s \, {\rm d}t\EEE \\
	\tfrac{1}{\eps}\partial^2_{12} w_\eps(x) &= \textstyle \tfrac{1}{\eps} \partial_{12} w_\eps(x) + \theta_\eps'(x_1) + \ZZZ \eps \int_{-1/2}^{x_2} \partial_1 \tilde \gamma_\eps(x_1,s) \, {\rm d}s \EEE\\
	\tfrac{1}{\eps^2}\partial^2_{22} w_\eps(x) &= \textstyle\tfrac{1}{\eps^2}\partial_{22} w_\eps(x) +\ZZZ \tilde \gamma_\eps(x_1,x_2)\EEE
	\end{align*}
%	By recalling the scaled gradients in \eqref{scaledoperators}, we \EEE compute 
%	\begin{align*}
%	\grad_\eps \tilde w_\eps (x) = \Big(\partial_1 w_\eps(x) + \tilde w' (x_1) - w'(x_1) + \eps x_2 \theta_\eps' (x_1) + \tfrac{\eps^2}{2} x_2^2 \tilde{\gamma}_\eps'(x_1),  \EEE \tfrac{1}{\eps} \partial_2 w_\eps (x) + \theta_\eps (x_1)   +\eps x_2 \tilde \gamma_\eps(x_1) \Big)
%	\end{align*}
%	and further see that $\grad_\eps^2 \tilde w_\eps(x)$ is given by \EEE
%	\begin{align*}
%	 \begin{pmatrix}\partial_{11} w_\eps(x) + \tilde w''(x_1) - w''(x_1)  + \eps x_2 \theta_\eps'' (x_1) + \tfrac{\eps^2}{2} x_2^2 \tilde{\gamma}_\eps''(x_1) \EEE  & \frac{1}{\eps} \partial_{12} w_\eps(x) + \theta_\eps'(x_1)  +\eps x_2 \tilde \gamma_\eps'(x_1)  \\ \frac{1}{\eps} \partial_{12} w_\eps(x) + \theta_\eps'(x_1) + \eps x_2 \tilde \gamma_\eps'(x_1)  & \frac{1}{\eps^2}\partial_{22} w_\eps(x) + \tilde \gamma_\eps(x_1) \\
%	\end{pmatrix}.
%	\end{align*}
 Let us now address the in-plane displacements. \EEE  We \EEE define
	\begin{align}\label{eq: rec2}
	({\bar{y}_\eps})_1(x_1,x_2):= \, &  \tilde y_1(x)- y_1(x) \EEE  + \eps x_2  \big(  w'(x_1)\theta(x_1)-  \tilde w'(x_1) \tilde \theta(x_1)\big),\notag\\
	({\bar{y}_\eps})_2(x_1,x_2):= \, & \textstyle \tilde y_2(x)- y_2(x) \EEE
	- \frac{\eps^2}{2} x_2  \tilde \theta^2(x_1)  + \eps^2  \int_{-1/2}^{x_2}  \EEE \zeta_\eps (x_1,s) \, {\rm d}s.
	\end{align}
	The characterization in (\ref{def:BN})  implies that $\partial_2 y_1 = -\partial_1 y_2$, $\partial_2 \tilde{y}_1 = -\partial_1 \tilde{y}_2$, and $\partial_2 y_2 = \partial_2 \tilde{y}_2=0$. \EEE
	Then, by the definition of $E^\eps$ we have
	\begin{align*}
	(E^\eps \bar y_\eps)_{11} &= \textstyle \partial_1 \tilde y_1 (x) \EEE - \partial_1 y_1   (x) \EEE + \eps x_2\big( w'(x_1)\theta'(x_1) + w''(x_1)\theta(x_1) - \tilde w'(x_1) \tilde \theta'(x_1) - \tilde w''(x_1) \tilde \theta(x_1)\big),\\
	(E^\eps \bar y_\eps)_{12} &=  \textstyle \frac{1}{2}\big( w'(x_1)\theta(x_1)-\tilde w'(x_1) \tilde \theta(x_1) \big) - \frac{\eps x_2 }{2} \tilde \theta(x_1) \tilde \theta'(x_1)  + \frac{\eps }{2}  \int_{ -1/2}^{x_2} \EEE \partial_1 \zeta_\eps (x_1,s) \, {\rm d}s,\\
	(E^\eps \bar y_\eps)_{22} &= \textstyle - \frac{1}{2} \tilde \theta^2(x_1)  + \zeta_\eps(x_1,x_2).
	\end{align*}
The \EEE construction ensures that $\bar y_\eps \in W^{1,2}(S;\R^2)$.  We let \EEE $\tilde y_\eps := y_\eps + \bar y_\eps$. 	Eventually, we define \EEE $ \tilde u_\eps$ and $\tilde v_\eps$ such that  $(\tilde y_\eps)_1(x_1,x_2)= (\tilde u_\eps)_1(x_1, \eps x_2)$, $(\tilde y_\eps)_2(x_1,x_2)= \eps (\tilde u_\eps)_2(x_1,\eps x_2)$, and $\tilde w_\eps(x_1,x_2)= \tilde v_\eps(x_1, \eps x_2)$, \ZZZ see~\eqref{scaledfunctions}. \EEE  As $\theta$, $\tilde \theta$, $\theta_\eps$,  $\tilde{y} - y$, and $\tilde{w} - w$ \EEE   vanish on $\partial I$,  and \ZZZ  $ \gamma_\eps$, $\zeta_\eps$ vanish on $\partial S$, \BBB an \EEE inspection of \eqref{eq: rec1}--\eqref{eq: rec2} shows that \EEE  $ (\tilde{u}_\eps,\tilde{v}_\eps) \EEE \in \mathscr{S}_\eps^{2D}$.\\
		\noindent \textit{Step 2: Proof of \eqref{mutualmetric}.} \EEE
Due to Proposition \ref{thm:compactness} and the compact embedding $W^{1,2}(S) \subset \subset L^4(S)$, we find that $\grad_\eps w_\eps  \to ( w', \theta)$ in $L^4(S;\R^2)$ and  hence 
\begin{align}\label{eq: stron-grad}
\grad_\eps w_\eps \otimes \grad_\eps w_\eps \to (w', \theta) \otimes ( w',  \theta) \quad \text{in} \quad L^2(S;\R^{2\times 2}).
\end{align}
Similarly, due to the fact that $\partial_1 w_\eps \to w'$ in \EEE $L^4(S)$, \EEE  $\tfrac{1}{\eps} \partial_2 w_\eps \to  \theta$ in $L^4 (S)$, and $\theta_\eps \to \tilde \theta - \theta$ in $W^{1,2}(I)$, \EEE we have that $\grad_\eps \tilde w_\eps  \to ( \tilde w', \tilde \theta)$ in $L^4(S;\R^2)$, and thus 
\begin{align*}
\grad_\eps \tilde w_\eps \otimes \grad_\eps \tilde w_\eps \to (\tilde w', \tilde \theta) \otimes (\tilde w', \tilde \theta) \quad \text{in} \quad L^2(S;\R^{2\times 2}).
\end{align*}
This along with $\zeta_\eps \to z- E_{22}$ in $L^2(S)$ and $\tilde y_\eps - y_\eps = \bar{y}_\eps$ \EEE implies by an elementary computation\EEE
	\begin{align}
	&\quad \, \, E^\eps (\tilde y_\eps - y_\eps) + \frac{1}{2} \grad_\eps \tilde w_\eps \otimes \grad_\eps \tilde w_\eps - \frac{1}{2} \grad_\eps  w_\eps \otimes \grad_\eps  w_\eps \nonumber \\ \longrightarrow & \quad \begin{pmatrix} \partial_1 \tilde y_1  - \partial_1 y_1 +\frac{1}{2} (\tilde w'^2 - w'^2)  &  0\\ 0&  z  - E_{22} - \tfrac{1}{2}  \theta^2\EEE \\
	\end{pmatrix} \label{mutual:firstintegralminnot0}
	\end{align}
strongly in $L^2(S;\R^{2\times2})$.	Moreover,  by $\theta_\eps \to \tilde \theta - \theta$ in $W^{1,2}(I)$ and $\tilde \gamma_\eps \to \tilde \gamma - \gamma$ in $L^2(S)$, \BBB \EEE we have
	\begin{align} 
	\grad_\eps^2 (\tilde w_\eps - w_\eps)  &\to \begin{pmatrix} \tilde w'' - w''  & \tilde \theta' - \theta'\\ \tilde \theta' - \theta' & \tilde \gamma - \gamma\\
	\end{pmatrix}  \label{mutual:secondintegralnotmin0}
	\end{align}
	strongly in $L^2(S;\R^{2\times2})$.  Proposition \ref{thm:compactness}  and \eqref{eq: stron-grad} also yield \EEE
	\begin{align}\label{mutual:thirdintegralnotmin0}
	E^\eps y_\eps+ \frac{1}{2} \grad_\eps  w_\eps \otimes \grad_\eps  w_\eps \rightharpoonup \begin{pmatrix} \partial_1y_1 + \frac{1}{2} w'^2  &  E_{12} + w'\theta \\ E_{12} +  w'\theta & E_{22} + \frac{1}{2} \theta^2\\
	\end{pmatrix} \quad \text{ and } \quad 
	\grad_\eps^2  w_\eps \rightharpoonup \begin{pmatrix}  w''  &  \theta'\\ \theta' & \gamma \\
	\end{pmatrix} 
	\end{align}
	weakly in $L^2(S;\R^{2\times2})$. Then, by  \eqref{mutual:firstintegralminnot0}, \eqref{mutual:secondintegralnotmin0}, and \eqref{eq:quadraticform:minimumnot0}  (including explicitly the $\eps$-dependence of $ Q_{R,\eps}^2$) \EEE we find that
	\begin{align*}
& \int_S  Q_{R,\eps}^2 \EEE \Big( E^\eps (\tilde y_\eps - y_\eps) + \tfrac{1}{2} \grad_\eps \tilde w_\eps \otimes \grad_\eps \tilde w_\eps - \tfrac{1}{2} \grad_\eps w_\eps \otimes \grad_\eps w_\eps \Big)+ \frac{1}{12} \int_S   Q_{R,\eps}^2 \EEE \Big(\grad_\eps^2 (\tilde w_\eps - w_\eps) \Big) \\
	\longrightarrow \quad & \int_S  Q_{R}^0 \Big( \partial_1 \tilde y_1 - \partial_1 y_1 + \frac{1}{2} (\tilde w'^2 - w'^2 ) \Big) + \frac{1}{12} \int_I Q_{R}^1  \big( \EEE  \tilde w'' - w'' , \tilde \theta' - \theta'  \big) .
	\end{align*}
	This shows \eqref{mutualmetric}.\\	
\noindent \textit{Step 3: Proof of  \eqref{mutualenergy}}.  We \EEE	additionally use \eqref{expansion} and the fact that a product of sequences converges weakly in $L^1$ if one factor converges weakly and the other one strongly: by \eqref{mutual:secondintegralnotmin0}, \eqref{mutual:thirdintegralnotmin0} \EEE we \EEE obtain
	\begin{align}\label{eq: nochmal1}
	& \int_S Q_{W}^2(\grad_\eps^2 w_\eps)   - \int_S Q_{W}^2(\grad_\eps^2 \tilde w_\eps) \notag\\
	=& \int_S Q_W^2 (\grad_\eps^2 w_\eps - \grad_\eps^2 \tilde w_\eps) + 2 \C_W^2 [\grad_\eps^2 w_\eps - \grad_\eps^2 \tilde w_\eps, \grad_\eps^2 \tilde w_\eps] \notag\\
	\ZZZ \longrightarrow \quad\quad \EEE  & \int_S Q_W^2 \bigg(\hspace{-3pt}\begin{pmatrix}  w'' - \tilde w''  &  \theta' - \tilde\theta'\\ \theta' - \tilde  \theta' & \gamma - \tilde \gamma \\
	\end{pmatrix} \hspace{-3pt} \bigg)  +  2\C_W^2 \bigg[\begin{pmatrix}  w'' - \tilde w''  &  \theta' - \tilde\theta'\\ \theta' - \tilde  \theta' & \gamma - \tilde \gamma \\
	\end{pmatrix}, \begin{pmatrix}  \tilde w''  &   \tilde\theta'\\  \tilde  \theta' & \tilde \gamma \\
	\end{pmatrix} \bigg] \notag \\
	 = & L_1:=   \int_S Q_W^2\bigg(\hspace{-3pt}\begin{pmatrix}  w''   &  \theta' \\ \theta'  & \gamma  \\
	\end{pmatrix} \hspace{-3pt}\bigg)  - \int_S Q_W^2\bigg(\hspace{-3pt}\begin{pmatrix}  \tilde w''   &  \tilde \theta' \\ \tilde \theta'  &  \tilde \gamma \\
	\end{pmatrix} \hspace{-3pt}\bigg) .
	\end{align}
	By \eqref{quadraticform1d} and  \eqref{eq: Wredu}  this implies \EEE
	\begin{align}\label{eq: first-stepi} 
	 L_1 \EEE  \geq  \int_I Q_W^1\big( ( w''   , \theta') \big)  - \int_I Q_W^1\big((  \tilde w''   ,  \tilde \theta' ) \big) . 
	\end{align}
	Similarly, due to \eqref{expansion}, \eqref{mutual:firstintegralminnot0}, and \eqref{mutual:thirdintegralnotmin0}, \EEE we have
	\begin{align}\label{eq: nochmal2}
	& \int_S Q_{W}^2\big( E^\eps y_\eps + \frac{1}{2} \grad_\eps w_\eps \otimes \grad_\eps w_\eps \big)  -  \int_S Q_{W}^2\big( E^\eps \tilde y_\eps +  \grad_\eps \tilde w_\eps \otimes \grad_\eps \tilde w_\eps \big)   \\
	\longrightarrow\quad& L_2:= \EEE \int_S Q_W^2 \bigg(\hspace{-3pt}\begin{pmatrix} \partial_1y_1 + \frac{1}{2} w'^2  &  E_{12} + w'\theta \\ E_{12} +  w'\theta & E_{22} + \frac{1}{2} \theta^2\\
	\end{pmatrix} \hspace{-3pt} \bigg)         - \int_S Q_W^2 \bigg( \hspace{-3pt}\begin{pmatrix} \partial_1 \tilde y_1  +\frac{1}{2} \tilde w'^2   &  E_{12} + w'\theta\\ E_{12} + w'\theta&  z \\
	\end{pmatrix} \hspace{-3pt} \bigg)     . \notag
	\end{align}
	By \eqref{eq:quadraticform:minimumnot0} we have $\argmin_{q_{12}} Q_W^2(q_{11},q_{12}, q_{22}) = 0$ and thus $Q_W^2(q_{11},q_{12}, q_{22}) = Q_W^2(q_{11},0, q_{22}) + \ZZZ bq_{12}^2$ for some \ZZZ $b>0$. \EEE This along with \eqref{quadraticform1d} and  \eqref{eq: Wredu} shows  
	\begin{align}\label{eq: relaxationL2}{L_2 \ge  \int_S \EEE Q_W^0\EEE \big(\partial_1 y_1 + \tfrac{1}{2}\vert w^{\prime}\vert^2\big) -\int_S \EEE Q_W^0 \EEE \big(\partial_1 \tilde{y}_1 + \tfrac{1}{2}\vert \tilde{w}^{\prime}\vert^2\big)} \end{align}
which in combination with \eqref{eq: first-stepi} concludes the proof of  \EEE \eqref{mutualenergy}.\\ \noindent \textit{Step 4: Adaptions for \eqref{quadraticforms}.}   \EEE We now suppose that \eqref{quadraticforms} holds. We redefine $\tilde{\gamma}$ and $z$ differently compared to \eqref{eq: Wredu}: let $\tilde \gamma = \gamma$ and $z =\theta^2/2 + E_{22}$. Then, \EEE \eqref{mutual:firstintegralminnot0}--\eqref{mutual:secondintegralnotmin0}  imply
\begin{align*}
& \int_S Q_{R}^2 \EEE \Big( E^\eps (\tilde y_\eps - y_\eps) + \tfrac{1}{2} \grad_\eps \tilde w_\eps \otimes \grad_\eps \tilde w_\eps - \tfrac{1}{2} \grad_\eps w_\eps \otimes \grad_\eps w_\eps \Big) + \frac{1}{12} \int_S  Q_{R}^2\Big(\grad_\eps^2 (\tilde w_\eps - w_\eps) \Big) \\
\longrightarrow \quad & \int_S  Q_{R}^2\bigg(\hspace{-3pt} \begin{pmatrix} \partial_1 \tilde y_1 - \partial_1 y_1 + \frac{1}{2} (\tilde w'^2 - w'^2 )& 0 \\ 0 &  0\\
\end{pmatrix}\hspace{-3pt}\bigg) + \frac{1}{12} \int_I Q_{R}^2\bigg(\hspace{-3pt}\begin{pmatrix} \tilde w'' - w'' & \tilde \theta' - \theta' \\  \tilde \theta' - \theta' &  0\\
\end{pmatrix} \hspace{-3pt}\bigg).
\end{align*}
%\begin{align*}
%\mathcal{D}_\eps((u_\eps,v_\eps),(\tilde u_\eps, \tilde v_\eps))^2 =& \int_S Q_{D,\eps}^2\Big( E^\eps (\tilde y_\eps - y_\eps) + \tfrac{1}{2} \grad_\eps \tilde w_\eps \otimes \grad_\eps \tilde w_\eps - \tfrac{1}{2} \grad_\eps w_\eps \otimes \grad_\eps w_\eps \Big)\\
% &+ \frac{1}{12} \int_S  Q_{D,\eps}^2\Big(\grad_\eps^2 (\tilde w_\eps - w_\eps) \Big) \\
%\to & \int_S  Q_{D}^2\bigg( \begin{pmatrix} \partial_1 \tilde y_1 - \partial_1 y_1 + \frac{1}{2} (\tilde w'^2 - w'^2 )& 0 \\ 0 &  0\\
%\end{pmatrix}\bigg)\\
% &+ \frac{1}{12} \int_S Q_{D}^2\Big(\begin{pmatrix} \tilde w'' - w'' & \tilde \theta' - \theta' \\  \tilde \theta' - \theta' &  0\\
%\end{pmatrix} \Big) \\
%=& \mathcal{D}_0((y,w,\theta),(\tilde y, \tilde w, \tilde \theta))^2.
%\end{align*}
This along with \EEE \eqref{quadraticforms}  \EEE shows \eqref{mutualmetric}.  By \eqref{eq: nochmal1} and \eqref{eq: nochmal2} we get  $\phi_\eps(u_\eps,v_\eps) - \phi_\eps(\tilde u_\eps, \tilde v_\eps) \to \frac{1}{2}L_2 + \frac{1}{24}L_1$, where $L_1$ and $L_2$ are given for $z = \theta^2/2 + E_{22}$ and $\tilde \gamma = \gamma$, respectively. \EEE The condition in  \eqref{quadraticforms} implies that
$
Q_{W}^2(q_{11},q_{12},q_{22}) = a q_{11}^2 + b q_{12}^2 + c q_{22}^2$ for suitable \ZZZ$a,b, c >0$. \EEE In view of  \eqref{def:enegeryphi}, this yields    $\frac{1}{2}L_2 + \frac{1}{24}L_1  =  \phi_0(y,w,\theta) - \phi_0 (\tilde y, \tilde w, \tilde \theta)$ and concludes the proof of \eqref{mutualenergy}. \EEE
\end{proof}

%
%As $\argmin_{q_{12},q_{22}} Q_W^2(q_{11},q_{12}, q_{22}) = (0,0)$ we get
%%we can find quadratic forms $\bar Q_1$ and $\bar Q_2$ such that $Q_S^0(q_{11}) =  Q_S^2(q_{11},q_{12}, q_{22}) + \bar Q_1(q_{12}) + \bar Q_2 (q_{22})$ and $Q_S^1(q_{11},q_{12}) =  Q_S^2(q_{11},q_{12}, q_{22}) + \bar Q_2 (q_{22})$. Thus, we get
%\begin{align}
%& \nonumber\\ \longrightarrow \quad & \frac{1}{2} \int_S Q_{W}^2\bigg(\hspace{-3pt} \begin{pmatrix} \partial_1 y_1 + \frac{1}{2} w'^2  &  E_{12}+ w' \theta\\ E_{12}+ w' \theta &  E_{22} + \frac{1}{2}\theta^2 \\
%\end{pmatrix}\hspace{-3pt} \bigg) dx + \frac{1}{24} \int_I Q_{W}^2\bigg(\hspace{-3pt}\begin{pmatrix} w''  & \theta'\\  \theta'  & \gamma \\ 
%\end{pmatrix}\hspace{-3pt}\bigg) dx \nonumber\\
%-& \frac{1}{2} \int_S Q_{W}^2\bigg(\hspace{-3pt} \begin{pmatrix} \partial_1 \tilde y_1 + \frac{1}{2}\tilde w'^2  &  E_{12}+ w' \theta\\ E_{12}+ w' \theta &  E_{22} + \frac{1}{2}\theta^2 \\
%\end{pmatrix} \hspace{-3pt}\bigg) dx - \frac{1}{24} \int_I Q_{W}^2\bigg(\hspace{-3pt}\begin{pmatrix} \tilde w''  & \tilde  \theta'\\ \tilde  \theta'  & \gamma \\
%\end{pmatrix}\hspace{-3pt}\bigg) dx  \nonumber \\
%=& , \nonumber
%\end{align}

\begin{rem}[Recovery sequences]\label{rem: recov}
	\normalfont{
		The ansatz for the recovery sequence in \eqref{eq: rec1} and \eqref{eq: rec2} is similar to \cite{Freddi2018} with the difference that \ZZZ \textbf{(a)} \EEE \eqref{eq: rec2} is slightly modified to deal with boundary conditions and \ZZZ \textbf{(b)} \EEE we add suitable corrections concerning the difference of $(y,w,\theta)$ and $(\tilde y, \tilde w, \tilde \theta)$ and the variables $\gamma$ and $E_{22}$ \ZZZ resulting from Proposition~\ref{thm:compactness}. \EEE
		\begin{itemize}
			\item[\ZZZ\textbf{(a)}\EEE] The definition ensures that we can construct recovery sequences for $(\tilde y, \tilde w, \tilde \theta)$ in Theorem~\ref{th: Gamma}(ii): choose $(u_\eps, v_\eps)=(0,0)$ and $(y,w,\theta)=(0,0,0)$ and construct $(\tilde u_\eps, \tilde v_\eps) \in \mathscr{S}^{2D}_{\eps}$ as in the above proof. Then we can again obtain  \eqref{mutualenergy}. This along with $\phi_\eps(u_\eps, v_\eps) =0 = \phi(y,w,\theta)$ yields $\limsup_{\eps \to 0}  \phi_\eps(\tilde u_\eps, \tilde v_\eps) \le \phi_0 (\tilde y, \tilde w, \tilde \theta)$.
			\item[\ZZZ\textbf{(b)}\EEE] The correction ensures strong convergence in \eqref{mutual:firstintegralminnot0}--\eqref{mutual:secondintegralnotmin0} which allows to pass to the limit in the metric term. Clearly, this is not relevant  in the purely static setting \cite{Freddi2018}. \RRR Our construction does not only take the limiting configuration  $(y,w,\theta)$  into account, but also the limits $\gamma$ and $E_{22}$ provided by Proposition~\ref{thm:compactness}. Therefore, as $\gamma$ and $E_{22}$ may be $x_2$-dependent, in contrast to \cite{Freddi2018}, we need to construct suitable approximations in $C^\infty_c(S)$ instead of using  correction terms only defined on the interval $I$. \EEE   In the case \eqref{eq:quadraticform:minimumnot0}, being more general for $Q^2_W$, the quantities $z  - E_{22} - \tfrac{1}{2}  \theta^2$ and   $\tilde \gamma - \gamma$ do not vanish in general and therefore an additional assumption on $Q_{R,\eps}^2$ is required.  \RRR In particular, due to the involved relaxation, see \eqref{eq: first-stepi} and \eqref{eq: relaxationL2}, we only expect an inequality in \eqref{mutualenergy}. \EEE The argument for \EEE \eqref{quadraticforms} is simpler and the \ZZZ mutual \EEE recovery sequence even satisfies an equality in \eqref{mutualenergy} \ZZZ which can be deduced from the structure of the quadratic form. \EEE
		\end{itemize}
	}
\end{rem}

Now we derive the lower semicontinuity of the slopes. The same argument for a single energy/metric was used in the proof of \cite[Theorem 3.2]{MFMKJV}. We include the proof here for the reader's convenience.

\begin{theorem}[Lower semicontinuity of slopes]\label{theorem: lsc-slope}
	\BBB	Suppose that \eqref{quadraticforms} or \eqref{eq:quadraticform:minimumnot0} holds. Then, for \EEE each sequence $(u_\eps, v_\eps)_\eps$ with $(u_\eps, v_\eps) \in  \mathscr{S}^{2D}_{\eps,M}$ such that $(u_\eps, v_\eps) \stackrel{\pi\sigma}{\to}  (y,w,\theta)$ we have 
$$\liminf_{\eps \to 0}|\partial {\phi}_{\eps}|_{{\mathcal{D}}_{\eps}}(u_\eps, v_\eps) \ge |\partial  {\phi}_0|_{ {\mathcal{D}_0}}(y,w,\theta).$$ 
\end{theorem}
\begin{proof}
First, by Theorem \ref{th: Gamma}(i) we get $\phi_0(y,w,\theta)\le M$. \EEE Let $\delta>0$ and use  Lemma~\ref{Thm:representationlocalslope1}(i) to find \EEE $(\tilde y , \tilde w, \tilde \theta) \in \mathscr{S}^{1D}$ with $(\tilde y, \tilde w, \tilde \theta) \neq (y,w,\theta)$ such that
\begin{align*}
		\vert \partial \phi_0 \vert_{\mathcal{D}_0}(y,w,\theta) -\delta \leq \frac{\big(\phi_0(y,w,\theta) - \phi_0 (\tilde y , \tilde w, \tilde \theta) - \Phi_M^2(\mathcal{D}_0((y,w,\theta),(\tilde y, \tilde w, \tilde \theta))\EEE ) \EEE \big)^+}{\Phi^1(\mathcal{D}_0((y,w,\theta),(\tilde y, \tilde w, \tilde \theta\EEE ) \EEE ))}.
\end{align*}
Let $(\tilde u_\eps, \tilde v_\eps)_\eps$ be the mutual recovery sequence constructed in Lemma \ref{lem:mutual}. Then, Lemma \ref{lem:mutual} yields
\begin{align*}
|\partial {\phi}_0|_{{\mathcal{D}}_0}(y,w,\theta) - \delta \leq \liminf\limits_{\eps \to 0} \frac{\big(\phi_\eps(u_\eps,v_\eps) - \phi_\eps (\tilde u_\eps , \tilde v_\eps) - \Phi_{M}^2(\mathcal{D}_\eps \EEE ((u_\eps,v_\eps),(\tilde u_\eps, \tilde v_\eps )))\big)^+}{\Phi^1( \mathcal{D}_\eps \EEE ((u_\eps,v_\eps),(\tilde u_\eps, \tilde v_\eps)))}.
\end{align*}
In view of Lemma  \ref{lem: rep2d}(ii), \EEE taking the supremum and sending $\delta \to 0$ yields
\begin{align*}
\liminf\limits_{\eps \to 0}|\partial {\phi}_{\eps}|_{{\mathcal{D}}_{\eps}}(u_\eps, v_\eps) \ge |\partial  {\phi}_0|_{ {\mathcal{D}}}(y,w,\theta)
\end{align*}
and concludes the proof.
\end{proof}

\subsection{Convergence of minimizers and strong convergence} 
In the next theorem, we analyze \EEE time-discrete \EEE solutions introduced in \EEE \eqref{eq: ds-new1}. We denote by $\mathbf{\Phi}_\eps$ the two-dimensional and by $\mathbf{\Phi}_0$ the one-dimensional scheme. We show that minimizers of the two-dimensional discretization scheme converge to minimizers of the one-dimensional scheme.
\begin{theorem}[Convergence of minimizer of the schemes]\label{thm:gammaofscheme}
		\BBB	Suppose that \eqref{quadraticforms} or \eqref{eq:quadraticform:minimumnot0} holds. \EEE
	Let $(u_\eps,v_\eps)_\eps$ with $(u_\eps,v_\eps) \in \mathscr{S}^{2D}_{\eps}$ be a sequence such that
	\begin{align}\label{ineq: mintominreverse}
	(u_\eps, v_\eps) \stackrel{\pi\sigma}{\to} (y,w,\theta) \quad \text{ and }
	\quad  \phi_\eps(u_\eps,v_\eps) \to \phi_0 (y,w,\theta)
	\end{align}
	 as  $\eps \to 0$. \EEE 	Moreover, let $\tau>0$ and consider a sequence $ (Y_{\eps,\tau})_\eps$, $Y_{\eps,\tau} \in \mathscr{S}^{2D}_{\eps}$, \EEE such that
	\begin{align*}
	\BBB Y_{\eps,\tau} \EEE = \argmin\nolimits_{z \in \mathscr{S}^{2D}_{ \eps}} \mathbf{\Phi}_\eps(\tau, (u_\eps,v_\eps),z) \ \text{  for all $\eps >0$\EEE} \quad \text{ and } \quad    \BBB Y_{\eps,\tau} \EEE \stackrel{\pi\sigma}{\to} \BBB U_{\tau}\EEE \ \text{ as $\eps \to 0$}.
	\end{align*}
 Then, 
 \begin{align*}
 {\rm (i)} & \quad \BBB U_{\tau}\EEE = \argmin\nolimits_{\BBB v \EEE \in \mathscr{S}^{1D}} \mathbf{\Phi}_0(\tau, (y,w,\theta),\EEE v\EEE ), \\
  {\rm (ii)} &\quad \mathbf{\Phi}_\eps(\tau, (u_\eps,v_\eps),\BBB Y_{\eps,\tau} \EEE) \to \mathbf{\Phi}_0(\tau, (y,w,\theta),\BBB U_{\tau}\EEE),\\
 {\rm (iii)} & \quad \phi_\eps (\BBB Y_{\eps,\tau} \EEE) \to \phi_0(\BBB U_{\tau}\EEE).
 \end{align*}
\end{theorem}
\begin{proof}
	The proof is based on the fundamental  property that $\Gamma$-convergence induces convergence of minima and  \EEE minimizers.  By Theorem \ref{th: Gamma}\BBB(i) \EEE and Theorem \ref{th: lscD} we have
	\begin{align}
		\inf\limits_{v \in \mathscr{S}^{1D}}  \mathbf{\Phi}_0(\tau, (y,w,\theta),v) \le    \mathbf{\Phi}_0(\tau, (y,w,\theta),\BBB U_{\tau}\EEE) \EEE \leq \liminf\limits_{\eps \to 0} \mathbf{\Phi}_\eps(\tau, (u_\eps,v_\eps),\BBB Y_{\eps,\tau} \EEE). \label{ineq: mintomin}
	\end{align}
	Let $\delta>0$ and consider $(\tilde y, \tilde w, \tilde \theta) \in \mathscr{S}^{1D}$ such that
	\begin{align*}
	\inf\limits_{v \in \mathscr{S}^{1D}}  \mathbf{\Phi}_0(\tau, (y,w,\theta),v) + \delta \geq \mathbf{\Phi}_0(\tau, (y,w,\theta),(\tilde y, \tilde w, \tilde \theta)).
	\end{align*}
	In view of Lemma \ref{lem:mutual} and \eqref{ineq: mintominreverse}, \EEE we find a mutual recovery sequence $(\tilde u_\eps, \tilde v_\eps)_\eps$ such that
	\begin{align}
		\inf\limits_{v \in \mathscr{S}^{1D}}  \mathbf{\Phi}_0(\tau, (y,w,\theta),v) + \delta &\geq \limsup\limits_{\eps \to 0}  \phi_\eps (\tilde u_\eps, \tilde v_\eps) \EEE + \limsup\limits_{\eps \to 0} \frac{1}{2\tau} \mathcal{D}_\eps^2 ((u_\eps, v_\eps),(\tilde u_\eps, \tilde v_\eps)) \label{ineq:mintominreverse2} \\
		&\geq \limsup\limits_{\eps \to 0} \mathbf{\Phi}_\eps(\tau, (u_\eps,v_\eps), (\tilde u_\eps, \tilde v_\eps)) \ge \limsup\limits_{\eps \to 0} \mathbf{\Phi}_\eps(\tau, (u_\eps,v_\eps), \BBB Y_{\eps,\tau} \EEE), \nonumber
	\end{align}
	where the last step follows from minimality of $\BBB Y_{\eps,\tau} \EEE$. \EEE
	By sending $\delta \to 0$, \eqref{ineq: mintomin} and \eqref{ineq:mintominreverse2} imply (i) and (ii) \EEE as all inequalities are actually equalities. \EEE Finally, \EEE (iii) follows by (ii) and Theorems~\ref{th: Gamma}\EEE (i) \EEE and \ref{th: lscD}.
\end{proof}

 The following lemma \ZZZ helps us to show that the convergence of the sequences in Theorems \ref{thm:curveofmaximalslope:energyidentity} and \ref{thm:relation2d1d} holds in a strong sense. \EEE

\begin{lemma}[Strong convergence of recovery sequences]\label{lem:strongconvergenceofrecoverysequences}
	Let $(u_\eps,v_\eps)_\eps$ \ZZZ with $(u_\eps,v_\eps) \in \mathscr{S}_\eps^{2D}$ \EEE be a sequence such that
	\begin{align*}
	(u_\eps, v_\eps) \stackrel{\pi\sigma}{\to} (y,w,\theta) \quad \text{ and }
	\quad  \phi_\eps(u_\eps,v_\eps) \to \phi_0 (y,w,\theta).
	\end{align*}
	Then, we have $	(u_\eps, v_\eps) \stackrel{\pi\rho}{\to} (y,w,\theta)$.
\end{lemma}
\begin{proof} Let $(y_\eps,w_\eps)$ be the scaled version corresponding to $(u_\eps,v_\eps)$ introduced in \eqref{scaledfunctions}. \EEE   By Proposition  \ref{thm:compactness} we have 
	\begin{align}\label{eq: weaki}
\grad_\eps^2 w_\eps \rightharpoonup \begin{pmatrix}
	w'' & \theta' \\
	\theta' & \gamma \\
	\end{pmatrix} \quad \text{ and } \quad E^\eps y_\eps + \tfrac{1}{2} \nabla_\eps w_\eps \otimes \nabla_\eps w_\eps \rightharpoonup \begin{pmatrix}
	\partial_1 y_1 + \tfrac{1}{2}w''^2 & E_{12} + \frac{1}{2} w' \theta \\
	E_{12} + \frac{1}{2} w' \theta & E_{22} + \frac{1}{2} \theta^2 \\
	\end{pmatrix}
	\end{align}
	weakly  in  $L^2(S; \R^{2\times 2}_{{\rm sym}})$, for suitable functions  $\gamma \in \EEE L^2(S)\EEE$, $E_{12} \in L^2(S)$, and $E_{22} \in L^2(S)$.  It suffices to show that these convergences are strong. Then all (strong) convergences indicated in Proposition~\ref{thm:compactness} follow   where  $E^\eps y_\eps$ converges due to the compact embedding  $W^{1,2}(S) \subset \subset L^4(S)$.
We start the proof by observing that   the convexity of $Q_W^2$ and  $(u_\eps, v_\eps) \stackrel{\pi\sigma}{\to} (y,w,\theta)$  yield  
	\begin{align*}
	 \liminf\limits_{\eps \to 0} \phi_\eps (u_\eps,v_\eps) \ge 	\frac{1}{2}\int_S Q_W^2  \bigg( \hspace{-3pt} \begin{pmatrix}
		\partial_1 y_1 + \tfrac{1}{2}w'^2 & E_{12} + \frac{1}{2} w' \theta \\
		E_{12} + \frac{1}{2} w' \theta & E_{22} + \frac{1}{2} \theta^2 \\
		\end{pmatrix}\hspace{-3pt} \bigg) + \frac{1}{24} \int_{\ZZZ S \EEE} Q_W^2 \bigg(\hspace{-3pt} \begin{pmatrix}
		w'' & \theta' \\
		\theta' & \gamma \\
		\end{pmatrix}\hspace{-3pt} \bigg),
	\end{align*}
	where we used the representation of $\phi_\eps$ in  \eqref{def:vK2D-scaled}. By  \eqref{quadraticform1d}--\eqref{quadraticform0d} and \eqref{def:enegeryphi}, the right-hand side is bigger or equal than \EEE $\phi_0(y,w,\theta)$. \EEE This along with $\lim_{\eps \to 0} \phi_\eps (u_\eps,v_\eps) = \phi_0(y,w,\theta)$ shows that all inequalities  are actually equalities.  Using once more \eqref{eq: weaki} and the convexity of $Q_W^2$, we derive by \eqref{def:vK2D-scaled} that
\begin{align*}
& \frac{1}{2} \int_S \ZZZ Q_{W}^2 \EEE \big( E^\eps y_\eps + \frac{1}{2} \nabla_\eps w_\eps \otimes \nabla_\eps w_\eps \big)  \to 	\frac{1}{2}\int_S Q_W^2  \bigg( \hspace{-3pt} \begin{pmatrix}
		\partial_1 y_1 + \tfrac{1}{2}w'^2 & E_{12} + \frac{1}{2} w' \theta \\
		E_{12} + \frac{1}{2} w' \theta & E_{22} + \frac{1}{2} \theta^2 \\
		\end{pmatrix}\hspace{-3pt} \bigg), \\  
		& \frac{1}{24} \int_S \ZZZ Q_{W}^2 \EEE (\nabla_\eps^2 w_\eps) \to \frac{1}{24} \int_{\ZZZ S \EEE} Q_W^2 \bigg(\hspace{-3pt} \begin{pmatrix}
		w'' & \theta' \\
		\theta' & \gamma \\
		\end{pmatrix}\hspace{-3pt} \bigg). 
		\end{align*}
This together with weak convergence \eqref{eq: weaki}, the expansion \eqref{expansion}\BBB, and the positive definiteness on $\R^{2\times 2}_\sym $ of $Q_W^2$ \EEE shows that \eqref{eq: weaki} holds with strong convergence. \EEE
\end{proof}

\section{Proof of the main results}\label{sec results}

In this section, we give the \PPP proofs \EEE of the main results.

\subsection{Passage from 2D to 1D}
 
In this subsection, we prove Theorem \ref{thm:relation2d1d}. Let $M>0$. We fix a null sequence $(\eps_l)_{ l\in \N}$ and a sequence of initial data $\EEE (u^0_{\eps_l},v^0_{\eps_l}) \EEE \in \mathscr{S}^{2D}_{\eps,M}$ such that $(u^0_{\eps_l},v^0_{\eps_l}) \stackrel{\pi\sigma}{\to} (y^0,w^0,\theta^0) \in \mathscr{S}^{1D}$ and $\phi_\eps(u^0_{\eps_l},v^0_{\eps_l}) \to \phi_0(y^0,w^0,\theta^0)$ as $l \to \infty$. Moreover, we assume that \eqref{quadraticforms} or \eqref{eq:quadraticform:minimumnot0} holds. \EEE

%We define $(y_{\eps_k},w_{\eps_k}):= \pi_{\eps_k}(u_{\eps_k}, v_{\eps_k})$ and obtain that $(E^{\eps_k} y_{\eps_k})_{2j} \to 0$ for $j=1,2$ and $(\grad_{\eps_k}^2 w_{\eps_k})_{22} \to 0$ in $L^2(S)$ due to Lemma \ref{lem:strongconvergenceofrecoverysequences}. 

\begin{proof}[Proof of Theorem \ref{thm:relation2d1d}(i)] Let $(u_{\eps_l},v_{\eps_l})_{\eps_l}$ be a sequence of curves of maximal slopes for $\phi_{\eps_l}$ with respect to $\vert \partial \phi_{\eps_l}\vert_{\mathcal{D}_{\eps_l}}$  satisfying $(u_{\eps_l}(0),v_{\eps_l}(0)) = (u^0_{\eps_l}, v^0_{\eps_l})$. We check  \EEE that the assumptions of Theorem \ref{thm: sandierserfaty} are satisfied. The spaces $(\mathscr{S}^{2D}_{\eps_l,M},\mathcal{D}_{\eps_l})$ and $(\mathscr{S}^{1D},\mathcal{D}_{0})$ are complete metric spaces due to Lemma~\ref{th: metric space-lin}(i) and
Lemma~\ref{lem:complete1d}(i). In the notation of Subsection \ref{sec: auxi-proofs}, \BBB we have $\mathscr{S} = \pi_\eps(\mathscr{S}^{2D}_{\eps_l})$ and $\mathscr{S}_0 := \mathscr{S}^{1D}$. \EEE Clearly, we have $\mathscr{S}_0 \subset \mathscr{S}$.   \EEE   Moreover, Proposition \ref{thm:compactness} yields \eqref{basic assumptions2} \EEE  and Theorems \ref{th: Gamma}, \ref{th: lscD}, and \ref{theorem: lsc-slope} give \eqref{compatibility} and \eqref{eq: implication}. \EEE Furthermore, the slopes are strong upper gradients due to Lemma \ref{thm:stronguppergradient2d}(iii) and Lemma~\ref{Thm:representationlocalslope1}(ii). \EEE As the energies of the curves of maximal slopes are uniformly bounded depending only on the initial data, \ZZZ see~\eqref{maximalslope}, \EEE we can apply Theorem \ref{thm: sandierserfaty}. This \EEE yields the existence of a curve of maximal slope $(y,w,\theta)$ for $\phi_0$ with respect to $\vert \partial \phi_0 \vert_{\mathcal{D}_0}$ and the convergence $(u_{\eps_l}(t),v_{\eps_l}(t)) \stackrel{\pi \sigma}{\to}  (y(t),w(t),\theta(t))$ for $t\geq 0$\BBB, up \EEE to a subsequence. It remains to observe that the convergence is actually strong. This follows from the fact that $\lim_{l \to \infty} \phi_{\eps_l}(u_{\eps_l}(t), \BBB v_{\eps_l}(t)) \EEE = \phi_0(y(t),w(t),\theta(t))$ for $t \geq 0$ (see Theorem \ref{thm: sandierserfaty}) and Lemma \ref{lem:strongconvergenceofrecoverysequences}.
\end{proof}

\begin{proof}[Proof of Theorem \ref{thm:relation2d1d}(iii)]
Let $(\tau_l)_l$ be a null sequence and let $\bar Y_{\eps_l,\tau_l}$ be a discrete solution to the two-dimensional problem.  As in the previous proof, we check that all assumptions \EEE of \EEE Theorem~\ref{th:abstract convergence 2} are satisfied. \EEE Thus, there exists a subsequence such that we have
\begin{align*}
\bar Y_{\eps_l,\tau_l} (t) \stackrel{\pi\sigma} \to (y(t),w(t),\theta(t))
\end{align*}
for all $t\geq 0$ as $l\to \infty$ and $(y,w,\theta)$ is a curve of maximal slope for $\phi_0$ with respect to $\vert \partial \phi_0 \vert_{\mathcal{D}_0}$.
Moreover, the convergence holds in a strong sense for all $t\ge 0$ due to Lemma \ref{lem:strongconvergenceofrecoverysequences}.
\end{proof} 
 
% 
%The space $(\mathscr{S}^{1D},\mathcal{D}_{0})$ is a complete metric space due to Lemma \ref{lem:complete1d} (i). Moreover, Proposition \ref{thm:compactness} yields the compactness and Theorems \ref{th: Gamma}, \ref{th: lscD} and \ref{theorem: lsc-slope} provide the lower semicontinuity. Furthermore, the local slope is a strong upper gradient due to  Theorem \ref{Thm:representationlocalslope1} and Theorem \ref{th:abstract convergence 2} is applicable. 
% 

\begin{proof}[Proof of Theorem \ref{thm:relation2d1d}(ii)]
 Let $\tau>0$  and let $\bar Y_{\eps_l,\tau}$   be \EEE  a discrete solution to the two-dimensional problem \EEE with $\bar Y_{\eps_l,\tau}(0) = (u_{\eps_l}^0, v_{\eps_l}^0)$. \EEE \EEE By \EEE construction, the energies of the discrete solution $\bar Y_{\eps_l,\tau}$ are uniformly bounded \ZZZ depending \EEE only on the initial data. \EEE Thus, by a diagonal argument and Proposition \ref{thm:compactness}, we obtain \EEE a subsequence and $(U_\tau^n)_{n \in \N}$ such that $\bar Y_{\eps_l,\tau}(n\tau)  \stackrel{\pi\sigma}{\to} U_\tau^n$ for all $n \in \N$. Now, we need to show that 
\begin{align*}
{\rm(i)} \quad U_{\tau}^n = \argmin\limits_{v \in \mathscr{S}^{1D}} \mathbf{\Phi_0}(\tau, U_{\tau}^{n-1}; v), \qquad {\rm(ii)}\quad \bar Y_{\eps_l,\tau}(n\tau)  \stackrel{\pi\rho}{\to} U_{\tau}^n, \qquad {\rm (iii)} \quad  \phi_{\eps_l}(\bar Y_{\eps_l,\tau}(n\tau) ) \to \phi_0(U_{\tau}^n) \EEE
\end{align*}
for all $n \in \N$.  We show properties (i)--(iii) \EEE by induction. Suppose that  $\bar Y_{\eps_l,\tau}( (n-1)\EEE \tau)  \stackrel{\pi\rho}{\to} U_\tau^{n-1}$ and $\phi_{\eps_l}(\bar Y_{\eps_l,\tau}((n-1)\tau) ) \to \phi_0(U^{n-1}_\tau)$ \EEE  for a fixed $n \in \N$. (Clearly, this holds for $n=1$ \RRR by Lemma~\ref{lem:strongconvergenceofrecoverysequences} as $\bar Y_{\eps_l,\tau}(0)  \stackrel{\pi\sigma}{\to} U^{0}_\tau$  and  $\phi_{\eps_l}(\bar Y_{\eps_l,\tau}(0) ) \to \phi_0(U^{0}_\tau)$ by assumption.) \EEE Due to Theorem~\ref{thm:gammaofscheme}, the element $U_\tau^{n}$ is a minimizer of $\mathbf{\Phi_0}(\tau, U_\tau^{n-1}; \cdot )$ which \EEE shows (i). Moreover,   Theorem~\ref{thm:gammaofscheme} also yields $\phi_{\eps_l}(\bar Y_{\eps_l,\tau}(n\tau)) \to \phi_0(U_\tau^{n})$ \EEE which gives \EEE (iii). Finally,  \EEE   (ii) follows by Lemma~\ref{lem:strongconvergenceofrecoverysequences}. \BBB By defining $\bar U_\tau$ as in \eqref{eq: ds-new2} we can conclude. \EEE
\end{proof}

\subsection{Solutions in 1D} In this subsection we \EEE prove Theorem \ref{thm:curveofmaximalslope:energyidentity}.

\begin{proof}[Proof of Theorem \ref{thm:curveofmaximalslope:energyidentity}(i)]
Our goal is to apply \EEE Theorem \ref{th:abstract convergence 2}: instead of a sequence of metric spaces, we only consider the single \EEE metric space $(\mathscr{S}^{1D},\mathcal{D}_0)$ which is complete due to Lemma~\ref{lem:complete1d}(i). We let $\sigma$ be the weak convergence in $(\mathscr{S}^{1D}, \Vert \cdot \Vert_{can})$, \EEE for which Lemma \ref{lem:complete1d}(iii) provides the compactness property \eqref{basic assumptions2}. By  Lemma~\ref{lem:complete1d}(iv) and Lemma~\ref{thm:stronguppergradient1d}(iii) we  get \eqref{compatibility} and \eqref{eq: implication}. \EEE As $\vert \partial \phi_0 \vert_{\mathcal{D}_0}$ is a strong upper gradient, see Lemma~\ref{Thm:representationlocalslope1}(ii), Theorem \ref{th:abstract convergence 2} yields the convergence of time-discrete solution to a curve of maximal slope. \EEE Strong convergence with respect to $\Vert \cdot \Vert_{can}$ can be obtained by repeating the arguments in the proof of Lemma \ref{lem:strongconvergenceofrecoverysequences}. We omit the details. Eventually, convergence with respect to $\mathcal{D}_0$ is induced by Lemma \ref{sec:prop1d}(ii). \EEE  
\end{proof}
\begin{rem}[Alternative existence proof]
	\normalfont{
	If one is only interested in existence and not in \EEE time-discrete \EEE approximation, one could directly use Theorem \ref{thm:relation2d1d}(i): by Theorem \ref{th: Gamma} we can construct a  recovery sequence $(u^0_\eps,v^0_\eps)_\eps$ such that $(u^0_\eps,v^0_\eps) \stackrel{\pi\rho}{\to} (y^0,w^0,\theta^0)$. Then, curves of maximal slope $(u_\eps,v_\eps)$ for $\phi_\eps$ with respect to $\vert \partial \phi_\eps \vert_{\mathcal{D}_\eps}$ in the two-dimensional setting exist by  \cite[Theorem 2.2]{MFMKDimension}\BBB, up to minor adjustments, \EEE and by \EEE Theorem~\ref{thm:relation2d1d}(i) we conclude the proof. Note, however, that in this way we need to require  \eqref{quadraticforms} or \eqref{eq:quadraticform:minimumnot0}. \eop \EEE}
\end{rem}

After having shown existence of  curves of maximal slope, our goal is to establish a relation  \EEE   to the effective one-dimensional equations, \EEE see Theorem \ref{thm:curveofmaximalslope:energyidentity}(ii). \EEE The natural idea is to make use of the energy identity \eqref{maximalslope}. For this purpose, we use a finer representation of the local slope $\vert \partial \phi_0 \vert_{\mathcal{D}_0}$. Afterwards, we give sharp estimates for the metric derivative and the derivative of $\phi_0 \circ (y,w,\theta)$. This will provide enough information for the relation to the equations\ZZZ. For \EEE the following arguments, \EEE it is convenient to introduce   the abbreviations
\begin{align}
&H(y,w,\theta \vert \hat w):= ( \partial_1 y_1 + w^\prime \hat w^\prime, w^{\prime\prime}, \theta^\prime) \in L^2(S; \R^3),\nonumber \\
&G(y,w,\theta):= (\partial_1 y_1 + \frac{\vert \wpr \vert^2}{2},\wprpr, \theta^\prime) \in L^2(S;\R^3)\label{def:HandG}
\end{align}
for $(y,w,\theta) \in \mathscr{S}^{1D}$ and $\hat w \in W^{2,2}(I)$. By an elementary computation we get \EEE 
\begin{align}
G(y,w, \theta)\hspace{-0.03cm} - \hspace{-0.03cm} G( \tilde y, \tilde w, \tilde \theta) & = (\partial_1 y_1 - \partial_1 \tilde y_1 + \vert\wpr\vert^2 - \twpr \wpr,\wprpr- \twprpr, \theta^\prime - \tilde \theta^\prime) 
\hspace{-0.03cm} - \hspace{-0.03cm}\big( \tfrac{\vert \wpr \vert^2}{2} - \wpr \twpr + \tfrac{\vert \twpr \vert^2}{2},0,0\big) \notag \\
& = H(y-\tilde y, w- \tilde w, \theta - \tilde \theta \vert w) - \tfrac{1}{2}((w'-\tilde w')^2,0,0).\label{eq:HandG}
\end{align}
We further introduce extended quadratic forms
\begin{align}\label{eq:extendedquadraticform}
\bar Q_S(x_1,x_2,x_3) := Q_S^0(x_1) + \frac{1}{12} Q_S^1(x_2,x_3)
\end{align} for $(x_1,x_2,x_3)^T \in \R^3$ and $S = W,R$.

\begin{lemma}[Representation of the energy and the metric]\label{lemma:representation} For $(y,w,\theta), (\tilde y, \tilde w, \tilde \theta) \in \mathscr{S}^{1D}$, \RRR it holds \EEE 
\begin{align*}
\phi_0(y,w,\theta)&= \frac{1}{2} \int_S \bar Q_W \big(G(y,w,\theta)\big) \ZZZ, \EEE \\
\mathcal{D}_0^2((y,w,\theta),(\tilde y, \tilde w, \tilde \theta)) &= \int_S \bar Q_R\big(G(y,w,\theta) - G(\tilde y, \tilde w, \tilde \theta)\big). 
\end{align*}

\end{lemma}
\begin{proof}
The statement follows directly from \eqref{def:metric}--\eqref{def:enegeryphi} and  \eqref{def:HandG}. \EEE
\end{proof}

By \EEE \eqref{quadraticform1d} and \eqref{quadraticform0d} we find that $\bar Q_S$ is positive definite on $\R^3$ for $S = W,R$. We denote by $\bar \C_S$ its associated bilinear form which induces a bijective mapping $(x_1,x_2,x_3) \mapsto \bar \C_S (x_1,x_2,x_3)$ from $\R^3$ to $\R^3$. By $\sqrt{\bar \C_S}$ we denote its unique root and by $\sqrt{\bar \C_S}^{-1}$ the inverse of $\sqrt{\bar \C_S}$.

%
% \BBB In the following, we drop the symbol for transposition as its clear from context. \EEE
%

\begin{lemma}[Fine representation of the one-dimensional slope]\label{thm:sloperepresentation}
There exists a differential operator ${\mathcal{L}}\colon\mathscr{S}^{1D} \to L^2(S; \R^3)$ satisfying
\begin{align*}
\int_S {\mathcal{L}}(y,w,\theta) \cdot H(\phi_y,\phi_w,\phi_\theta\vert w) = 0
\end{align*}
for all $(y,w,\theta) \in \mathscr{S}^{1D}$ and $(\phi_y,\phi_w,\phi_\theta) \in BN_{(0,0)}(S;\R^2)\times W_0^{2,2}(I)\times W_0^{1,2}(I)$ such that the local slope at $(y,w,\theta) \in \mathscr{S}^{1D}$ can be represented by 
\begin{align*}
\vert \partial \phi_0 \vert_{\mathcal{D}_0} (y,w,\theta) = \Big \Vert \sqrt{\bar \C_R}^{-1} \big(\bar \C_W G(y,w,\theta) + {\mathcal{L}}(y,w,\theta) \big) \Big \Vert_{L^2(S;\R^3)}.
\end{align*}
\end{lemma}

The proof follows along the lines of the corresponding representation in dimension two, see \EEE \cite[Lemma 6.1]{MFMKDimension}. For the convenience of the reader, we give a self-contained proof in Appendix~\ref{Proof of Lemma thm:sloperepresentation}.

\begin{proof}[Proof of Theorem \ref{thm:curveofmaximalslope:energyidentity}(ii)]
We now use a standard technique to relate curves of maximal slope to PDEs in Hilbert spaces, see \cite[Section 1.4]{AGS}. More precisely,  \EEE  the proof follows the lines of \cite[Theorem 2.2(ii)]{MFMKDimension}. We divide the proof into \EEE  three steps. First, we construct \EEE the curve $(\xi_1,\xi_2,w,\theta)$ and prove the regularity stated in the theorem. Step 2 consists in deriving \EEE sharp estimates for $\vert v' \vert_{\mathcal{D}_0}$ and $\frac{d}{dt}\phi_0 \circ v$ for $v = (y,w,\theta)$. This allows \EEE to relate the curve to the one-dimensional system of equations in Step 3.

\noindent \textit{Step 1:} Let $(y,w,\theta)\colon [0,\infty) \to \mathscr{S}^{1D}$ be a curve of maximal slope for $\phi_0$ with respect to $\vert \partial \phi_0 \vert_{\mathcal{D}_0}$. Using \eqref{def:BN} there exist functions $\xi_1\colon[0,\infty) \to  W_{\hat u_1}^{1,2}(I)$ and $\xi_2\colon[0,\infty) \to W_{\hat u_2}^{2,2}(I)$ such that $y_1(t) (x) = \xi_1(t) (x_1) - x_2 \xi_2(t)' (x_1)$ and $y_2(t) (x) = \xi_2(t)(x_1)$ for all $t \geq 0$ almost everywhere in $S$.
%We apply Fubini's Theorem and obtain for all $0 \leq s\leq t \leq +\infty$
%\begin{align*}
%&\Vert y_1(t)-y_1(s) \Vert_{L^{2}(S,\R^2)}^2\\
% \geq &\int_S \big((\xi_1(t)-\xi_1(s)) - x_2 (\xi_2'(t) -\xi_2'(s))\big)^2 \\
% = &\int_S (\xi_1(t)-\xi_1(s))^2 -2x_2 (\xi_1(t)-\xi_1(s)) (\xi_2'(t) -\xi_2'(s)) + (\xi_2'(t) -\xi_2'(s))^2 \\
%\geq & \Vert \xi_1(t)-\xi_1(s) \Vert_{L^2(I)}^2 + \Vert \xi_2'(t) -\xi_2'(s) \Vert_{L^2(I)}^2.
%\end{align*}
%This estimation can be adapted to $\partial_1 y_1(t)- \partial_1y_1(s)$ analogously. Thus, using \linebreak $\Vert \xi_2(t) -\xi_2(s) \Vert_{L^2(I)} = \Vert y_2(t) -y_2(s) \Vert_{L^2(I)}$ yields
%\begin{align}
%C \Vert y(t)-y(s) \Vert_{W^{1,2}(S,\R^2)} \geq \Vert \xi_1(t)-\xi_1(s) \Vert_{W^{1,2}(I)} + \Vert \xi_2(t)-\xi_2(s) \Vert_{W^{2,2}(I)} \label{regularitymetric}
%\end{align}
%for all $0 \leq s\leq t \leq +\infty$. Similarly, we have for every $t \geq 0$
%\begin{align}
%C \Vert y(t) \Vert_{W^{1,2}(S,\R^2)} \geq \Vert \xi_1(t)\Vert_{W^{1,2}(I)} + \Vert \xi_2(t)\Vert_{W^{2,2}(I)}. \label{regularity}
%\end{align}
By definition, we have for every $t \geq 0$
\begin{align}
 \Vert \xi_1(t)\Vert_{W^{1,2}(I)} + \Vert \xi_2(t)\Vert_{W^{2,2}(I)} \leq C \Vert y(t) \Vert_{W^{1,2}(S;\R^2)} . \label{regularity}
\end{align}
As $(y,w,\theta)$ is a curve of maximal slope we get that $\phi_0(y(t),w(t),\theta(t))$ is decreasing in time, see \eqref{maximalslope}. This together with Lemma \ref{lem:complete1d}(iii) and \eqref{regularity} yields 
\begin{align}
(\xi_1,\xi_2,w,\theta) \in L^\infty\big([0,\infty);\mathcal{K} \big), \label{boundedlimiting}
\end{align}
where $\mathcal{K}$ is defined in \eqref{eq: mathcalK}. \EEE 
%Next, we want to show that $(\xi_1,\xi_2,w,\theta)$ is an absolutely continuous curve with respect to the Hilbert space $\mathcal{K}$.
As $(y(t),w(t),\theta(t))_{t \geq 0}$ is absolutely continuous with respect to $(\mathscr{S}^{1D},\mathcal{D}_0)$ we have that $\vert (y,w,\theta)'\vert_{\mathcal{D}_0} \in L^2([0,\infty))$. Then, by Lemma \ref{lem:complete1d}(ii), \eqref{regularity}, \EEE and  \eqref{boundedlimiting} we observe that $(\xi_1,\xi_2,w,\theta)$ is an absolutely continuous curve with respect to $\mathcal{K}$. 
Thus, by \cite[Remark~1.1.3]{AGS} we observe that $\xi_1$, $\xi_2$, $w$, and $\theta$ are differentiable for   a.e.\ $t$ and we have (\ref{regularitylimiting}). More precisely, for all $0 \leq s <t$, and almost everywhere in $S$ (respectively $I$) it holds that 
\begin{align}\label{lemma:fundamental}
f(t)-f(s) = \int_s^t \partial_t f(r) \, \EEE {\rm d}r \EEE \quad \text{for } f \in \{\partial_1 y_1, w', w'', \theta'\} .
\end{align}
% Thus, it holds for all $0 \le s \leq t \leq \infty$
%\begin{align*}
%&\Vert (\xi_1(t),\xi_2(t),w(t),\theta(t))-(\xi_1(s),\xi_2(s),w(s),\theta(s)) \Vert\\
%\leq  &\Vert \xi_1(t)-\xi_1(s) \Vert_{W^{1,2}(I)}+ \Vert \xi_2(t)-\xi_2(s) \Vert_{W^{2,2}(I)} + \Vert w(t)-w(s) \Vert_{W^{2,2}(I)} + \Vert \theta(t)-\theta(s) \Vert_{W^{1,2}(I)}\\
% = &C \Vert y(t)-y(s) \Vert_{W^{1,2}(S,\R^2)} + \Vert w(t)-w(s) \Vert_{W^{2,2}(I)} + \Vert \theta(t)-\theta(s) \Vert_{W^{1,2}(I)}\\
% \leq &C D((y(t),w(t), \theta (t)),(y(s), w (s), \theta(s))) + C\Vert \vert w^\prime (t)\vert^2 - \vert w'(s) \vert^2\Vert_{L^2(I)}  \\
% \leq &C D((y(t),w(t), \theta (t)),(y(s), w (s), \theta(s))) + C\Vert (w'(t)- w'(s)) \Vert_{L^4(I)} \Vert (w'(t) + w'(s)) \Vert_{L^4(I)}\\
% \leq &C D((y(t),w(t), \theta (t)),(y(s), w (s), \theta(s))) + C\Vert (w'(t)- w'(s)) \Vert_{W^{1,2}(I)}\\
% \leq &C D((y(t),w(t), \theta (t)),(y(s), w (s), \theta(s)))\\
% \leq &\int_s^t C \vert (y,w,\theta)'\vert_D(r) dr,
%\end{align*}
%where we have used that the energy identity from Theorem \ref{curveofmaximalslope:energyidentity} yields $\vert (y,w,\theta)' \vert_D \in L^2([0,\infty))$.
\textit{Step 2:}
As preparation for the representation of the metric derivative, we now consider the difference $G(y,w,\theta)(t) - G(y,w,\theta)(s)$. The identity (\ref{eq:HandG}) and the linearity of $H(\cdot,\cdot,\cdot \vert w(t))$ yield  \EEE for a.e.\ $t$ and a.e.\ $x \in S$
\begin{align}
\lim\limits_{s \to t} \frac{G(y,w,\theta)(t) - G(y,w,\theta)(s)}{t-s} 
=& H(\partial_t y(t), \partial_t w(t), \partial_t \theta(t)\vert w(t)) - \partial_t w(t) \EEE \lim\limits_{s \to t} \frac{1}{2}(w(t)-w(s))  \nonumber \\
=& H(\partial_t y(t), \partial_t w(t), \partial_t \theta(t)\vert w(t)). \label{derivativeG}
\end{align}
 Using \eqref{def:HandG}, \EEE (\ref{lemma:fundamental}), Poincar\'e's and \EEE Jensen's inequality, and Fubini's Theorem, we obtain for all $0\leq s \le t$
\begin{align}
\Big\Vert G(y(t),w(t),\theta(t))&-G(y(s),w(s),\theta(s)) - \int_s^t H(\partial_t y(r), \partial_t w(r), \partial_t \theta(r)\vert w(t)) \,\EEE {\rm d}r \EEE\Big\Vert^2_{L^2(S;\R^3)}\nonumber  \\ 
& = \int_S \frac{1}{2}(w'(t)-w'(s))^4
\leq  C \Big(\int_S (w''(t)-  w''(s))^2 \Big)^2  \nonumber \\
&=  C\Big( \int_S \vert t-s \vert^2 \Big( \frac{1}{\vert t-s \vert} \int_s^t \partial_t w^{\prime\prime}(r)\,\EEE {\rm d}r \EEE \Big)^2 \,\EEE {\rm d}x \EEE\Big)^{ 2}\nonumber \\
& \leq C \Big( \int_S \vert t-s \vert \int_s^t \partial_t w^{\prime\prime}(r)^2 \,\EEE {\rm d}r \EEE \,\EEE {\rm d}x \EEE \Big)^2
= C \vert t-s\vert^2 \Big( \int_s^t \Vert \partial_t w^{\prime\prime}(r)\Vert^2_{L^2(S)} \Big)^2. \label{estimation1identification}
\end{align}
We now estimate the metric derivative $\vert (y,w,\theta)'\vert_{\mathcal{D}_0}$. By Lemma \ref{lemma:representation}, (\ref{derivativeG}), and Fatou's Lemma we get for a.e.\ $t\geq 0$
\begin{align}
\vert (y,w,\theta)'\vert_{\mathcal{D}_0}(t) &= \lim\limits_{s \to t} \Bigg(\frac{{\mathcal{D}^2_0}\big((y(t),w(t),\theta(t)),(y(s),w(s),\theta(s))\big)}{\vert t-s\vert^{ 2}}\Bigg)^{1/2}\nonumber\\
&\geq \Bigg(\int_S \liminf\limits_{s \to t} \bar Q_R\bigg(\frac{G(y(t),w(t),\theta(t))-G(y(s),w(s),\theta(s))}{\vert t-s\vert}\bigg) \Bigg)^{ 1/2}\nonumber\\
&= \Big\Vert \sqrt{\bar \C_R}H(\partial_t y(t), \partial_t w(t), \partial_t \theta(t)\vert w(t))\Big\Vert_{L^2(S;\R^3)}. \label{inequality:metricderivative}
\end{align}
We now analyze the derivative $\frac{\rm d}{{\rm d}t} (\phi_0 \circ (y,w,\theta))(t)$ of the absolutely continuous curve $\phi_0 \circ (y,w,\theta)$. Note that for a.e.\ $t \geq 0$ we have $\lim_{s \to t} \int_s^t \Vert \partial_t w^{\prime\prime}(r) \Vert^2_{L^2(S)}\, \EEE {\rm d}r \EEE = 0$ by (\ref{regularitylimiting}) and, in a similar fashion, $\lim_{s \to t} {\vert s-t\vert^{-1} } {\Vert \int_s^t H(\partial_t y(r), \partial_t w(r), \partial_t \theta(r)\vert w(t)) \, \EEE {\rm d}r \EEE \Vert^2_{L^2(S;\R^3)}}= 0 $ by (\ref{regularitylimiting}) and Hölder's inequality. Thus, using Lemma \ref{lemma:representation}, \eqref{expansion}, \EEE and \eqref{derivativeG}--\EEE (\ref{estimation1identification}), we obtain for a.e.\ $t \geq 0$
\begin{align*}
\frac{\rm d}{{\rm d}t} ( \EEE \phi_0 \circ (y,w,\theta))(t) = &\lim\limits_{s \to t} \frac{\phi_0(y(t),w(t),\theta(t))-\phi_0(y(s),w(s),\theta(s))}{t-s}\\
\geq &\liminf\limits_{s \to t} \frac{1}{t-s}\int_S \bar \C_W \big[G(y(t),w(t),\theta(t)), G(y(t),w(t),\theta(t)- G(y(s),w(s),\theta(s))\big] \\
&- \limsup\limits_{s \to t} \frac{1}{2(t-s)}\int_S \bar Q_W\big(G(y(t),w(t),\theta(t) \EEE ) \EEE - G(y(s),w(s),\theta(s))\big)\\
\geq &\int_S \bar \C_W \big[G(y(t),w(t),\theta(t) \EEE ) \EEE , H(\partial_ty(t), \partial_tw(t), \partial_t \theta(t)\vert w(t))\big].
\end{align*}
By the property of $\mathcal{L}$ stated in Lemma \ref{thm:sloperepresentation}, and the fact that $\partial_ty(t), \partial_tw(t), \partial_t \theta(t)$ vanish on  $\partial I$ and $\partial S$, respectively, we get \EEE 
\begin{align*}
\frac{\rm d}{{\rm d}t}& ( \EEE \phi_0 \circ (y,w,\theta))(t)   \ge \hspace{-0.06cm} \int_S \big(\bar \C_W G(y(t),w(t),\theta(t))+{\mathcal{L}}(y(t),w(t),\theta(t))\big) \cdot H(\partial_ty(t), \partial_tw(t), \partial_t \theta(t)\vert w(t)) \\
= &\int_S \sqrt{\bar \C_R}^{-1}\big(\bar \C_WG(y(t),w(t),\theta(t))+{\mathcal{L}}(y(t),w(t),\theta(t))\big) \cdot \sqrt{\bar \C_R} H(\partial_ty(t), \partial_tw(t), \partial_t \theta(t)\vert w(t)).
\end{align*}
We find by
Lemma \ref{thm:sloperepresentation}, (\ref{inequality:metricderivative}), and Young's inequality
\begin{align*}
\frac{\rm d}{{\rm d}t}\phi_0((y,w,\theta)(t))\
 \geq - \frac{\vert (y,w,\theta)'\vert_{\mathcal{D}_0}^2(t)}{2} - \frac{\vert \partial \phi_0 \vert^2_{\mathcal{D}_0}(y(t),w(t),\theta(t))}{2}
\geq \frac{\rm d}{{\rm d}t}\phi_0((y,w,\theta)(t))
\end{align*}
for a.e.\ $t\geq 0$, where the last step is a consequence of the fact that $(y(t),w(t),\theta(t))$ is a curve of maximal slope with respect to $\phi_0$. Consequently, all inequalities employed in the proof are in fact equalities, and we get
\begin{align*}
\sqrt{\bar \C_R}^{-1}\big(\bar \C_WG(y(t),w(t),\theta(t))+{\mathcal{L}}(y(t),w(t),\theta(t))\big) +\sqrt{\bar \C_R} H(\partial_t y(t), \partial_t w(t), \partial_t \theta(t)\vert w(t))= 0
\end{align*}
pointwise a.e.\ in $S$ for a.e.\ $t \geq 0$. Multiplying the equation with $\sqrt{\bar \C_R}$ from the left and testing with $H(\phi_y,\phi_w,\phi_\theta\vert w(t))$ from the right for $(\phi_y,\phi_w,\phi_\theta) \in BN_{(0,0)}(S,\R^2) \times W_0^{2,2}(I) \times W_0^{1,2}(I)$ yields with the property of Lemma \ref{thm:sloperepresentation} for a.e.\ $t \geq 0$ 
\begin{align}
\int_S \Big( \bar \C_WG(y(t),w(t),\theta(t))+\bar \C_R H(\partial_ty(t), \partial_tw(t), \partial_t \theta(t)\vert w(t)) \Big) \cdot H(\phi_y,\phi_w,\phi_\theta\vert w(t)) = 0. \label{Fundamentalequation}
\end{align}
\textit{Step 3:}
Eventually, we verify that $(\xi_1,\xi_2,w,\theta)$ solve the one-dimensional equations. To this end, we use the identity (\ref{Fundamentalequation}) by choosing functions such that $\phi_i=\phi_j=0$ for $i,j \in \{y,w,\theta\}$ with $i\neq j$ in (\ref{Fundamentalequation}).
The simplest case is the derivation of (\ref{equation:4}) by setting $\phi_y = 0$ and $\phi_w = 0$. 
To this end, we recall \eqref{def:HandG} and \eqref{eq:extendedquadraticform} and remark that $\partial_i \bar Q_S(\cdot) = 2 ( \bar \C_S (\cdot))_i$ as $\bar \C_S$ is symmetric. Thus, an integration leads to \EEE (omitting the variable $t$ from now on) \EEE
\begin{align*}
0=& \int_S \Big( \bar \C_WG(y,w,\theta)+\bar \C_R H(\partial_ty, \partial_tw, \partial_t \theta\vert w) \Big) \cdot (0,0, \phi_\theta') \\
=& \frac{1}{24} \int_I \big(\partial_2 Q_W^1(w'',\theta') + \partial_2  Q_R^1(\partial_t w'', \partial_t \theta')\big)  \phi_\theta',
\end{align*}
which is exactly equation (\ref{equation:4}). Setting \EEE $\phi_y = 0$ and $\phi_\theta= 0$ leads to
%\begin{align*}
%\textstyle
%0 =& \int_S \bar \C_W\Big((\partial_1y_1 + \frac{\vert w'\vert^2}{2},0)+x_3(w'',\theta')] + \bar \C_R[(\partial_t \partial_1y_1 + w' \partial_tw',0) + x_3(\partial_t w'', \partial_t \theta')\Big)\\
%& \quad \quad \quad \cdot \Big(\phi_w' w,0)+x_3((\phi_w'',0)\Big).
%\end{align*}
%We now proceed similar to the formal derivation of the equations. We use the definition of $\C_R^0$ and obtain
%\begin{align*}
%0 =& \int_I \Big(\C^0_W \xi_1' + \frac{\vert w' \vert^2}{2} + \C_R^0\partial_t \xi_1' + w' \partial_t w'\Big)\cdot w\phi_w'\\
%& \frac{1}{12} \int_I \Big( \bar \C_W(w'',\theta')+ \bar \C_R(\partial_t w'', \partial_t \theta')\Big) \cdot (\phi_w'',0),
%\end{align*}
\begin{align*}
0 =& \hspace{-0.07cm} \int_I \Big(C^0_W (\xi_1' + \frac{\vert w' \vert^2}{2}) + C_R^0(\partial_t \xi_1' + w' \partial_t w')\Big) w'\phi_w'  + \hspace{-0.07cm}   \frac{1}{24} \EEE \int_I \Big(  \partial_1 Q_W^1(w'',\theta')+ \partial_1 Q_R^1(\partial_t w'', \partial_t \theta')\Big)  \phi_w'', \EEE
\end{align*}
\EEE where we used (\ref{def:BN}) and $\int_{-1/2}^{1/2}x_2  \, {\rm d}x_2 = 0$. \EEE This gives  (\ref{equation:3}).
The missing equations follow by setting $\phi_\theta = 0$ and $\phi_w =0$:  \BBB characterization \EEE (\ref{def:BN}) yields the existence of functions $\phi_{\xi_1} \in W^{1,2}_0 (I)$ and  $\phi_{\xi_2} \in W_0^{2,2}(I)$ such that $\partial_1 (\phi_y)_1 = \phi_{\xi_1}' - x_2 \phi_{\xi_2}''$. Inserting this and using $\int_{-1/2}^{1/2} x_2^2 \, {\rm d}x_2 = 1/12$ \EEE yields
\begin{align*}
0 = \int_I \Big( C_W^0(\xi_1' + \frac{\vert w' \vert}{2}) + C_R^0(\partial_t \xi_1' + w' \partial_t w')\Big) \phi_{\xi_1}' + \frac{1}{12} \Big( C_W^0 \xi_2''  + C_R^0 \partial_t \xi_2'' \Big) \phi_{\xi_2}''.
\end{align*}
As we can choose $\phi_{\xi_1}=0$ and $\phi_{\xi_2}=0$ independently, we obtain (\ref{equation:1}) and (\ref{equation:2}).
\end{proof}

\section*{Acknowledgements} \BBB This work was funded by  the DFG project FR 4083/5-1 and  by the Deutsche Forschungsgemeinschaft (DFG, German Research Foundation) under Germany's Excellence Strategy EXC 2044 -390685587, Mathematics M\"unster: Dynamics--Geometry--Structure. \EEE

\appendix
\section{One-dimensional properties}\label{sec:Appendix}

In this section we give the proofs of the properties in the one-dimensional setting.  The arguments are similar to the two-dimensional case and one can follow closely the lines of \cite{MFMK, MFMKDimension}. Yet, we include complete proofs here for \EEE  the reader's convenience. We start by proving Lemma~\ref{lem:complete1d}  which is rather elementary.   \EEE The proof of Lemma~\ref{Thm:representationlocalslope1} is \EEE more technical and is divided into the following steps: by constructing suitable ``generalized geodesics'', \EEE see Lemma \ref{PhiMPhi}, we establish the \EEE representation of the local slope $\vert \partial \phi_0 \vert_{\mathcal{D}_0}$ stated in \EEE Lemma~\ref{Thm:representationlocalslope1}(i).  This is at the core of proving Lemma~\ref{Thm:representationlocalslope1}(ii),(iii), where we additionally construct suitable \emph{mutual recovery sequences} for the weak lower semicontinuity of the slope similar to Lemma \ref{lem:mutual} (see Lemma \ref{lem:mutualrecovery}). \EEE Finally, at the end of the section, \EEE we give the  proof of \EEE Lemma \ref{thm:sloperepresentation}. In the following, \EEE $C>0$ denotes \EEE a universal constant that may change from line to line.  Moreover, the symbol $\rightharpoonup$ will stand for weak convergence in the space $(\mathscr{S}^{1D}, \Vert \cdot \Vert_{can})$. \EEE

\begin{proof}[Proof of Lemma \ref{lem:complete1d}]
	We first derive the lower bounds \eqref{ineq:lowerbound} and \eqref{ineq:lowerboundy}. By \EEE Poincaré's inequality and the fact that $Q_R^1$ \EEE is \EEE positive definite, we get \EEE for all $(y,w,\theta),(\tilde y,\tilde w,\tilde\theta) \in \mathscr{S}^{1D}$ that \EEE
	\begin{align*}
	  \Vert{w- \tilde w}\Vert_{W^{2,2}(I)}^2 + \Vert{\theta- \tilde \theta}\Vert_{W^{1,2}(I)}^2 & \le  C \Vert{w^{\prime\prime}- \tilde w^{\prime\prime}}\Vert_{L^2(I)}^2 + C\Vert{\theta^{\prime}- \tilde \theta^{\prime}}\Vert_{L^2(I)}^2  \nonumber\\
	 &  \le  C\int_I Q_R^1\big(w^{\prime\prime}- \tilde w^{\prime\prime}, \theta^\prime - \tilde \theta^{\prime}\big) 
	 \le C\mathcal{D}_0^2((y,w,\theta),(\tilde y,\tilde w,\tilde\theta)),
	\end{align*} \EEE
	 i.e.,  \eqref{ineq:lowerbound}  holds. \EEE By (\ref{def:BN}) and Korn-Poincaré's inequality we get
	\begin{align*}
	\Vert y - \tilde y\Vert_{W^{1,2}(S;\R^2)}^2 \leq C \Vert e(y-\tilde y) \Vert_{L^2(S; \R^{2 \times 2}_{\rm sym})}^2 = \EEE C \Vert \partial_1 y_1 - \partial_1 \tilde y_1\Vert_{L^2(S)}^2. 
	\end{align*}
 This along with the triangle inequality,  the positivity of $Q_R^0$, \EEE and H\"older's inequality \EEE yields 
	\begin{align*}
	\Vert y-\tilde y \Vert^2_{W^{1,2}(S;\R^2)} &\leq  C\int_S Q_R^0\Big(\partial_1 y_1 - \partial_1 \tilde y_1 + \tfrac{\vert w^\prime\vert^2}{2} - \tfrac{\vert \tilde w^{\prime}\vert^2}{2}\Big) + C\Vert \vert w^\prime \vert^2 - \vert \tilde w^\prime \vert^2\Vert^2_{L^2(I)} \\
	& \le  C {\mathcal{D}^2_0}((y,w, \theta),(\tilde y, \tilde w, \tilde \theta)) + C\Vert w^\prime +  \tilde w^\prime \Vert^2_{L^4(I)} \Vert w^\prime  -  \tilde w^\prime \Vert^2_{L^4(I)}.
	\end{align*}
	This shows \eqref{ineq:lowerboundy}. \EEE Thus, the embedding $W^{1,2}(I) \subset \subset L^4(I)$ and Hölder's inequality \EEE  yield that every converging sequence with respect to the topology induced by $\mathcal{D}_0$ converges with respect to $\Vert \cdot \Vert_{can}$ and vice versa. This shows (ii). \EEE  Using analogous computations to (ii), with $\phi_0$ in place of $\mathcal{D}_0^2$, \EEE  we find that (iii) holds. To see  (iv),  we consider \EEE  $(y_k,w_k,\theta_k) \rightharpoonup (y,w,\theta)$ and $( \tilde y_k,\tilde w_k,\tilde \theta_k) \rightharpoonup (\tilde y,\tilde w,\tilde \theta)$. As \EEE  $W^{1,2}(I) \subset\subset L^4(I)$, we obtain
	\begin{align*}
	\partial_1 ({y_k})_1 + \frac{\vert w_k^\prime \vert^2}{2} \rightharpoonup \partial_1 {y}_1  + \frac{\vert w^\prime \vert^2}{2}, \quad \partial_1 ({\tilde y_k})_1 + \frac{\vert \tilde w_k^\prime \vert^2}{2} \rightharpoonup \partial_1 \tilde{y}_1  + \frac{\vert \tilde w^\prime \vert^2}{2} \EEE \quad \text{in } L^2(S).
	\end{align*}
	As the quadratic forms $Q_W^0$, $Q_W^1$, $Q_R^0$, and $Q_R^1$ \EEE are positive definite, \EEE weak lower semicontinuity follows. \EEE We finally  prove (i). The positivity and the completeness follow from (ii). Eventually, the triangle inequality  follows from the fact that $\mathcal{D}_0^2$ is the sum of two    quadratic forms. \EEE
\end{proof}

We now aim at proving that $\vert \partial \phi_0 \vert_{\mathcal{D}_0}$ is \EEE weakly lower semicontinuous and a strong upper gradient. To verify this, we follow the approach of \cite{MFMKJV} which is based on a generalized convexity condition as the metric $\mathcal{D}_0$ and the energy $\phi_0$ are non-convex, due to the nonlinearity $|w'|^2$. We refer to \EEE \cite[Remark 1]{MFMKJV} for a detailed discussion. \EEE In the sequel, we frequently replace $(y,w,\theta)$ by a single variable $u$ for notational convenience. \EEE

\begin{lemma}[Convexity and generalized geodesics in the one-dimensional setting]\label{PhiMPhi}
	Let $M >0$.  Let $\Phi^1(t) := \sqrt{t^2 + C  t^3 + C  t^4}$ and $\Phi_{M}^2(t):= C\sqrt{ M} t^2 + C t^3 +  Ct^4$ for any $C>0$ large enough.  Then, \EEE for all $u_0 := (y_0,w_0, \theta_0) \EEE \in \mathscr{S}^{1D}$ satisfying $\phi_0(u_0) \leq M$ and all $u_1 := (y_1, w_1, \theta_1) \EEE \in \mathscr{S}^{1D}$ we have
	\begin{itemize}
		\item[(i)] $\mathcal{D}_0(u_0,u_s) \le s \Phi^1(\mathcal{D}_0(u_0,u_1))$
		\item[(ii)] $\phi_0(u_s) \leq (1-s) \phi_0(u_0) + s \phi_0(u_1) + s \Phi_M^2(\mathcal{D}_0(u_0,u_1))$
	\end{itemize}
	for $u_s:= (1-s) u_0 + s u_1$ and $s \in [0,1]$.
\end{lemma}

\begin{proof}
	Let $M >0$, $u_0   \in \mathscr{S}^{1D}$ with $\phi_0(u_0) \leq M$, and $u_1 \in \mathscr{S}^{1D}$. For convenience, we first introduce some abbreviations and provide some preliminary estimates. We set $\mathcal{D}= \mathcal{D}_0 (u_0,u_1)$, $B_\diff= \tfrac{1}{2}(w_0'-w_1')^2$, and $G_0^s = \partial_1  (y_s)_1 \EEE + \tfrac{1}{2} |w_s'|^2$ for $s \in [0,1]$. Then, by  Sobolev embedding and \eqref{ineq:lowerbound} we find \EEE
	\begin{align}\label{convexityproofbdiff}
	\Vert B_\diff \Vert_{L^2(I)} \leq C \Vert w'_0-  w'_1 \Vert_{L^4(I)}^2 \leq C\Vert w_0- w_1 \Vert_{W^{2,2}(I)}^2 \leq C \mathcal{D}^2.
	\end{align} \EEE
	Then,  by  \EEE the positivity of $Q_R^0$ \EEE and the definition of $\mathcal{D}_0$ we derive
	\begin{align}\label{convexityproofGdiff}
	\Vert G_0^1 - G_0^0 \Vert_{L^2({\EEE S})}^2 \leq C \int_{\EEE S}  Q_R^0 \EEE (G_0^1 - G_0^0) \leq C \mathcal{D}^2.
	\end{align}
	Similarly, we observe by the definition of $\phi_0$, $Q_W^0>0$, \EEE and the fact that $\phi_0(u_0) \leq M$ that
	\begin{align} \label{convexityproofG_0}
	\Vert G_0^0 \Vert_{L^2({\EEE S})}^2 \leq C \int_{\EEE S}  Q_W^0 \EEE (G_0^0)\leq CM.
	\end{align}
	We now start with the proof of (i). First, we observe
	\begin{align*}
	\frac{1}{12} \int_I Q_R^1 \big((w_s'',\theta_s')-( w_0'', \theta_0')\big) = s^2 \frac{1}{12} \int_I Q_R^1 \big((w_1'',\theta_1')-( w_0'', \theta_0')\big).
	\end{align*}
	We will show that there exists $C>0$ independently \EEE of $s$ such that
	\begin{align}\label{convexityprooftoshowphi}
	\int_S Q_R^0 (G_0^s - G_0^0) \leq s^2 \int_S Q_R^0(G_0^1 - G_0^0) + C s^2 \mathcal{D}^3 + C s^2 \mathcal{D}^4
	\end{align}
	for $s \in [0,1]$. Then, recalling the definition of $\mathcal{D}_0$, (i) follows for the function $\Phi^1(t) = \sqrt{t^2 + Ct^3 + Ct^4}$. To show \eqref{convexityprooftoshowphi}, we obtain by an elementary expansion
	%\begin{align*}
	%\frac{1}{2} (w_s'^2 - w_0'^2 )= s \frac{1}{2} (w_1'^2 - w_0'^2) - s(1-s) B_\diff 
	%\end{align*}
	%and, thus,
	\begin{align}\label{convexityproofGsG0difference}
	G_0^s - G_0^0 = s (G_0^1 - G_0^0 - (1-s) B_\diff).
	\end{align}
	Then \EEE \eqref{quadraticform0d} \EEE yields
	\begin{align*}
	\int_S Q_R^0(G_0^s - G_0^0) =& s^2 \int_S  Q_R^0 (G_0^1 - G_0^0) + s^2(1-s) \int_S Q_R^0(B_\diff)- 2s^2(1-s) \int_{\EEE S} \BBB C_R^0 \EEE (G_0^1-G_0^0)B_\diff,
	\end{align*}
	from which we deduce \eqref{convexityprooftoshowphi}, using \eqref{convexityproofbdiff}, \eqref{convexityproofGdiff}, and the Cauchy-Schwarz \EEE inequality.
	
	We now show (ii) for the function $\Phi_M^2(t) = C \sqrt{M} t^2 + C t^3 + C t^4$ for some $C>0$. Due to convexity of $s \mapsto \int_{\EEE I} Q_W^1(w_s'',\theta_s')$, it suffices to show
	\begin{align}\label{convexityprooftoshowD}
	\int_S Q_W^0 (G_0^s) \leq (1-s) \int_S Q_W^0 (G_0^0) + s \int_S Q_W^0 (G_0^1) + C \sqrt{M} s \mathcal{D}^2 + C s \mathcal{D}^3 + C s^2 \mathcal{D}^4.
	\end{align}
	By \eqref{convexityproofGsG0difference} we have $G_0^s = (1-s) G_0^0 + s G_0^1 - s(1-s) B_\diff$, and then \EEE  an elementary expansion yields
	\begin{align*}
	(G_0^s)^2 =&	(1-s) (G_0^0)^2 + s (G_0^1)^2 \\&- (1-s)s (G_0^0-G_0^1)^2 - 2s(1-s)^2 G_0^0B_\diff - 2 s^2 (1-s) G_0^1 B_\diff + s^2 (1-s)^2 B_\diff^2.
	\end{align*}
	Thus, by taking the integral and using the Cauchy-Schwarz inequality we get \EEE 
	\begin{align*}
	\int_S Q_W^0 (G_0^s) \leq &(1-s) \int_S Q_W^0 (G_0^0)  + s \int_S Q_W^0 (G_0^1) \\
	&+ C s ( \Vert G_0^0 \Vert_{L^2(S)} + \Vert G_0^1 \Vert_{L^2(S)} ) \Vert B_\diff \Vert_{L^2(S)} + C s^2 \Vert B_\diff \Vert_{L^2(S)}^2.
	\end{align*}
	By using $\Vert G_0^1 \Vert_{L^2(S)} \leq \Vert G_0^0 \Vert_{L^2(S)} + \Vert G_0^1 - G_0^0 \Vert_{L^2(S)}$, \eqref{convexityproofbdiff}, \eqref{convexityproofGdiff}, and \eqref{convexityproofG_0} we get \eqref{convexityprooftoshowD}. This concludes the proof of (ii).
\end{proof}

We are now ready to prove  the representation of the local slope stated in Lemma \ref{Thm:representationlocalslope1}(i), \EEE which will help us to show that the local slope is weakly lower semicontinuous and a strong upper gradient.

\begin{proof}[Proof of Lemma \ref{Thm:representationlocalslope1}(i)]
	Let $\Phi^1$ and $\Phi^2_M$ be defined as in Lemma \ref{PhiMPhi} and note \EEE that $\lim_{t \to 0} \Phi^1(t)/t = 1$ and $\lim_{t \to 0} \Phi^2_M(t)/ t=0$. This and the property that $(a+b)^+ \leq a^+ + b^+$ for every $a,b \in \R$ yields
	\begin{align*}
	\vert \partial \phi_0 \vert_{\mathcal{D}_0}(y,w,\theta) =& \limsup\limits_{(\tilde y, \tilde w, \tilde \theta) \to (y,w,\theta)} \frac{\big(\phi_0(y,w,\theta) - \phi_0(\tilde y, \tilde w, \tilde \theta)\big)^+}{{\mathcal{D}_0}((y,w,\theta),(\tilde y, \tilde w, \tilde \theta))}\\
	 \leq & \limsup\limits_{(\tilde y, \tilde w, \tilde \theta) \to (y,w,\theta)} \EEE \frac{\big(\phi_0(y,w,\theta) - \phi_0(\tilde y, \tilde w, \tilde \theta) -\Phi_M^2({\mathcal{D}_0}((y,w,\theta),(\tilde y, \tilde w, \tilde \theta)))\big)^+}{ \Phi^1({\mathcal{D}_0}((y,w,\theta),(\tilde y, \tilde w, \tilde \theta)))}\\
	\leq & \sup\limits_{(y,w,\theta)\neq (\tilde y, \tilde w, \tilde \theta) \in \mathscr{S}^{1D}} \frac{\big(\phi_0(y,w,\theta) - \phi_0(\tilde y, \tilde w, \tilde \theta) -\Phi_M^2({\mathcal{D}_0}((y,w,\theta),(\tilde y, \tilde w, \tilde \theta)))\big)^+}{\Phi^1({\mathcal{D}_0}((y,w,\theta),(\tilde y, \tilde w, \tilde \theta)))}.
	\end{align*}
	To see the other inequality, \EEE we write  $u_0 = (y,w,\theta)$ and \EEE fix $u_1:=(\tilde y, \tilde w, \tilde \theta)$   \EEE   with $u_0 \neq u_1$. \EEE  Define \linebreak $u_s:= (1-s) u_0 + su_1$. By Lemma \ref{PhiMPhi} \EEE we obtain the inequality
	\begin{align*}
	\frac{(\phi_0(u_0)- \phi_0(u_s))^{ +}}{{\mathcal{D}_0}(u_0,u_s)}\geq \frac{(s\phi_0(u_0)- s \phi_0(u_1) -s\Phi_M^2({\mathcal{D}_0}(u_0,u_1)))^{ +}}{s\Phi^1({\mathcal{D}_0}(u_0,u_1))}. 
	\end{align*}
	In view of  Theorem \ref{lem:complete1d}(ii), \EEE   $u_s \to u_0$ as $s \to 0$ with respect to the topology induced by  $\mathcal{D}_0$,  \EEE i.e.,   by letting $s \to 0$ the left-hand side is smaller or equal to $\vert \partial \phi_0 \vert_{\mathcal{D}_0}(u_0)$.  Taking the supremum on the right-hand  side \EEE over all $u_1 \in \mathscr{S}^{1D}$, $u_1\neq u_0$, \EEE yields the lower bound for the local slope.
\end{proof}
%In contrast to Lemma \ref{lem:representationphi}, it is more difficult to derive the weak lower semicontinuity of the local slope. Since the metric occurs in the denominator of (\ref{def:slopes}), it seems that we need the reverse inequality in Lemma \ref{lem:representationphi} (i), which is false in general. However, the representation in Theorem \ref{Thm:representationlocalslope1} will solve this problem: Given a sequence $(z_k)_k \subset \mathscr{S}^{1D}$ and $z$, $u \in \mathscr{S}^{1D}$ such that $z_k \rightharpoonup z$, the additional dependence on the supremum allows us to define a recovery sequence $(u_n)_n$ such that the "reverse inequality" is true, more precisely, we have

Similarly to Lemma \ref{lem:mutual},  we construct a mutual recovery sequence as a key ingredient to show weak \EEE lower semicontinuity of slopes. \EEE
\begin{lemma}[Mutual recovery sequence in the one-dimensional setting]\label{lem:mutualrecovery}
	Let \ZZZ $(z_k)_k \subset \mathscr{S}^{1D}$ be a sequence such that $z_k \rightharpoonup z$ and $u \in \mathscr{S}^{1D}$\EEE. Then there \ZZZ exists \EEE a sequence $(u_k)_k \subset \mathscr{S}^{1D}$ \EEE such that
	\begin{itemize}
		\item[(i)] $\lim\limits_{k \to \infty} {\mathcal{D}_0}(z_k, u_k) = {\mathcal{D}_0}(z,u)$,
		\item[(ii)] $\phi_0(z) - \phi_0(u) = \lim\limits_{k \to \infty} (\phi_0(z_k) - \phi_0(u_k))$.
	\end{itemize}
\end{lemma}
\begin{proof}
	Let $z_k \EEE = \EEE (y_k,w_k, \theta_k)$ such that $z_k \rightharpoonup z = (y,w, \theta) \in \mathscr{S}^{1D}$ and \EEE consider \EEE $ u = (\tilde y, \tilde w, \tilde \theta) \in \mathscr{S}^{1D}$. We define the mutual recovery sequence by 
	$$u_k:=( \tilde y_k, \tilde w_k, \tilde \theta_k)= (y_k + \tilde y - y, w_k + \tilde w - w, \theta_k + \tilde \theta - \theta).$$Then, by the compact embedding $W^{1,2}(I) \subset \subset L^4(I)$, we have
	\begin{align*}
	\vert w_k^\prime \vert^2 -\vert \tilde w_k^\prime \vert^2  \to \vert w^\prime \vert^2 -\vert \tilde w^\prime \vert^2 \text{ in } L^2(I).
	\end{align*}
	Moreover, by construction it holds that $\partial_1 (y_k)_1 \EEE - \partial_1  (\tilde{y}_k)_1 = \partial_1 y_1 - \partial_1 \tilde y_1$, $w_k'' - \tilde w_k'' = w''- \tilde w''$, and $\theta_k' - \tilde \theta_k' = \theta' - \tilde \theta'$. This implies (i) since
	\begin{align*}
	\textstyle
	\lim\limits_{k \to \infty} {\mathcal{D}^2_0}((y_k,w_k, \theta_k),(\tilde y_k , \tilde w_k, \tilde \theta_k)) = \int_S Q_R^0(\partial_1 y_1 - \partial_1\tilde y_1 + \frac{\vert w^\prime \vert^2}{2} -\frac{\vert \tilde w^\prime \vert^2}{2} ) + \EEE \tfrac{1}{12} \EEE \int_I Q_R^1(\wprpr - \twprpr, \theta^\prime - \tilde \theta^\prime).
	\end{align*}
	We now address (ii). \EEE We \ZZZ only have weak convergence of $(\partial_1 \EEE (y_k)_1 \EEE + \frac{1}{2}\vert w_k^\prime \vert^2)_k$ in $L^2(S)$, but strong convergence of $(\partial_1 (\tilde{y}_k)_1 \EEE + \frac{1}{2}\vert \tilde w_k^\prime \vert^2 - \partial_1 (y_k)_1 + \frac{1}{2}\vert w_k^\prime \vert^2)_k$ in $L^2(S)$. Thus, by adapting the arguments \EEE in Step 3 of the proof \EEE of Lemma \ref{lem:mutual}, in particular using \eqref{expansion}, \EEE we find as $k \to \infty$ \EEE
	\begin{align*}
	&\int_S   \Big( Q_W^0\big(\partial_1 (y_k)_1 + \tfrac{\vert w_k^\prime \vert^2}{2}\big) - Q_W^0\big(\partial_1 (\tilde {y}_k)_1 + \tfrac{\vert \tilde w_k^\prime \vert^2}{2}\big)      \Big)\to \int_S \Big( Q_W^0\big(\partial_1 {y}_1 + \tfrac{\vert w^\prime \vert^2}{2}\big) - Q_W^0\big(\partial_1 \tilde {y}_1 + \tfrac{\vert \tilde w^\prime \vert^2}{2}\big)      \Big). 
	\end{align*}
	The observation that 
	\begin{align*}
	\int_I Q_W^1(w_k'' - \tilde w_k'', \theta'_k - \tilde \theta'_k) = \int_I Q_W^1(w'' - \tilde w'', \theta' - \tilde \theta')
	\end{align*}
	for every $k \in \N$ concludes the proof of (ii).
\end{proof}

 We now proceed with the proof of Lemma \ref{Thm:representationlocalslope1}(ii),(iii). \EEE

\begin{proof}[Proof of Lemma \ref{thm:stronguppergradient1d}(ii),(iii)]
	\BBB We first show (ii). \EEE As \EEE the local slope is a weak upper gradient in the sense of Definition \cite[Definition 1.2.2]{AGS} by \cite[Theorem  1.2.5]{AGS}, we only need to show that for an absolutely continuous curve $z\colon(a,b) \to \mathscr{S}^{1D}$ satisfying $\vert \partial \phi_0\vert_{\mathcal{D}_0}(z) \vert z '\vert_{\mathcal{D}_0} \in L^1(a,b)$ the curve $\phi_0 \circ z$ is absolutely continuous. It is  not  \EEE   restrictive to assume that $(a,b)$ is a bounded interval and  that \EEE the curve $z$ is extended by continuity to $[a,b]$. Thus, $\mathscr{S}^{1D}_z:= z([a,b])$ is compact and we define $\text{diam}(\mathscr{S}^{1D}_z):= \sup_{s,t \in [a,b]}  {\mathcal{D}_0}(z(s),z(t)) < +\infty$. Thus, by Lemma \ref{lem:complete1d}(ii) \EEE we find an $M>0$ such that $\phi_0(z(s)) \leq M$ for every $s \in [a,b]$. Since $\phi_0$ is \BBB ${\mathcal{D}_0}$-lower semicontinuous by \BBB Lemma \ref{lem:complete1d}(ii) and (iv)\EEE, the global slope
	\begin{align*}
	I_{\phi_0}(v):= \sup\limits_{v \neq w \in \mathscr{S}^{1D}_z} \frac{(\phi_0(v)-\phi_0(w))^+}{{\mathcal{D}_0}(v,w)}
	\end{align*}
	is a strong upper gradient with respect to $\mathscr{S}^{1D}_z$ by \cite[Theorem 1.2.5]{AGS}. Thus, it holds for all $a < s \leq t < b$ that
	\begin{align*}
	\vert \phi_0 (z(t)) -\phi_0(z(s)) \EEE \vert \leq \int_s^t I_{\phi_0}(z(r))\vert  z' \BBB \vert_{\mathcal{D}_0}(r) \, {\rm d}r,
	\end{align*} 
	\EEE see Definition \ref{main def2}(i). \EEE 	The claim follows once \EEE we bound \EEE  $I_{\phi_0}(z)\vert z'\vert_{\mathcal{D}_0}$ \EEE with an integrable function. To this end, we define the constants
	\begin{align*}
	C_1:= \sup\limits_{t \in [0,\text{diam}(\mathscr{S}^{1D}_z)]} \frac{\Phi^1(t)}{t} < +\infty \quad \text{ and } \quad
	C_2:= \sup\limits_{t \in [0,\text{diam}(\mathscr{S}^{1D}_z)]} \frac{\Phi^2_M(t)}{t} < +\infty,
	\end{align*}
	where $\Phi^1$ and $\Phi^2_M$ are  given in Lemma~\ref{thm:stronguppergradient1d}(i).
	Hence, for every $\EEE v \EEE \in \mathscr{S}^{1D}_z$ it holds that
	\begin{align*}
	I_{\phi_0}(v)
	\leq  &\sup\limits_{v \neq w \in \mathscr{S}^{1D}_z} \frac{\big(\phi_0(v)-\phi_0(w)- \Phi_M^2({\mathcal{D}_0}(v,w))\big)^+}{\Phi^1({\mathcal{D}_0}(v,w))}\frac{\Phi^1({\mathcal{D}_0}(v,w))}{{\mathcal{D}_0}(v,w)}  
	+\sup\limits_{v \neq w \in \mathscr{S}^{1D}_z} \frac{\big(\Phi_M^2({\mathcal{D}_0}(v,w))\big)^+ }{{\mathcal{D}_0}(v,w)} \\
	\leq & C_1 \vert \partial \phi_0 \vert_{\mathcal{D}_0}(v) + C_2.
	\end{align*}
	By using the assumption that $\vert \partial \phi_0 \vert_{\mathcal{D}_0} (z)\vert z' \vert_{\mathcal{D}_0} \in L^1(a,b)$ we get the absolute continuity.

	\EEE We now show (iii). \EEE Let $(y_k,w_k,\theta_k)_k \EEE \subset \EEE \mathscr{S}^{1D}$ be such that $(y_k,w_k,\theta_k) \rightharpoonup (y,w,\theta) \EEE$ for some $(y,w,\theta) \in \mathscr{S}^{1D}$. We let $\eps >0$ and define $z_k:= (y_k,w_k,\theta_k)$, $z := (y,w,\theta)$, and $M:= \phi_0(z)$. \EEE By Lemma~\ref{Thm:representationlocalslope1}(i) there exists  $u \in \mathscr{S}^{1D}$ such that
	\begin{align*}
	\vert \partial \phi_0\vert_{\mathcal{D}_0}(z) \leq \frac{\big(\phi_0(z)-\phi_0(u) - \Phi_M^2\big({\mathcal{D}_0}(z,u)\big) \big)^+}{\Phi^1\big({\mathcal{D}_0}(z,u)\big)} + \eps.
	\end{align*}
	Let $(u_k)_k \EEE \subset \EEE \mathscr{S}^{1D}$ be the sequence given by Lemma \ref{lem:mutualrecovery}. Since $\Phi^1$ and $\Phi^2_M$ are continuous, we obtain
	\begin{align*}
	\vert \partial \phi_0\vert_{\mathcal{D}_0}(z) \leq &\liminf\limits_{k \to \infty} \frac{\big(\phi_0(z_k)-\phi_0(u_k) - \Phi_M^2\big({\mathcal{D}_0}(z_k,u_k) \big)\big)^+}{\Phi^1\big({\mathcal{D}_0}(z_k,u_k)\big)} + \eps \leq  \liminf\limits_{k \to \infty} \vert \partial \phi_0\vert_{\mathcal{D}_0}(z_k)+\eps,
	\end{align*}
	 where the last step follows again by Lemma \ref{Thm:representationlocalslope1}(i). The statement \EEE follows by sending $\eps \to 0$.
\end{proof}

We close this section by  proving \EEE  the fine representation for the local slope given in  Lemma \ref{thm:sloperepresentation}. \EEE

\begin{proof}[Proof of Lemma \ref{thm:sloperepresentation}]\label{Proof of Lemma thm:sloperepresentation}
	To simplify the notation, we will write $(\tilde y, \tilde w, \tilde \theta) \to (y, w, \theta)$ instead of \linebreak $\mathcal{D}_0((\tilde y, \tilde w, \tilde \theta),(y,w, \theta)) \to 0$. \BBB Recall the definitions in \eqref{def:HandG} and \eqref{eq:extendedquadraticform}. \EEE
	The embedding $W^{1,2}(I) \subset \subset L^4(I)$ yields $\Vert G(y,w,\theta) - G(\tilde y, \tilde w, \tilde \theta)\Vert_{L^2(S; \R^3)} \to 0$ as $(\tilde y, \tilde w, \tilde \theta) \to (y, w, \theta)$. Then, \BBB by the positivity of $\bar{Q}_R$ \EEE we get
	\begin{align*}
	\frac{\int_S \bar Q_W(G(y,w,\theta)-G(\tilde{y}, \tilde w, \tilde \theta))}{\big(\int_S \bar Q_R(G(y,w,\theta)-G(\tilde{y}, \tilde w, \tilde \theta))\big)^{1/2}} \leq C \frac{\Vert G(y,w,\theta)-G(\tilde y, \tilde w, \tilde \theta) \Vert_{L^2(S;\R^3)}^2}{\Vert G(y,w,\theta)-G(\tilde y, \tilde w, \tilde \theta) \Vert_{L^2(S;\R^3)}} \to 0.
	\end{align*}
	%	In addition, the identity 
	%	\begin{align*} 
	%	&\frac{1}{2} \bar Q_W\big(G(y,w,\theta)\big)-\bar Q_W\big(G(\tilde y,\tilde w, \tilde \theta)\big) \\
	%	= &\bar \C_W[G(y,w,\theta),G(y,w,\theta)-G(\tilde y,\tilde w, \tilde \theta)] - \frac{1}{2} \bar Q_W\big(G(y,w,\theta)-G(\tilde y,\tilde w, \tilde \theta)\big)
	%	\end{align*}
	 This along with \EEE the expansion \eqref{expansion} and  Lemma \ref{lemma:representation} yields \EEE
	\begin{align*}
	\vert \partial \phi_0\vert_{\mathcal{D}_0}(y,w, \theta)&= \limsup\limits_{(\tilde y, \tilde w, \tilde \theta) \to (y, w, \theta)} \frac{\big( \int_S \frac{1}{2} \bar Q_W(G(y,w,\theta))-  \frac{1}{2}\bar Q_W(G(\tilde y,\tilde w, \tilde \theta)) \big)^+ }{\big(\int_S \bar Q_R(G(y,w,\theta)-G(\tilde{y}, \tilde w, \tilde \theta))\big)^{1/2}} \nonumber \\
	&= \limsup\limits_{(\tilde y, \tilde w, \tilde \theta) \to (y, w, \theta)} \frac{\big( \int_S \bar \C_W[G(y,w,\theta),G(y,w,\theta)-G(\tilde y,\tilde w, \tilde \theta)] \big)^+ }{\big(\int_S \bar Q_R(G(y,w,\theta)-G(\tilde{y}, \tilde w, \tilde \theta))\big)^{1/2}} \nonumber \\
	&= \limsup\limits_{(\tilde y, \tilde w, \tilde \theta) \to (y, w, \theta)} \frac{\big( \int_S \bar \C_W[G(y,w,\theta),H(y-\tilde y, w- \tilde w, \theta - \tilde \theta \vert w) - \frac{1}{2}((w'-\tilde w')^2,0,0)] \big)^+ }{\big(\int_S \bar Q_R(G(y,w,\theta)-G(\tilde{y}, \tilde w, \tilde \theta))\big)^{1/2}} ,
	\end{align*}
	where the last equality follows from (\ref{eq:HandG}). Due to \eqref{ineq:lowerbound}, \eqref{eq:HandG}, \EEE  and Lemma \ref{lemma:representation},  we find that  
	\begin{align*} 
	\frac{\int_S \bar Q_R(G(y,w,\theta)-G(\tilde{y}, \tilde w, \tilde \theta))}{\int_S \bar Q_R(H(y-\tilde y, w- \tilde w, \theta - \tilde \theta \vert w))} \to 1, \quad \quad  \frac{\int_S (w'-\tilde w')^2}{\big(\int_S \bar Q_R(G(y,w,\theta)-G(\tilde{y}, \tilde w, \tilde \theta))\big)^{1/2}} \to 0 \EEE
	\end{align*}
	as $(\tilde y, \tilde w, \tilde \theta) \to (y, w, \theta)$. \EEE
	Thus, we obtain
	\begin{align*}
	\vert \partial \phi_0\vert_{\mathcal{D}_0}(y,w, \theta)&=\limsup\limits_{(\tilde y, \tilde w, \tilde \theta) \to (y, w, \theta)} \frac{\big( \int_S \bar \C_W[G(y,w,\theta),H(y-\tilde y, w- \tilde w, \theta - \tilde \theta \vert w)] \big)^+ }{\big(\int_S \bar Q_R(H(y-\tilde y, w- \tilde w, \theta - \tilde \theta \vert w))\big)^{1/2}}.
	\end{align*}
	We introduce the space of test functions $\mathcal{P}:= BN_{(0,0)}(S,\R^2)\times W_0^{2,2}(I) \times W_0^{1,2}(I)$. Since the operator $H$ is linear, we can simplify this expression by substitution with sequences that converge to $0$.  Moreover, as the enumerator and denominator are \BBB positively homogeneous of degree one, \EEE  we derive the representation 
	\begin{align}
	\vert \partial \phi_0\vert_{\mathcal{D}_0}(y,w, \theta)&=\sup\limits_{ 0\neq (\hat y, \hat w, \hat \theta) \in \mathcal{P}} \frac{\big( \int_S \bar \C_W[G(y,w,\theta),H(\hat y, \hat w, \hat \theta \vert w)] \big)^+ }{\big\Vert \sqrt{\bar \C_R}H(\hat y, \hat w, \hat \theta \vert w)\big\Vert_{L^2(S;\R^3)}}, \label{slopeproof}
	\end{align}
	where we have used \EEE the definition \EEE of $\sqrt{\bar \C_R} $. We now want to show that the supremum is attained by considering the minimization problem
	\begin{align*}
	\min\limits_{(\bar y, \bar w,\bar \theta) \in \mathcal{P}}\mathcal{F}(\bar y,\bar w,\bar \theta),
	\end{align*}
	where
	\begin{align*}
	\mathcal{F}(\bar y, \bar w,\bar \theta):= \frac{1}{2} \int_S \Big\vert \sqrt{\bar \C_R} H(\bar y, \bar w,\bar \theta \vert w)\Big\vert^2 - \int_S \bar \C_W[G(y,w,\theta),H(\bar y, \bar w,\bar \theta \vert w)].
	\end{align*}
	The existence of a solution can be guaranteed by the direct method of the calculus of variations. The functional $\mathcal{F}$ is weakly lower semicontinuous as $\sqrt{\bar \C_R}$ and $H$ are linear operators and $\vert \cdot \vert^2$ is convex. To show coercivity, we consider a constant $C>0$ such that $\mathcal{F}(\bar y, \bar w,\bar \theta) \leq C$. Since $\bar \C_R[H(\bar y, \bar w,\bar \theta \vert w), H(\bar y, \bar w,\bar \theta \vert w)] \ge \EEE c \vert H(\bar y, \bar w,\bar \theta \vert w) \vert^2$ by the positivity of $\bar{Q}_R$, \EEE we obtain by Cauchy-Schwarz $$\Vert H(\bar y, \bar w,\bar \theta \vert w) \Vert_{L^2(S;\R^3)}\leq C,$$ where $C$ depends on $y,w$, and $\theta$. Arguing similarly to the proof of Lemma \ref{lem:complete1d}(ii), we find that
	$	\Vert \bar w \Vert_{W^{2,2}(I)} \leq C$, $\Vert \bar \theta \Vert_{W^{1,2}(I)} \leq C$,  and 
	$\Vert \bar y \Vert_{W^{1,2}(S; \R^2)} \leq C$.
	%\begin{align*}
	%\Vert \bar w \Vert_{W^{2,2}(I)} \leq C_{G(y,w,\theta)},\quad \Vert \bar \theta \Vert_{W^{1,2}(I)} \leq C_{G(y,w,\theta)}
	%\end{align*}
	%and
	%\begin{align*}
	%\Vert \bar y \Vert_{W^{1,2}(S, \R^2)} &\leq C \Bigg( \int_S \vert \partial_1 \bar y_1 + \bar w' w' \vert \Bigg)^\frac{1}{2} + C \Vert \bar w' \Vert_{L^4(I)} \Vert w' \Vert_{L^4(I)}\\
	%&\leq C_{G(y,w,\theta),w}.
	%\end{align*}
	Thus, there exists a unique minimizer $( y_*, w_*, \theta_*) \in \mathcal{P}$. By computing the Euler-Lagrange equations, we observe that the minimum satisfies
	\begin{align*}
	\int_S \sqrt{\bar \C_R} H( y_*, w_*, \theta_* \vert w) \cdot \sqrt{\bar \C_R} H( \phi_y, \phi_w, \phi_\theta \vert w) - \int_S \bar \C_W[G(y,w,\theta),H(\phi_y,\phi_w,\phi_\theta \vert w)]
	\end{align*} 
	for all $(\phi_y, \phi_w, \phi_\theta) \in \mathcal{P}$. This equation can also be formulated as	\begin{align}
	\int_S {\mathcal{L}}(y,w,\theta)\cdot H(\phi_y,\phi_w,\phi_\theta\vert w) = 0\label{property:L}
	\end{align}
	for all $(\phi_y, \phi_w, \phi_\theta) \in \mathcal{P}$, where we define the operator ${\mathcal{L}}$ by
	\begin{align*}
	{\mathcal{L}}(y,w,\theta):= \bar \C_RH(y_*,w_*,\theta_*\vert w) - \bar \C_WG(y,w,\theta).
	\end{align*}
	By \eqref{def:HandG} and the regularity of the functions, we find
	${\mathcal{L}}(y,w,\theta) \in L^2(S; \R^3)$.	By \eqref{slopeproof}, \eqref{property:L}, and the definition of ${\mathcal{L}}$ we then get
	\begin{align*}
	\vert \partial \phi_0\vert_{\mathcal{D}_0}(y,w,\theta) &\geq \frac{\int_S (\bar \C_WG(y,w,\theta) + {\mathcal{L}}(y,w,\theta))\cdot H(y_*, w_*, \theta_* \vert w)}{\Big\Vert \sqrt{\bar \C_R}H( y_* , w_*, \theta_* \vert w)\Big\Vert_{L^2(S;\R^3)}}\\
	&= \frac{\int_S \sqrt{\bar \C_R}^{-1}(\bar \C_WG(y,w,\theta) + {\mathcal{L}}(y,w,\theta))\cdot\sqrt{\bar \C_R}H(y_*, w_*, \theta_* \vert w)}{\Big\Vert \sqrt{\bar \C_R}H( y_* , w_*, \theta_* \vert w)\Big\Vert_{L^2(S;\R^3)}}\\
	&= \Big\Vert \sqrt{\bar \C_R}H( y_* , w_*, \theta_* \vert w)\Big\Vert_{L^2(S; \EEE \R^3)} = \Big\Vert \sqrt{\bar \C_R}^{-1}(\bar \C_W G(y,w,\theta) + {\mathcal{L}}(y,w, \theta))\Big\Vert_{L^2(S;\R^3)}.
	\end{align*}
	On the other hand, by a similar argument, in view of \eqref{slopeproof} and \eqref{property:L}, we find
	\begin{align}
	\vert \partial \phi_0\vert_{\mathcal{D}_0}(y,w,\theta) &= \sup\limits_{ 0\neq (\hat y,\hat w, \hat \theta)\in\mathcal{P}} \frac{\int_S (\bar \C_WG(y,w,\theta) + {\mathcal{L}}(y,w,\theta)) \cdot H(\hat y, \hat w, \hat \theta \vert w)}{\Big\Vert \sqrt{\bar \C_R}H(\hat y , \hat w, \hat \theta \vert w)\Big\Vert_{L^2(S;\R^3)}}\nonumber \\
	&\leq \Big\Vert \sqrt{\bar \C_R}^{-1}(\bar \C_W G(y,w,\theta) + {\mathcal{L}}(y,w, \theta))\Big\Vert_{L^2(S;\R^3)}, \nonumber
	\end{align}
	where we again distributed $\sqrt{\bar \C_R}$ suitably to the two terms and used the Cauchy-Schwarz inequality. This concludes the proof.
\end{proof}

%--------------------------------------------------------------------------

%--------------------------------------------------------------------------
%--------------------------------------------------------------------------
 \typeout{References}

\end{document}